\documentclass[11pt,a4paper]{scrartcl}
\usepackage[utf8]{inputenc}
\usepackage[english]{babel}
\usepackage{helvet, stmaryrd, ulem}
\usepackage{amsmath}
\usepackage{amsthm}
\usepackage{amsfonts}
\usepackage{amssymb}
\usepackage{mathtools}
\usepackage{mathrsfs}
\usepackage{dsfont}
\usepackage{tabularx}
\usepackage{xcolor}
\usepackage{bbm}
\usepackage[paper=a4paper,left=25mm,right=25mm,top=25mm,bottom=35mm]{geometry}

\usepackage{hyperref}
\usepackage[noabbrev,capitalize]{cleveref}

\author{Helena Kremp, Nicolas Perkowski}
\title{Weak solutions for singular Lévy SDEs in the rough regime}

\newtheorem{theorem}{Theorem}[section]
\newtheorem{definition}[theorem]{Definition}     
\newtheorem{proposition}[theorem]{Proposition}
\newtheorem{lemma}[theorem]{Lemma}	

\newtheorem{corollary}[theorem]{Corollary}

\newtheorem{remark}[theorem]{Remark}
\newtheorem{assumption}[theorem]{Assumption}

\numberwithin{equation}{section}

\newcommand{\R}{\mathbb{R}}

\newcommand{\N}{\mathbb{N}}
\newcommand{\E}{\mathds{E}}
\newcommand{\p}{\mathds{P}}

\newcommand{\F}{\mathcal{F}}
\newcommand{\La}{\mathfrak{L}^{\alpha}_{\nu}}
\let\oldmathcal\mathcal
\newcommand{\calC}{\mathscr{C}}
\newcommand{\CTcalC}{C_T\calC}
\newcommand{\calK}{\oldmathcal{K}}
\newcommand{\calX}{\mathscr{X}}
\newcommand{\calV}{\mathscr{V}}

\newcommand{\calG}{\mathscr{G}}
\renewcommand{\mathcal}[1]{\mathscr{#1}}
\renewcommand{\epsilon}{\varepsilon}

\newcommand{\para}{\varolessthan}
\newcommand{\arap}{\varogreaterthan}
\newcommand{\reso}{\varodot}

\newcommand{\tmop}[1]{\ensuremath{\operatorname{#1}}}

\newcommand{\squeeze}[2][0]{\mbox{$\medmuskip=#1mu\displaystyle#2$}}

\DeclarePairedDelimiter{\abs}{\lvert}{\rvert}
\DeclarePairedDelimiter{\norm}{\lVert}{\rVert}

\DeclarePairedDelimiter{\paren}{(}{)}

\begin{document}
\begin{center}
\begin{huge}
Rough weak solutions for singular Lévy SDEs\\
\end{huge}
\begin{Large}
\vspace{0.5cm}
Helena Kremp\footnote{Technische Universität Wien, helena.kremp@asc.tuwien.ac.at}, Nicolas Perkowski\footnote{Freie Universität Berlin, perkowski@math.fu-berlin.de}\\
\end{Large}
\vspace{1cm}
\hrule
\end{center}
We introduce a weak solution concept (called "rough weak solutions") for singular SDEs with additive $\alpha$-stable Lévy noise (including the Brownian noise case) and prove its equivalence to martingale solutions from \cite{kp} in the Young and rough regularity regime. In the rough regime this leads to the construction of certain rough stochastic sewing integrals involved. For rough weak solutions, we can then prove a generalized Itô formula. Furthermore, we show that canonical weak solutions are wellposed in the Young case (and equivalent to rough weak solutions), while ill-posed in the rough case. For the latter, we construct a counterexample for uniqueness in law.\\
\textit{Keywords: fractional Laplace operator, singular Lévy SDEs, equivalence of weak solution concepts, paracontrolled distributions, rough stochastic integrals}\\
\textit{MSC2020: 60H10, 60G51, 60L20, 60L40, 60K37.}
\vspace{0,5cm}
\hrule
\vspace{1cm}

\begin{section}{Introduction}
We study the weak well-posedness of SDEs
\begin{equation}\label{eq:intro-sde}
dX_{t}=V(t,X_{t})dt + dL_{t},\qquad X_{0}=x\in\R^{d},
\end{equation} 
driven by non-degenerate symmetric $\alpha$-stable Lévy noise $L$ for $\alpha\in (1,2]$ and with drift $V(t,\cdot)$ that is a Besov distribution in the space variable.

\noindent The special case where $L$ is a Brownian motion has received lots of attention in recent years, since such \textit{singular diffusions} arise as models for stochastic processes in random media. Examples are random directed polymers~\cite{Alberts2014, Delarue2016, Caravenna2017}, self-attracting Brownian motion in a random medium~\cite{Cannizzaro2018}, or a continuum analogue of Sinai's random walk in random environment (Brox diffusion,~\cite{Brox1986}). Singular diffusions also arise as ``stochastic characteristics'' of singular SPDEs, for example the KPZ equation, cf.~\cite{Gubinelli2017KPZ}, or the parabolic Anderson model, cf.~\cite{Cannizzaro2018}.

\noindent SDEs with distributional drifts were first considered in \cite{Bass2001, Flandoli2003} in the one-dimensional time-homogeneous setting.
Of course, for distributional $V$ the point evaluation $V(t,X_t)$ is not meaningful, so a priori it is not clear how to make sense of~\eqref{eq:intro-sde}. The appropriate perspective is not to consider $V(t,X_t)$ at fixed time $t$, but rather to work with the integral $\int_0^t V(s,X_s) ds$. The intuition is that, because of small scale oscillations of $X$ induced by the oscillations of $L$, we only ``see an averaged version'' of $V$ and this gives rise to some regularization, at least for a Brownian motion or a sufficiently ``wild'' L\'evy jump process. On the other hand we would not expect any regularization from a Poisson process. In the Brownian case, this intuition can be made rigorous in different ways. For example via a Zvonkin transform which removes the drift, cf.~\cite{Zvonkin1974, Veretennikov1981, Bass2001, Krylov2005, Flandoli2010, Flandoli2017}, by considering the associated martingale problem and by constructing a domain for the singular generator, cf.~\cite{Flandoli2003, Delarue2016, Cannizzaro2018}, or by Dirichlet forms as in \cite{Mathieu1994}. In the one-dimensional case it is also possible to apply an It\^o-McKean construction based on space and time transformations, cf.~\cite{Brox1986}.

\noindent In \cite{kp} we followed the martingale problem approach in the spirit of~\cite{Delarue2016, Cannizzaro2018} who considered the Brownian case. Formally, $X$ solves~\eqref{eq:intro-sde} if and only if it solves the martingale problem for the generator $\mathcal G^V = \partial_t - \La + V\cdot \nabla$, where the fractional Laplacian $(-\La)$ is the generator of $L$. That is, for all functions $u$ in the domain of $\mathcal G^V$, the process $u(t,X_t) - u(0,x) - \int_0^t \mathcal G^V u(s,X_s) ds$, $t \geqslant 0$, is a martingale. One difficulty is that the domain of $\mathcal G^V$ necessarily has a trivial intersection with the set of smooth functions: If $u$ is smooth, then $(\partial_t - \La)u$ is smooth as well, while for non-constant $u$ the product $V \cdot \nabla u$ is only a distribution, but not a continuous function. If we want $\mathcal G^V u$ to be a continuous function, then $u$ has to be non-smooth, so that $(\partial_t - \La)u$ is also a distribution which has appropriate cancellations with $V\cdot \nabla u$.

\noindent We can find such $u$ by solving the Kolmogorov backward equation for continuous right-hand sides and regular terminal conditions, such that $\mathcal{G}^{\calV}u=f$ by construction, as carried out in \cite{kp}.

\noindent There have been several results on singular L\'evy SDEs in the Young regime in recent years. \cite{Athreya2018} consider the time-homogeneous one-dimensional case and construct weak solutions via a Zvonkin transform. They additionally establish strong uniqueness and existence by a Yamada-Watanabe type argument, which in particular is restricted to $d=1$. Related to that, but for Hölder continuous drift of regularity at least $1-\alpha/2$, \cite{Priola2012} proves pathwise uniqueness in the multidimensional, time-homogeneous case. Two nearly simultaneous articles,  \cite{Ling2019} and \cite{deRaynal2019}, consider the multidimensional case (time-homogeneous, respectively time-inhomogeneous) and prove existence and uniqueness for the martingale problem. They consider $V \in C([0,T], B_{p,q}^\beta)$ for general $p,q$ (subject to suitable conditions), where for $p=q=\infty$ they require the Young regime, $\beta > (1-\alpha)/2$. Let us also mention~\cite{Harang2020Regularity}, who prove pathwise regularization by noise results for SDEs driven by (very irregular) fractional L\'evy noise, based on the methods of~\cite{Catellier2016, Harang2020Cinfinity}.

\noindent We treat the multidimensional time-inhomogeneous case with drifts in the rough regime. However, we concentrate on $B_{\infty,\infty}$ Besov spaces and do not consider $B_{p,q}$ for general $p,q$.

\noindent The main contribution of this paper is the derivation of a well-posed rough weak solution concept. In the case of bounded and measurable coefficients the equivalence of probabilitic weak solutions and solutions of the martingale problem is by now classical, cf.~\cite{Stroock2006} (in the Brownian noise case) and \cite{Kurtz2011} (in the Lévy noise case). It has so far been an open problem (in both the Brownian and Lévy noise case), whether these results can be generalized to distributional drifts in the rough regime. In the Young regime with time-independent drift \cite{Athreya2018} introduces a canonical weak solution concept, replacing the singular drift by the limit of smooth drift terms. 
The same concept can be considered in multiple dimensions and with time-depending drift. A canonical weak solution is then a tuple of stochastic processes $(X,L)$ on some probability base space, such that $L$ is a symmetric $\alpha$-stable Lévy process and $X$ is given by
\begin{align}\label{eq:XZL}
X=x+Z+L,
\end{align} for a continuous drift process $Z$ that is given as a limiting object (in probability) 
\begin{align*}
Z=\lim_{n\to\infty}\int_{0}^{\cdot}V^{n}(r,X_{r})dr=:\lim_{n\to\infty}Z^{n}
\end{align*} for smooth $(V^{n})$ approximating $V$. Furthermore, $Z$ shall satisfy the following Hölder-type bound: there exists $C>0$, such that for all $0\leqslant s<t\leqslant T$,
\begin{align}\label{eq:Z-bound-intro}
\E[\abs{Z_{t}-Z_{s}}^{2}]\leqslant C\abs{t-s}^{2(\alpha+\beta)/\alpha}.
\end{align} 
The particular Hölder-regularity in \eqref{eq:Z-bound-intro} originates from the time regularity of the solution of the Kolmogorov backward equation. This implies that the solution $X$ is a Dirichlet process and Itô formulas are available.\\
The recent article \cite{IR} consideres the multidimensional Brownian noise case and proves equivalence of martingale solutions and so-called $B$-solutions from \cite[Definition 6.2]{IR} with $B=C_{T}\calC^{\beta}_{\R^{d}}$ for $\beta\in (-1/2,0)$, that additionally satisfy the ``reinforced local time property" (cf. \cite[Definition 6.7]{IR}). 
The reinforced local time property boils down to stability of the integral 
\begin{align}\label{eq:int-u-Z}
\int_{0}^{t} \nabla u (s,X_{s})dZ_{s},
\end{align} when approximating the drift term by $(Z^{n})$ and $u$ solving $\mathcal{G}^{V}u=f$, by $(u^{n})\subset C^{1,2}_{b}([0,T]\times\R^{d})$.\\ In the Young case we show that the regularity of the PDE solution, stability of the PDE solution map and the bound \eqref{eq:Z-bound-intro} together yield stability for the stochastic integral \eqref{eq:int-u-Z} using the stochastic sewing lemma by \cite{le}. However, this fails in the rough case as the PDE solution is too irregular.  
In the setting of general dimension, time-depending drift and $\alpha$-stable noise for $\alpha\in (1,2]$, we prove that canonical weak solutions are equivalent to martingale solutions in the Young regime. 
In the rough case, we prove that canonical weak solutions are in general non-unique in law. Heuristically, this is due to the fact that the canonical weak solution does not uniquely determine the enhancement $\calV$ of $V$. In the one-dimensional Brownian case the situation becomes particularily interesting.  While for any smooth approximating sequence $(V^{n})$ of the singular drift $V$, the strong solutions $X^{n}$ of  
\begin{align*}
dX_{t}^{n}=V^{n}(t,X_{t}^{n})dt+dB_{t}
\end{align*}
converge in distribution to the same limit, given by the solution of the $\mathcal{G}^{\cal V}$- martingale problem, it is however not the case that there exists a unique canonical weak solution in the above sense.\\ 
Our approach to obtain a well-posed weak solution concept in the rough regime is to impose further assumptions that ensure the uniqueness of the extension of the integral \eqref{eq:int-u-Z}.
To that aim, we use ideas from rough paths theory, specifically the construction of rough stochastic integrals from \cite{le-friz-hocq}. Lifting $Z$ to a rough \textit{stochastic} integrator process $(Z,\mathbb{Z}^{V})$ enables to extend the integral from smooth integrands $\nabla u^{n}$ in a stable manner to (para-)controlled integrands $\nabla u$. The lift is chosen in such a way as to correct for the most irregular terms in the paracontrolled decomposition of $\nabla u$. 
We then define rough weak solutions accordingly. That is, we require a rough weak solution $X$ to satisfy the conditions of a canonical weak solution and to be furthermore such that the iterated integrals $\mathbb{Z}^{V}$
are well-defined and satisfy suitable Hölder-moment-bounds. Here $\mathbb{Z}^{V}$ formally corresponds to the resonant product component in the enhancement $\calV$. 
We say that $(Z,\mathbb{Z}^{V})$ is a so-called \textit{rough stochastic integrator}, while the process $(\nabla u(t,X_{t}))$ for the solution $u$ of the Kolmogorov equation for regular right-hand side, is \textit{stochastically controlled}.\\ The terms ``stochastically controlled'' and ``rough (stochastic) integrator'' are motivated by \cite{le-friz-hocq}, but our definitions and the integral we construct differ. The difficulty here does not arise due to a low regularity integrator, but due to a integrand of low regularity, that is however controlled. Further difficulties arise due to the integrability issue in the pure stable noise case, which means that we can construct the integral in $L^{2}(\p)$, but possibly not in $L^{p}(\p)$ for $p>2$. Summarizing, for a rough weak solution $X$ we obtain a stable rough stochastic integral 
\begin{align}\label{eq:int-u-ZZ}
\int_{0}^{t}\nabla u(s,X_{s})d(Z,\mathbb{Z}^{V})_{s}
\end{align} in $L^{2}(\p)$. For regular integrands the rough stochastic integral \eqref{eq:int-u-ZZ} coincides with the stochastic integral against $Z$ in \eqref{eq:int-u-Z}. More generally, we construct the rough stochastic integral $\int_{0}^{\cdot} f_{t}d(Z,\mathbb{Z}^{A})_{t}$, for rough stochastic integrators $(Z,\mathbb{Z}^{A})$ for a given stochastic process $A$ (formally $\mathbb{Z}^{A}_{st}=\int_{s}^{t}A_{s,r}dZ_{r}$) and an integrand $f$, that is stochastically controlled by $A$.\\
The stability of the rough stochastic integral \eqref{eq:int-u-ZZ} then enables to prove that a rough weak solution is indeed a solution of the $\mathcal{G}^{\calV}$-martingale problem, in particular unique in law. 
To prove that a martingale solution is a rough weak solution we need to show the existence of the iterated integrals $\mathbb{Z}^{V}$ satisfying suitable bounds. We show that the bounds on $\mathbb{Z}^{V}$ are implied by regularity properties and fortunate cancellations between the solutions of Kolmogorov backward equations for certain \textit{singular} paracontrolled terminal conditions,  
whose existence theory we carried out in \cite{kp-sk}.

\noindent The paper is structured as follows.
\cref{sec:ws} introduces the rough weak solution concept and states the main \cref{thm:mainthm,thm:igenIto}. In \cref{sec:roughint} we construct the general rough stochastic integral. In \cref{sec:wsemp} we prove in \cref{thm:wsimp,thm:mpiws} equivalence of the weak solution concepts, which together imply \cref{thm:mainthm}. Furthermore we prove a generalized Itô formula for rough weak solutions in \cref{thm:genIto}.  \cref{sec:wsyc} investigates the canonical weak solution concept. We prove well-posedness in the Young case and ill-posedness in the rough case.

\end{section}
\begin{section}{Preliminaries}

\noindent Let $(p_{j})_{j\geqslant -1}$ be a smooth dyadic partition of unity, i.e. a family of functions $p_{j}\in C^{\infty}_{c}(\R^{d})$ for $j\geqslant -1$, such that 
\begin{enumerate}
\item[1.)]$p_{-1}$ and $p_{0}$ are non-negative radial functions (they just depend on the absolute value of $x\in\R^{d}$), such that the support of $p_{-1}$ is contained in a ball and the support of $p_{0}$ is contained in an annulus;
\item[2.)]$p_{j}(x):=p_{0}(2^{-j}x)$, $x\in\R^{d}$, $j\geqslant 0$;
\item[3.)]$\sum_{j=-1}^{\infty}p_{j}(x)=1$ for every $x\in\R^{d}$; and
\item[4.)]$\operatorname{supp}(p_{i})\cap \operatorname{supp}(p_{j})=\emptyset$ for all $\abs{i-j}>1$.
\end{enumerate}
We then define the Besov spaces for $p,q\in [1,\infty]$,
\begin{align}\label{def:bs}
B^{\theta}_{p,q}:=\{u\in\mathcal{S}':\norm{u}_{B^{\theta}_{p,q}}=\norm[\big]{(2^{j\theta}\norm{\Delta_{j}u}_{L^{p}})_{j\geqslant -1}}_{\ell^{q}}<\infty\},
\end{align}
where $\Delta_{j}u=\mathcal{F}^{-1}(p_{j}\mathcal{F}u)$ are the Littlewood-Paley blocks, and the Fourier transform is defined with the normalization $\hat{\varphi}(y):=\F\varphi (y):=\int_{\R^{d}}\varphi(x)e^{-2\pi i\langle x,y\rangle}dx$ (and $\F^{-1}\varphi(x)=\hat{\varphi}(-x)$); moreover, $\mathcal{S}$ are the Schwartz functions and $\mathcal S'$ are the Schwartz distributions.\\
Let $C^{\infty}_{b}=C^{\infty}_{b}(\R^{d},\R)$ denote the space of bounded and smooth functions with bounded partial derivatives.
For $q=\infty$, the space $B^{\theta}_{p,\infty}$ has the unpleasant property that $C^{\infty}_b \subset B^\theta_{p,\infty}$ is not dense.
Therefore, we rather work with the following space:
\begin{align}\label{eq:sep}
B^{\theta}_{p,\infty}:=\{u\in\mathcal{S}'\mid\lim_{j\to\infty}2^{j\theta}\norm{\Delta_{j}u}_{L^{p}}=0\},
\end{align} 
for which $C^\infty_b$ is a dense subset (cf. \cite[Remark 2.75]{Bahouri2011}). We also use the notation $\calC^{\theta}_{\R^{d}}:=(\calC^{\theta})^{d}=\calC^{\theta}(\R^{d},\R^{d})$, $\calC^{\theta-}:=\bigcap_{\gamma<\theta}\calC^{\gamma}$ and $\calC^{\theta+}=\bigcup_{\gamma>\theta}\calC^{\gamma}$. Furthermore, we introduce the notation $\calC^{\theta}_{p}:=B^{\theta}_{p,\infty}$ for $\theta\in\R$ and $p\in [1,\infty]$, where $\calC^{\theta}:=\calC^{\theta}_{\infty}$ with norm denoted by $\norm{\cdot}_{\theta}:=\norm{\cdot}_{\calC^{\theta}}$.\\ 
For $1\leqslant p_{1}\leqslant p_{2}\leqslant\infty$, $1\leqslant q_{1}\leqslant q_{2}\leqslant\infty$ and $s\in\R$, the Besov space $B^{s}_{p_{1},q_{1}}$ is continuously embedded in $B_{p_{2},q_{2}}^{s-d(1/p_{1}-1/p_{2})}$ (cf. \cite[Proposition 2.71]{Bahouri2011}). Furthermore, we will use that for $u\in B^{s}_{p,q}$ and a multi-index $n\in\N^{d}$, $\norm{\partial^{n}u}_{B^{s-\abs{n}}_{p,q}}\lesssim \norm{u}_{B^{s}_{p,q}}$, which follows from the more general multiplier result from \cite[Proposition 2.78]{Bahouri2011}.\\
We recall from Bony's paraproduct theory (cf. \cite[Section 2]{Bahouri2011}) that in general for $u\in\calC^{\theta}$ and $v\in\calC^{\beta}$ with $\theta,\beta\in\R$, the product $u v:=u\para v+u\arap v +u \reso v$ , is  well defined in $\calC^{\min(\theta,\beta,\theta+\beta)}$ if and only if $\theta+\beta>0$. Denoting $S_{i}u=\sum_{j=-1}^{i-1}\Delta_{j}u$, the paraproducts are defined as follows
\begin{align*}
u\para v:=\sum_{i\geqslant -1} S_{i-1}u\Delta_{i}v,\quad u\arap v:=v\para u, \quad u\reso v:= \sum_{\abs{i-j}\leqslant 1}\Delta_{i}u\Delta_{j}v.
\end{align*} Here, we use the notation of \cite{Martin2017, Mourrat2017Dynamic} for the para- and resonant products $\para, \arap$ and  $\reso$.\\
In estimates we often use the notation $a\lesssim b$, which means, that there exists a constant $C>0$, such that $a\leqslant C b$. In the case that we want to stress the dependence of the constant $C(d)$ in the estimate on a parameter $d$, we write $a\lesssim_{d} b$.\\
The paraproducts satisfy the following estimates for $p,p_{1},p_{2}\in[1,\infty]$ with $\frac{1}{p}=\frac{1}{p_{1}}+\frac{1}{p_{2}}\leqslant 1$ and $\theta,\beta\in\R$ (cf. \cite[Theorem A.1]{PvZ} and \cite[Theorem 2.82, Theorem 2.85]{Bahouri2011})
\begin{equation}
\begin{aligned}\label{eq:paraproduct-estimates}
\norm{u\reso v}_{\calC^{\theta+\beta}_{p}} & \lesssim\norm{u}_{\calC^{\theta}_{p_{1}}}\norm{v}_{\calC^{\beta}_{p_{2}}}, \qquad \text{if }\theta +\beta > 0,\\
\norm{u\para v}_{\calC^{\beta}_{p}} \lesssim\norm{u}_{L^{p_{1}}}\norm{v}_{\calC^{\beta}_{p_{2}}}& \lesssim\norm{u}_{\calC^{\theta}_{p_{1}}}\norm{v}_{\calC^{\beta}_{p_{2}}}, \qquad \text{if } \theta > 0,\\
\norm{u\para v}_{\calC^{\beta+\theta}_{p}}& \lesssim\norm{u}_{\calC^{\theta}_{p_{1}}}\norm{v}_{\calC^{\beta}_{p_{2}}}, \qquad \text{if } \theta < 0.
\end{aligned}
\end{equation}
So if $\theta + \beta > 0$, we have $\norm{u v}_{\calC^{\gamma}_{p}}\lesssim\norm{u}_{\calC^{\theta}_{p_{1}}}\norm{v}_{\calC^{\beta}_{p_{2}}}$ for $\gamma:=\min(\theta,\beta,\theta+\beta)$.\\
Next, we collect some facts about $\alpha$-stable Lévy processes and their generators and semigroups. For $\alpha\in (0,2]$, a symmetric $\alpha$-stable Lévy process $L$ is a Lévy process, that moreover satisfies the scaling property $(L_{k t})_{t \geqslant 0}\stackrel{d}{=}k^{1/\alpha}(L_{t})_{t \geqslant 0}$ for any $k>0$ and $L\stackrel{d}{=}-L$, where $\stackrel{d}{=}$ denotes equality in law. These properties determine the jump measure $\mu$ of $L$, see \cite[Theorem 14.3]{Sato1999}.  That is, if $\alpha\in (0,2)$, the Lévy jump measure $\mu$ of $L$ is given by
\begin{align}\label{eq:mu}
\mu(A):=\E\bigg[\sum_{0\leqslant t\leqslant 1}\mathbf{1}_{A}(\Delta L_{t})\bigg]=\int_{S}\int_{\R^{+}}\mathbf{1}_{A}(k\xi)\frac{1}{k^{1+\alpha}}dk\tilde \nu(d\xi),\quad A\in \mathcal{B}(\R^{d}\setminus\{0\}),
\end{align} where $\tilde \nu$ is a finite, symmetric, non-zero measure on the unit sphere $S\subset\R^{d}$.  
Furthermore, we also define for $A\in\mathcal{B}(\R^{d}\setminus\{0\})$ and $t\geqslant 0$ the Poisson random measure
\begin{align*}
\pi(A\times [0,t])=\sum_{0\leqslant s\leqslant t}\mathbf{1}_{A}(\Delta L_{s}),
\end{align*} with intensity measure $dt\mu(dy)$. Denote the compensated Poisson random measure of $L$ by $\hat{\pi}(dr,dy):=\pi(dr,dy)-dr\mu(dy)$.
We refer to the book by Peszat and Zabczyk \cite{peszat_zabczyk_2007} for the integration theory against Poisson random measures and for the Burkholder-Davis-Gundy inequality \cite[Lemma 8.21 and 8.22]{peszat_zabczyk_2007}, which we will both use in the sequel.
The generator $A$ of $L$ satisfies $C_{b}^{\infty}(\R^{d})\subset \operatorname{dom}(A)$ and is given by
\begin{align}\label{eq:functional}
A\varphi(x)=\int_{\R^{d}}\paren[\big]{\varphi(x+y)-\varphi(x)-\mathbf{1}_{\{\abs{y}\leqslant 1\}}(y) \nabla \varphi(x) \cdot y}\mu(dy)\qquad\text{for }\varphi\in C_{b}^{\infty}(\R^{d}).
\end{align} If $(P_t)_{t\geqslant 0}$ denotes the semigroup of $L$, the convergence $t^{-1}(P_{t}f(x)-f(x))\to Af(x)$ is uniform in $x\in\R^{d}$ (see \cite[Theorem 5.4]{peszat_zabczyk_2007}).

Next, we recall the equivalent Fourier definition of the operator $A$ and the Schauder  and commutator estimates for its semigroup from \cite{kp, kp-sk}.

\begin{definition}\label{def:fl}
Let $\alpha \in (0,2)$ and let $\nu$ be a symmetric (i.e. $\nu(A)=\nu(-A)$), finite and non-zero measure on the unit sphere $S\subset\R^{d}$. We define the operator $\La$ as
\begin{align}
\La\F^{-1}\varphi=\F^{-1}(\psi^{\alpha}_{\nu} \varphi)\qquad\text{for $\varphi\in C^\infty_b$,}
\end{align}  where
$\psi^{\alpha}_{\nu} (z):=\int_{S}\abs{\langle z,\xi\rangle}^{\alpha}\nu(d\xi).$
For $\alpha=2$, we set $\La:=-\frac{1}{2}\Delta$.
\end{definition}

\begin{remark}
If we take $\nu$ as a suitable multiple of the Lebesgue measure on the sphere, then $\psi^\alpha_\nu(z) = |2\pi z|^\alpha$ and thus $\La$ is the fractional Laplace operator $(-\Delta)^{\alpha/2}$. 
\end{remark}

\begin{lemma}\label{rem:lgen}\footnote{\cite[Lemma 2.3]{kp}}
Let $\alpha \in (0,2)$ and let again $\nu$ be a symmetric, finite and non-zero measure on the unit sphere $S\subset\R^{d}$. Then for $\varphi \in C^\infty_b$ we have $-\La \varphi = A\varphi$, where $A$ is the generator of the symmetric, $\alpha$-stable Lévy process $L$ with characteristic exponent $\E[\exp(2\pi i\langle z,L_{t}\rangle )]=\exp(-t\psi^{\alpha}_{\nu}(z))$. The process $L$ has the jump measure $\mu$ as defined in~\cref{eq:mu}, with $\tilde \nu = C \nu$ for some constant $C>0$.
\end{lemma}

\noindent If $\alpha=2$, then the generator of the symmetric, $\alpha$-stable process coincides with $\sum_{i,j}C(i,j)\partial_{x_{i}}\partial_{x_{j}}$ for an explicit covariance matrix $C$ (cf.~\cite[Theorem 14.2]{Sato1999}), that is, the generator of $\sqrt{2C}B$ for a standard Brownian motion $B$. To ease notation, we consider here $C=\frac{1}{2}\operatorname{Id}_{d\times d}$ and whenever we refer to the case $\alpha=2$, we mean the standard Brownian motion noise case and $\La=-\frac{1}{2}\Delta$.

\begin{assumption}\label{ass}
Throughout the paper, we assume that the measure $\nu$ from \cref{def:fl} has $d$-dimensional support, in the sense that the linear span of its support is $\R^d$. This means that the process $L$ can reach every open set in $\R^d$ with positive probability.
\end{assumption}
\noindent An $\alpha$-stable, symmetric Lévy process, that satisfies \cref{ass}, we also call non-degenerate.

\begin{proposition}[Continuity of the operator $\La$]\label{prop:contfl}\footnote{\cite[Proposition 2.4]{kp-sk}}
Let $\alpha\in (0,2]$. Then for $\beta\in\R$ and $u \in C^\infty_b$, $p\in[1,\infty]$, we have
\begin{align*}
\norm{\La u}_{\calC^{\beta-\alpha}_{p}}\lesssim\norm{u}_{\calC^{\beta}_{p}}.
\end{align*}
In particular, $\La$ can be uniquely extended to a continuous operator from $\calC^{\beta}_{p}$ to $\calC^{\beta-\alpha}_{p}$.
\end{proposition}

\begin{remark}\label{rmk:A}
One can show that the operators $A$ and $-\La$ agree on $\bigcup_{\varepsilon>0}\calC^{2+\epsilon}$. 
\end{remark}

\noindent The formal generator associated to the diffusion \eqref{eq:intro-sde} is thus given by
\begin{align}\label{eq:gen-sde}
\mathcal{G}^{V}=\partial_{t}-\La+V\cdot \nabla.
\end{align}
The connection between the generator $\mathcal{G}^{V}$ and the solution $X$ is made rigorous in \cite{kp} by solving the associated martingale problem, that we recall below. Solving the martingale problem is based on solving the singular Kolmogorov backward PDE via the paracontrolled approach (in the rough regularity regime).
Before stating the result from \cite{kp-sk}, let us introduce the notation 
\begin{align*}
\mathcal{K}(\eta_{1},\eta_{2}):=\big[\mathring{\Delta}_{T}\ni(s,t)\mapsto \sum_{j=1}^{d}P_{t-s}\partial_{j}\eta_{1}^{i}(t)\reso\eta_{2}^{j}(s)\big],
\end{align*} for $\eta_{1},\eta_{2}\in C_{T}C^{\infty}_{b}(\R^{d},\R^{d})$ and $\mathring{\Delta}_{T}:=\{(s,t)\in [0,T]^{2}\mid s< t\}$. Furthermore, we set $\mathcal{K}(\eta):=\mathcal{K}(\eta,\eta)$. Let also $\Delta_{T}:=\{(s,t)\in [0,T]^{2}\mid s\leqslant t\}$.\\ 
For a Banach space $X$ and $\gamma\in (0,1)$, we define the blow-up space
\begin{align*}
\mathcal{M}_{\mathring{\Delta}_{T}}^{\gamma}X:=\{g:\mathring{\Delta}_{T}\to X\mid \sup_{0\leqslant s<t\leqslant T}(t-s)^{\gamma}\norm{g(s,t)}_{X}<\infty\}.
\end{align*}
Moreover, for a sequence $(a_{m,n})_{m,n\in\N}$ in a Banach space, for which the convergence 
\begin{align*}
\lim_{m\to\infty}\lim_{n\to\infty}a_{m,n}=a=\lim_{n\to\infty}\lim_{m\to\infty}a_{m,n}
\end{align*} holds, we use the short-hand notation $a^{m,n}\to a$ for $m,n\to\infty$ or $\lim_{m,n\to\infty}a^{m,n}=a$.\\
Next, we define the space of enhanced distributions $\mathcal{X}^{\beta,\gamma}$ as follows. Comparing with \cite[Definition 4.2]{kp-sk}, we furthermore assume that for $(\mathcal{V}_{1},\mathcal{V}_{2})\in\mathcal{X}^{\beta,\gamma}$, the mixed resonant products $(\mathcal{K}(V^{n},V^{m}))$ converge to $\calV_{2}$.
\begin{definition}[Enhancement] \label{def:enhanced-dist}
Let $T>0$. For $\beta\in (\frac{2-2\alpha}{3},\frac{1-\alpha}{2})$ and $\gamma\in [\frac{2\beta+2\alpha-1}{\alpha},1)$, we define $\mathcal{X}^{\beta,\gamma}$ as the closure of
\begin{align*}
\{(\eta_{1}, \mathcal{K}(\eta_{1},\eta_{2}))\mid \eta_{1},\eta_{2}\in C_{T}C^{\infty}_{b}(\R^d,\R^d)\}
\end{align*} in $C_{T}\calC^{\beta+(1-\gamma)\alpha}\times\mathcal{M}_{\mathring{\Delta}_{T}}\calC^{2\beta+\alpha-1}_{\R^{d\times d}}$. We say that $\calV$ is a lift or enhancement of $V$ if $\mathcal{V}_{1}=V$ and also shortly write $V\in\mathcal{X}^{\beta,\gamma}$.\\
For $\beta\in (\frac{1-\alpha}{2},0)$ and $\gamma\in [\frac{\beta-1}{\alpha},1)$, we set $\mathcal{X}^{\beta,\gamma}=C_{T}\calC^{\beta+(1-\gamma)\alpha}$.
\end{definition} 
\begin{remark}
$V\in\mathcal{X}^{\beta,\gamma}$ in the sense of \cref{def:enhanced-dist} implies that $V$ is also enhanced in the sense of \cite[Definition 4.2]{kp-sk} and in the sense of \cite[Definition 3.5]{kp}.
\end{remark}

\noindent Let us define the function spaces, that we encounter in the sequel.
For $\gamma\in (0,1)$, $T>0$ and $\overline{T}\in (0,T]$, and a Banach space $X$, we define the blow-up space
\begin{align*}
\mathcal{M}_{\overline{T},T}^{\gamma}X:=\{u:[T-\overline{T},T)\to X\mid t\mapsto (T-t)^{\gamma}u_{t}\in C([T-\overline{T},T),X)\},
\end{align*} with $\norm{u}_{\mathcal{M}_{\overline{T},T}^{\gamma}X}:=\sup_{t\in [T-\overline{T},T)}(T-t)^{\gamma}\norm{u_t}_{X}$ and $\mathcal{M}_{\overline{T},T}^{0}X:=C([T-\overline{T},T),X)$. For $\overline{T}=T$, we use the notation $\mathcal{M}^{\gamma}_{T}X:=\mathcal{M}_{T,T}^{\gamma}X$. For $\vartheta\in (0,1]$, $\gamma\in (0,1)$, we furthermore define
\begin{align*}
C_{\overline{T},T}^{\gamma,\vartheta}X:=\biggl\{u:[T-\overline{T},T)\to X\biggm| \norm{f}_{C_{T}^{\gamma,\vartheta}X}:=\sup_{0\leqslant s<t< T}\frac{(T-t)^{\gamma}\norm{f_{t}-f_{s}}_{X}}{\abs{t-s}^{\vartheta}}<\infty\biggr\}
\end{align*} and $C_{T}^{\gamma,\vartheta}X:=C_{T,T}^{\gamma,\vartheta}X$. 
Let us also define for $\vartheta\in (0,1]$, $\overline{T}\in (0,T]$, the space of $\vartheta$-Hölder continuous functions on $[T-\overline{T},T]$ with values in $X$,
\begin{align*}
C_{\overline{T},T}^{\vartheta}X:=\biggl\{u:[T-\overline{T},T]\to X\biggm| \norm{u}_{C_{T}^{\vartheta}X}:=\sup_{T-\overline{T}\leqslant s<t\leqslant T}\frac{\norm{u_{t}-u_{s}}_{X}}{\abs{t-s}^{\vartheta}}<\infty\biggr\} 
\end{align*} and $C_{T}^{\vartheta}X:=C_{T,T}^{\vartheta}X$. We set $C_{\overline{T},T}^{0,\vartheta}X:=C^{\vartheta}([T-\overline{T},T),X)$.\\
Then we define for $\gamma\in (0,1)$ and $\theta\in \R$, $p\in[1,\infty]$, the space
\begin{align}
\mathcal{L}_{T}^{\gamma,\theta}:=\mathcal{M}_{T}^{\gamma}\calC^{\theta}_{p}\cap C_{T}^{1-\gamma}\calC^{\theta-\alpha}_{p}\cap C_{T}^{\gamma,1}\calC^{\theta-\alpha}_{p}.
\end{align} 
We moreover define for $\gamma=0$, 
\begin{align}\label{eq:gamma0}
\mathcal{L}_{T}^{0,\theta}:= C_{T}^{1}\calC^{\theta-\alpha}_{p}\cap C_{T}\calC^{\theta}_{p},
\end{align} where $C^{1}_{T}X$ denotes the space of $1$-Hölder or Lipschitz functions with values in $X$.\\ 
For $\overline{T}\in (0,T)$, we define $\mathcal{L}_{\overline{T},T}^{\gamma,\theta}:=\mathcal{M}_{\overline{T},T}^{\gamma}\calC^{\theta}_{p}\cap C_{\overline{T},T}^{1-\gamma}\calC^{\theta-\alpha}_{p}\cap C_{\overline{T},T}^{\gamma,1}\calC^{\theta-\alpha}_{p}$ and similarly $\mathcal{L}_{\overline{T},T}^{0,\theta}$.\\
The spaces $\mathcal{L}_{T}^{\gamma,\theta}$ are Banach spaces equipped with the norm
\begin{align*}
\norm{u}_{\mathcal{L}_{T}^{\gamma,\theta}}&:=\norm{u}_{\mathcal{M}_{T}^{\gamma}\calC^{\theta}_{p}}+\norm{u}_{C_{T}^{1-\gamma}\calC^{\theta-\alpha}_{p}} +\norm{u}_{C_{T}^{\gamma,1}\calC^{\theta-\alpha}_{p}}\\&=\squeeze[1]{\sup_{t\in[0,T)}(T-t)^{\gamma}\norm{u_{t}}_{\calC^{\theta}_{p}}+\sup_{0\leqslant s<t\leqslant T}\frac{\norm{u_{t}-u_{s}}_{\calC^{\theta-\alpha}_{p}}}{\abs{t-s}^{1-\gamma}}+\sup_{0\leqslant s<t< T}\frac{(T-t)^{\gamma}\norm{u_{t}-u_{s}}_{\calC^{\theta-\alpha}_{p}}}{\abs{t-s}}.}
\end{align*}

\noindent We can finally recall the result from \cite[Theorem 4.1, Theorem 4.7, Corollary 4.10 and Theorem 4.12]{kp-sk}, which treat singular paracontrolled terminal conditions. 

\begin{theorem}\label{thm:singular}
Let $T>0$, $\alpha\in (1,2]$, $p\in[1,\infty]$ and $\beta\in (\frac{2-2\alpha}{3},0)$ and $\mathcal{V}\in\calX^{\beta,\gamma'}$.\\
For $\beta\in (\frac{1-\alpha}{2},0), \gamma'\in [\frac{\beta-1}{\alpha},1)$, let $f\in \mathcal{L}_{T}^{\beta,\gamma'}$ and $u^{T}\in\calC_{p}^{(1-\gamma')\alpha+\beta}$.\\ For $\beta\in (\frac{2-2\alpha}{3},\frac{1-\alpha}{2}], \gamma'\in[\frac{2\beta+2\alpha-1}{\alpha},1) $, let  
\begin{align*}
f=f^{\sharp}+f'\para V
\end{align*} for $f^{\sharp}\in \mathcal{L}_{T}^{\gamma',\alpha+2\beta-1}$, $f'\in (\mathcal{L}_{T}^{\gamma',\alpha+\beta-1})^{d}$ and 
\begin{align*}
u^{T}=u^{T,\sharp}+u^{T,\prime}\para V_{T}
\end{align*} for $u^{T,\sharp}\in \calC^{(2-\gamma')\alpha+2\beta-1}_p$, $u^{T,\prime}\in (\calC^{\alpha+\beta-1}_p)^{d}$.\\
Then there exists a unique mild solution $u\in\mathcal{L}_{T}^{\gamma,\alpha+\beta}$ for $\gamma\in (\gamma',\frac{\alpha}{2-\alpha-3\beta}\gamma')$ to the 
singular Kolmogorov backward PDE
\begin{align*}
\mathcal{G}^{\mathcal{V}}u=f,\qquad u(T,\cdot)=u^{T}.
\end{align*}
and if $\beta\in (\frac{2-2\alpha}{3},\frac{1-\alpha}{2}]$, $u$ has the paracontrolled structure
\begin{align*}
u=u^{\sharp}+\nabla u\para J^{T}(V)+u^{T,\prime}\para P_{T-\cdot}V_{T}
\end{align*} for $u^{\sharp}\in\mathcal{L}_{T}^{\gamma,2\alpha+2\beta-1}$.\\
Furthermore in case of regular terminal condition, that is $f'\in \mathcal{L}_{T}^{0,\alpha+\beta-1}, f^{\sharp}\in\mathcal{L}_{T}^{0,\alpha+2\beta-1]}$ and $u^{T}=u^{T,\sharp}\in\calC_{p}^{2\alpha+2\beta-1}$, the solution satisfies $u\in\mathcal{L}_{T}^{0,\alpha+\beta}, u^{\sharp}\in\mathcal{L}_{T}^{0,2\alpha+2\beta-1}$.\\
Moreover, the solution map $(u^{T},f,\calV)\mapsto (u,u^{\sharp})\in \mathcal{L}_{T}^{\gamma,\alpha+\beta}\times \mathcal{L}_{T}^{\gamma,2\alpha+2\beta-1} $ is locally Lipschitz continuous. 
\end{theorem}

\noindent We use the definition of a solution to the singular martingale problem from \cite[Definition 4.1]{kp}, cf. also \cite{Cannizzaro2018}.

\begin{definition}[Martingale problem]\label{def:martp}
Let $\alpha\in (1,2]$ and $\beta\in(\frac{2-2\alpha}{3},0)$, and let $T>0$ and $V\in \mathcal{X}^{\beta,\gamma}$. Then, we call a probability measure $\p$ on the Skorokhod space $(\Omega,\mathcal{F})$ a solution of the martingale problem for $(\mathcal{G}^{V},\delta_x)$, if
\begin{enumerate}
\item[\textbf{1.)}] $\p(X_{0}\equiv x)=1$ (i.e. $\p^{X_{0}}=\delta_{x}$), and
\item[\textbf{2.)}] for all $f\in C_{T}\calC^{\epsilon}$ with $\varepsilon > 2-\alpha$ and for all $u^{T}\in\mathcal{C}^{3}$, the process $M=(M_{t})_{t\in [0,T]}$ is a martingale under $\p$ with respect to $(\mathcal{F}_{t})$, where
\begin{align}
M_{t}=u(t,X_{t})-u(0,x)-\int_{0}^{t}f(s,X_{s})ds
\end{align} and where $u$ solves the Kolmogorov backward equation $\mathcal{G}^{V}u=f$ with terminal condition $u(T,\cdot)=u^{T}$.
\end{enumerate} 
\end{definition}
\begin{remark}
To solve the martingale problem, we do not need to solve the Kolmogorov equation for \textit{singular} terminal conditions. Instead it suffices to consider regular terminal conditions $u^{T}\in\mathcal{C}^{3}$, cf.~also \cite{kp}. To prove the equivalence of rough weak solutions and martingale solutions, we will employ \cref{thm:singular} in its full generality.
\end{remark}

\noindent We state a version of \cite[Theorem 4.2]{kp} about the existence and uniqueness of martingale solutions and \cite[Corollary 4.5]{kp} concerning moment estimates for the drift term.
\begin{theorem}\label{thm:mainthm1}\footnote{The statement of the theorem holds also under the weaker enhancement assumption on $V$ from \cite[Definition 3.5]{kp}, cf.~\cite[Theorem 4.2]{kp}.}
Let $\alpha\in (1,2]$ and $L$ be a symmetric, $\alpha$-stable L\'evy process, such that the measure $\nu$ satisfies \cref{ass}. Let $T>0$ and $\beta\in ((2-2\alpha)/3,0)$ and let $V\in\calX^{\beta,\gamma}$. Then for all $x\in\R^{d}$, there exists a unique solution $\mathbb{Q}$ on $(\Omega,\mathcal F)$ of the martingale problem for $(\mathcal{G}^{V},\delta_x)$. Under $\mathbb{Q}$ the canonical process is a strong Markov process.
\end{theorem}

\begin{lemma}\label{lem:hd}
In the setting of Theorem~\ref{thm:mainthm1}, let $(V^{n})_{n \in \N} \subset C_{T}C^{\infty}_{b}(\R^d,\R^d)$ be a smooth approximation with $(V^{n},\calK(V^{n}))\to \mathcal{V}$ in $\calX^{\beta,\gamma}$. Let $(X^{n}_{t})_{t\in [0,T]}$ be the strong solution of the SDE
\begin{align*}
dX_{t}^{n}=V^{n}(t,X^{n}_{t})dt+dL_{t},\qquad X_{0}=x\in \R^{d}.
\end{align*}
Let\footnote{In \cite[Corollary 4.5]{kp}, the result was proven for $\theta<\alpha+\beta$, with $\theta$ being the regularity of the PDE solution from \cite[section 3]{kp}. Due to \cref{thm:singular}, we can indeed generalize to $\theta=\alpha+\beta$.} $\theta=\alpha+\beta$ and $\rho\in 2\N$.\\ Then, there exists $N\in\N$, such that uniformly in
$0\leqslant r \leqslant t \leqslant T$:
\begin{align}\label{eq:h1}
\sup_{n\geqslant N}\E\bigg[\abs[\bigg]{\int_{r}^{t}V^{n}(s,X^{n}_{s})ds}^{\rho}\bigg]\lesssim\abs{t-r}^{\theta\rho/\alpha}.
\end{align}
\end{lemma}

\noindent Let us end the preliminaries by stating the stochastic sewing lemma from \cite[Theorem 2.1]{le}, that will be a key tool in this paper.\\
Here and in the following, we use the notation $Z_{st}=Z_{s,t}=Z_{t}-Z_{s}$, $s\leqslant t$, for the increment of a stochastic process $(Z_{t})$ and $\E_{s}[\cdot]:=\E[\cdot\mid\F_{s}]$ for the conditional expectation. A stochastic process $(\Xi_{st})_{(s,t)\in\Delta_{T}}$ indexed by $\Delta_{T}$ we call adapted to $(\F_{t})$ if $\Xi_{st}$ is $\F_{t}$-measurable for all $(s,t)\in\Delta_{T}$. For $p\in [2,\infty]$, $m\in\N$ and an adapted $\R^{m}$-valued stochastic process $(\Xi_{st})_{(s,t)\in\Delta_{T}}$ indexed by $\Delta_{T}$, we define for $\theta\in (0,1)$,
\begin{align}\label{eq:ss-notation}
\norm{\Xi}_{\theta,p}:=\sup_{(s,t)\in\Delta_{T}}\frac{\norm{\Xi_{st}}_{L^{p}(\p,\R^{m})}}{\abs{t-s}^{\theta}},\quad \norm{\Xi\mid \F}_{\theta,p}:=\sup_{(s,t)\in\Delta_{T}}\frac{\norm{\E_{s}[\Xi_{st}]}_{L^{p}(\p,\R^{m})}}{\abs{t-s}^{\theta}}.
\end{align}
If $(Z_{t})_{t\in[0,T]}$ is an adapted process indexed by $[0,T]$, we use the same notation \eqref{eq:ss-notation} for $\Xi_{st}=Z_{st}=Z_{t}-Z_{s}$.
\begin{lemma}[Stochastic sewing lemma]\label{lem:ss}
Let $(\Omega,\F, (\F_{t})_{t\in [0,T]},\p)$ be a complete probability space. Let $(\Xi_{s,t})_{0\leqslant s\leqslant t\leqslant T}$ be a two-parameter stochastic process with values in $\R^{d}$, that is adapted ($\Xi_{s,t}$ being $\F_{t}$-measurable for $s\leqslant t$) and $L^{2}(\p)$-integrable. Let $\delta\Xi_{s,u,t}:=\Xi_{st}-\Xi_{su}-\Xi_{ut}$, $0\leqslant s\leqslant u\leqslant t\leqslant T$. Suppose that there are constants $\Gamma_{1},\Gamma_{2},\epsilon_{1},\epsilon_{2}>0$, such that for all $0\leqslant s\leqslant u\leqslant t\leqslant T$,
\begin{align}
\norm{\E[\delta\Xi_{s,u,t}\mid\F_{s}]}_{L^{2}(\p)}\leqslant\Gamma_{1}\abs{t-s}^{1+\epsilon_{1}},\quad\norm{\delta\Xi_{s,u,t}}_{L^{2}(\p)}\leqslant\Gamma_{2}\abs{t-s}^{\frac{1}{2}+\epsilon_{2}}.
\end{align}
Then, there exists a unique (up to modifications) stochastic process $(I_{t})_{t\in [0,T]}:=(I_{t}(\Xi))_{t\in[0,T]}$ with values in $\R^{d}$ satisfying the following properties
\begin{enumerate}
\item[•] $I_{0}=0$, $(I_{t})_{t\in [0,T]}$ is $(\F_{t})$-adapted and $L^{2}(\p)$-integrable and
\item[•] there exist constants $C_{1}=C(\epsilon_{1}),C_{2}=C(\epsilon_{2})>0$, such that for all $0\leqslant s\leqslant t\leqslant T$,
\begin{align}
&\norm{I_{t}-I_{s}-\Xi_{s,t}}_{L^{2}(\p)}\leqslant C_{1}\Gamma_{1}\abs{t-s}^{1+\epsilon_{1}}+C_{2}\Gamma_{2}\abs{t-s}^{\frac{1}{2}+\epsilon_{2}},\nonumber\\& \norm{\E[I_{t}-I_{s}-\Xi_{s,t}\mid\F_{s}]}_{L^{2}(\p)}\leqslant C_{1}\Gamma_{1}\abs{t-s}^{1+\epsilon_{1}}.
\end{align}
\end{enumerate}
Furthermore, for every $t\in[0,T]$ and any partition $\Pi=\{0=t_{0}<t_{1}<\dots<t_{N}=T\}$, the Riemann sums $I^{\Pi}_{t}=\sum_{i=0}^{N-1}\Xi_{t_{i},t_{i+1}}$ converge to $I_{t}$ in $L^{2}(\p)$ for vanishing mesh size $\abs{\Pi}:=\max_{i}\abs{t_{i+1}-t_{i}}\to 0$.
\end{lemma}
\begin{remark}
We can apply the stochastic sewing lemma to a germ $\Xi_{st}:=f_{s}Y_{st}=f_{s}(Y_{t}-Y_{s})$ for stochastic processes $(f_t),(Y_t)$. Then, typically the constants $\Gamma_{1},\Gamma_{2}$ depend linearily on the Hölder-type moment bounds of the stochastic processes $f$ and $Y$. The bounds from \cref{lem:ss} then imply that $\sup_{0\leqslant s<t\leqslant T}\norm{I_{t}-I_{s}-\Xi_{s,t}}_{L^{2}(\p)}\abs{t-s}^{-(1/2+\epsilon_{2})}\lesssim T^{1/2+\epsilon_{1}-\epsilon_{2}}\Gamma_{1}+\Gamma_{2}$ (assuming $\epsilon_{2}\in (0,1/2)$), which  yields the stability of the stochastic sewing integral.
\end{remark}

\end{section}
\begin{section}{Weak rough-path-type solutions}\label{sec:ws}
\noindent In this section, we define a (rough) weak solution to the singular SDE 
\begin{align}\label{eq:ssde}
dX_{t}=V(t,X_{t})dt+dL_{t},\quad X_{0}=x\in\R^{d}
\end{align}
for an enhanced Besov drift $V\in\mathcal{X}^{\beta,\gamma}$ with $\beta\in (\frac{2-2\alpha}{3},0)$ and $\gamma\in [\frac{2\beta+2\alpha-1}{\alpha},1)$ in the rough case, respectively $\gamma\in (\frac{1-\beta}{\alpha},1)$ in the Young case (cf. \cref{def:enhanced-dist}). We furthermore state our main \cref{thm:mainthm}, that proves equivalence of rough weak solutions and martingale solutions \cref{def:martp}. The proof of \cref{thm:mainthm} follows from \cref{thm:wsimp} and \cref{thm:mpiws} in \cref{sec:wsemp}.\\
In the case of (locally) bounded drifts $V$ the equivalence of solutions to the martingale problem and weak solutions is known by \cite[Theorem 8.1.1]{Stroock2006} in the Brownian noise case and by \cite[Theorem 1.1]{Kurtz2011} in the Lévy noise setting. 
The first attempt to define weak solutions in the singular drift case, is to replace the singular drift by the limit $Z$ of the drift terms $\int_{0}^{\cdot}V^{n}(t,X_{t})dt$ for a smooth sequence $(V^{n})$ approximating $V$. In this way, one obtains a \textit{canonical} weak solution concept. In the Young case, this yields a well-posed solution concept (if moreover requiring regularity bounds on the drift $Z$). However, it turns out that canonical weak solutions are in general non-unique in the rough regime (cf.~\cref{sec:wsyc} below). In this section, we thus adapt the canonical weak solution concept, imposing additional assumptions to ensure well-posedness. The idea for a (rough) weak solution is to impose rough-paths-type assumptions on certain iterated integrals $\mathbb{Z}^{V}$, that formally correspont to the resonant component in the enhancement $\calV$. 
This motivates \cref{def:ws} below.\\

\noindent We call a filtered probability space $(\Omega,\F,(\F)_{t\in[0,T]},\p)$ a stochastic basis, if $(\Omega,\F,(\F)_{t\in[0,T]},\p)$ is complete and the filtration $(\F_{t})$ is right-continuous. We call a process $L$ a $(\F_{t})$-Lévy process, if $L$ is adapted to $(\F_{t})$ with $L_{0}=0$, $L_{t}-L_{s}$ being independent of $\F_{s}$ and $L_{t}-L_{s}\stackrel{d}{=}L_{t-s}$ for all $0\leqslant s<t\leqslant T$. 

\begin{definition}[Rough weak solution]\label{def:ws}
Let $V\in\calX^{\beta,\gamma}$ for $\beta\in (\frac{2-2\alpha}{3},0)$ (with $\gamma\in [\frac{2\beta+2\alpha-1}{\alpha},1)$ in the rough case, respectively $\gamma\in [\frac{1-\beta}{\alpha},1)$ in the Young case). Let $x\in\R^{d}$. We call a triple $(X,L,\mathbb{Z}^{V})$ a rough weak solution to the SDE \eqref{eq:ssde} starting at $X_{0}=x\in\R^{d}$, if there exists a stochastic basis $(\Omega,\F, (\F_{t})_{t\geqslant 0},\p)$, such that $L$ is an $\alpha$-stable symmetric non-degenerate $(\F_{t})$-Lévy process and almost surely
\begin{align*}
X=x+Z+L,
\end{align*} where $Z$ is a continuous and $(\F_{t})$-adapted process with the property that
\begin{align}
\norm{Z}_{\frac{\alpha+\beta}{\alpha},2}+\norm{Z\mid\F}_{\frac{\alpha+\beta}{\alpha},\infty}<\infty.\label{eq:zhoelder1}
\end{align} 
Moreover $(\mathbb{Z}^{V}_{st})_{(s,t)\in\Delta_{T}}$ is a continuous, $(\F_{t})$-adapted, $\R^{d\times d}$-valued stochastic process indexed by $\Delta_{T}$ with 
\begin{align}
\norm{\mathbb{Z}^{V}}_{\frac{\alpha+\beta}{\alpha},2}+\norm{\mathbb{Z}^{V}\mid\F}_{\frac{2\alpha+2\beta-1}{\alpha},\infty}<\infty.\label{eq:ahoelder1}
\end{align} 
Furthermore, $Z$ and $\mathbb{Z}^{V}$ are given as follows.  There exists a sequence $(V^{n})\subset C_{T}(C^{\infty}_{b})^{d}$ with $(V^{n},\mathcal{K}(V^{m},V^{n}))\to (V,\mathcal{V}_{2})$ in $\calX^{\beta,\gamma}$ for $m,n\to\infty$ with the following properties:
\begin{enumerate}
\item[1.)] $Z^{n}:=\int_{0}^{\cdot}V^{n}(s,X_{s})ds$ converges to $Z$ in the sense that
\begin{align}
\lim_{n\to\infty}[\norm{Z^{n}-Z}_{\frac{\alpha+\beta}{\alpha},2}+\norm{Z^{n}-Z\mid\F}_{\frac{\alpha+\beta}{\alpha},\infty}]=0
\end{align} 
and 
\item[2.)] $\mathbb{Z}^{m,n}_{st}=(\mathbb{Z}^{m,n}_{st}(i,j))_{i,j}=\paren[\big]{\int_{s}^{t} [J^{T}(\partial_{i} V^{m,j})(r,X_{r})-J^{T}(\partial_{i} V^{m,j})(s,X_{s})]dZ_{r}^{n,i}}_{i,j}$ converges to $\mathbb{Z}^{V}$ in the sense that
\begin{align}
\lim_{m,n\to\infty}[\norm{\mathbb{Z}^{m,n}-\mathbb{Z}^{V}}_{\frac{\alpha+\beta}{\alpha},2}+\norm{\mathbb{Z}^{m,n}-\mathbb{Z}^{V}\mid\F}_{\frac{2\alpha+2\beta-1}{\alpha},\infty}]=0.
\end{align} 
\end{enumerate}
We also call $X$ a (rough) weak solution, if there exists a stochastic basis, an $\alpha$-stable symmetric non-degenerate Lévy process $L$ and a stochastic process $\mathbb{Z}^{V}$, such that $(X,L,\mathbb{Z}^{V})$ is a rough weak solution.
\end{definition}
\begin{remark}[Notation] We also use the following abbreviations. Let
\begin{align*}
\mathbb{Z}^{m,n}_{st}(i,j)
&=\int_{s}^{t} J^{T}(\partial_{i} V^{m,j})(r,X_{r})dZ_{r}^{n,i}-J^{T}(\partial_{i} V^{m,j})(s,X_{s})Z^{n,i}_{st}\\&=:\mathbb{A}^{m,n}_{st}(i,j)-J^{T}(\partial_{i} V^{m,j})(s,X_{s})Z^{n,i}_{st}.
\end{align*}  
Let furthermore 
\begin{align*}
\mathbb{A}^{i}_{st}:=\lim_{m,n\to\infty}(\mathbb{A}^{m,n}_{st}(i,j))_{j=1,\dots,d}=:\lim_{m,n\to\infty}\mathbb{A}^{m,n,i}_{st}
\end{align*} for $i=1,\dots,d$. Here, the convergence with respect to $\norm{\cdot}_{\frac{\alpha+\beta}{\alpha},2}$ follows from the convergences $1.)$ and $2.)$ and that $J^{T}(\partial_{i}V^{m,j})\to J^{T}(\partial_{i}V^{j})$ in $C_{T}L^{\infty}$. Moreover, $(\mathbb{A}^{m,n})_{m,n},\mathbb{A}$ satisfy the same bounds as $Z$ in \eqref{eq:zhoelder1}.\\
We write 
\begin{align*}
\mathbb{Z}^{V,i}_{st}:=\lim_{m,n\to\infty}(\mathbb{Z}^{m,n}_{st}(i,j))_{j=1,\dots,d}=:\lim_{m,n\to\infty}\mathbb{Z}^{m,n,i}_{st}
\end{align*} 
for $i=1,\dots,d$. And for fixed $m\in\N$, we define
$\mathbb{A}^{m,\infty,i}_{st}:=\int_{s}^{t}J^{T}(\partial_{i}V^{m})(r,X_{r})dZ_{r}^{i}$,  
and analogously we define $\mathbb{Z}^{m,\infty,i}_{st}:=\int_{s}^{t}[J^{T}(\partial_{i}V^{m})(r,X_{r})-J^{T}(\partial_{i}V^{m})(s,X_{s})]dZ_{r}^{i}$.
\end{remark}
\begin{remark}[$X$ is a Dirichlet process]\label{rem:Dirichlet} By the $L^{2}$-moment bound on $Z$ from \eqref{eq:zhoelder1}, $Z$ has zero quadratic variation as $2(\alpha+\beta)/\alpha>1$. Thus, $X=x+Z+L$ is a Dirichlet process, i.e.~the sum of a local martingale and a zero quadratic variation process, cf.~\cite[Definition 2.4]{cjml}. In particular, for $F\in C^{1,2}_{b}([0,T]\times\R^{d},\R)$, the Itô-formula for $F(t,X_{t})$ from \cite[Theorem 3.1]{cjml} holds, cf.~also the classical article \cite{Follmer1981}, which can be extended to time depending $F$, that are $C^{1}$ in time. Notice that $Z$ is continuous and thus in the Itô-formula the terms involving the quadratic variation and the pure jump part of $Z$ vanish. In particular, the stochastic integral $\int_{0}^{\cdot} \nabla F(s,X_{s})\cdot dZ_{s}$ is the limit, in probability, of the classical Riemann sums.
\end{remark}

\begin{remark}[One and all sequences $(V^{n})$]
If $(X,L,\mathbb{Z}^{V})$ is a weak solution and $\calV\in\calX^{\beta,\gamma}$, then the convergences in $1.)$ and $2.)$ are true for \underline{all} sequences $(V^{n})$ with $(V^{n},\mathcal{K}(V^{n},V^{m}))\to \calV$ in $\calX^{\beta,\gamma}$. This means, that $(Z^{n})$, $(\mathbb{Z}^{m,n})$ converge and the limit is the same for any such sequence $(V^{n})$. Indeed, this will follow from the equivalence proof of the solution concepts below (specifically \cref{thm:mpiws}).
\end{remark}

\noindent The following lemma is an application of \cref{lem:ss}. \cref{lem:generalws} will in particular be used to prove (together with \cref{thm:mainthm} and \cref{thm:class} below), that in the Young case rough weak solutions are equivalent to canonical weak solutions, defined in \cref{sec:wsyc}.  
\begin{lemma}\label{lem:generalws}
Let $(X,L,\mathbb{Z}^{V})$ be a weak solution and let $f\in C_{T}^{\theta/\alpha}L^{\infty}\cap C_{T}\calC^{\theta}$ for $\theta\in (0,1)$ with $(\theta+\alpha+\beta)/\alpha>1$. Let $i\in\{1,\dots,d\}$.\\ Then, for every $t\in[0,T]$, the stochastic sewing integral 
\begin{align*}
I(f,Z^i)_{t}:=\lim_{\abs{\Pi}\to 0}\sum_{s,r\in\Pi}f(s,X_{s})(Z_{r}^{i}-Z_{s}^{i})\in L^{2}(\p)
\end{align*} for finite partitions $\Pi$ of $[0,t]$ with $\abs{\Pi}=\max_{s,r\in\Pi}\abs{r-s}\to 0$, is well-defined and allows for the bound,
\begin{align*}
&\norm{I(f,Z^i)_{t}-I(f,Z^i)_{s}-f(s,X_{s})Z_{s,t}^{i}}_{L^{2}(\p)}\\&\quad\lesssim \norm{f}_{C_{T}L^{\infty}}\norm{Z}_{\frac{\alpha+\beta}{\alpha},2}\abs{t-s}^{(\alpha+\beta)/\alpha}\\&\qquad\qquad+\norm{f}_{C_{T}^{\theta/\alpha}L^{\infty}\cap C_{T}\calC^{\theta}}(1+\norm{Z}_{\frac{\alpha+\beta}{\alpha},2}^{\theta})\norm{Z\mid\F}_{\frac{\alpha+\beta}{\alpha},\infty}\abs{t-s}^{(\theta+\alpha+\beta)/\alpha}.
\end{align*}
In particular, for $\beta\in (\frac{1-\alpha}{2},0)$ (Young regime), the existence of the integral 
\begin{align*}
\mathbb{A}_{0,t}=\paren[\Big]{\int_{0}^{t}J^{T}(\partial_{i}V^{j})(s,X_{s})dZ^{i}_{s}}_{i,j}\in L^{2}(\p,\R^{d\times d}),\quad t\in[0,T]
\end{align*} and the bound $\norm{\mathbb{Z}^{V}}_{(\alpha+\beta)/\alpha,2}<\infty$, as well as the convergence $\norm{\mathbb{Z}^{m,n}-\mathbb{Z}^{V}}_{\frac{\alpha+\beta}{\alpha},2}\to 0$ in $2.)$ follow. 
Furthermore, the convergence $\norm{\mathbb{Z}^{m,n}-\mathbb{Z}^{V}\mid\F}_{\frac{2\alpha+2\beta-1}{\alpha},2}\to 0$ follows.
\end{lemma}
\begin{remark}
\cref{lem:generalws} implies well-definedness of $\mathbb{A}^{m,\infty}_{0,t}(i,j)$ and $\mathbb{A}^{m,n}_{0,t}(i,j)$ as the limit of the Riemann sums $\sum_{s,r\in\Pi} J^{T}(\partial_{i} V^{m,j})(s,X_{s})Z^{i}_{s,r}$, analogously for $\mathbb{A}^{m,n}_{0,t}(i,j)$. In particular, the stability of the stochastic sewing integral yields that, for fixed $m\in\N$, $\lim_{n\to\infty}\mathbb{A}^{m,n}_{st}=\mathbb{A}^{m,\infty}_{st}$ in $L^{2}(\p)$, uniformly in $(s,t)\in\Delta_{T}$.\\
Furthermore, the lemma shows that in the Young case $f=J^{T}(\partial_{i}V^{j})\in C_{T}^{(\alpha+\beta-1)/\alpha}L^{\infty}\cap C_{T}\calC^{\alpha+\beta-1}$ satisfies the assumptions for $\theta=\alpha+\beta-1$. Thus, existence of the integral $\mathbb{A}_{st}(i,j)$ follows. In the rough case, 
the regularity of $f=J^{T}(\partial_{i}V^{j})$ does not suffice, because $2(\alpha+\beta)-1\leqslant\alpha$ if $\beta\leqslant (1-\alpha)/2$. Thus, the bounds \eqref{eq:ahoelder1} on $\mathbb{Z}^{V}$ and the convergence in $2.)$ are non-trivial requirements on a weak solution in the rough case.
\end{remark}
\begin{remark}
Notice that the lemma yields the convergence $\norm{\mathbb{Z}^{m,n}-\mathbb{Z}^{V}}_{\frac{2\alpha+2\beta-1}{\alpha},2}\to 0$ in the Young case, but not $\norm{\mathbb{Z}^{m,n}-\mathbb{Z}^{V}}_{\frac{2\alpha+2\beta-1}{\alpha},\infty}\to 0$. In fact the latter also holds true and we refer to \cref{thm:class} below. 
\end{remark}
\begin{proof}[Proof of \cref{lem:generalws}]
We apply the stochastic sewing lemma, \cref{lem:ss}, to the germ 
\begin{align*}
\Xi_{st}=f(s,X_{s})(Z^{i}_{t}-Z^{i}_{s}).
\end{align*}
We have that 
\begin{align*}
\delta\Xi_{srt}=\Xi_{st}-\Xi_{sr}-\Xi_{rt}=[f(r,X_{r})-f(s,X_{s})]Z^{i}_{rt}.
\end{align*} To prove the bound on the expectation in the stochastic sewing lemma, we will use the trivial estimate $\abs{f(r,X_{r})-f(s,X_{s})}\leqslant 2\norm{f}_{C_{T}L^{\infty}}\lesssim 2 \norm{f}_{C_{T}\calC^{\theta}}$ as $\theta>0$ and the bound on $Z$ from \eqref{eq:zhoelder1}, such that
\begin{align*}
\E[\abs{\delta\Xi_{srt}}^{2}]\leqslant\norm{f}_{C_{T}L^{\infty}}^{2}\E[\abs{Z^{i}_{rt}}^{2}]\leqslant\norm{f}_{C_{T}\calC^{\theta}}^{2}\norm{Z}_{\frac{\alpha+\beta}{\alpha},2}^{2}\abs{t-s}^{2(\alpha+\beta)/\alpha}
\end{align*} with $(\alpha+\beta)/\alpha>1/2$.
For bound on the conditional expectation, we utilize the bounds \eqref{eq:zhoelder1} on $Z$ 
and the time and space regularity of $f$, such that
\begin{align*}
\abs{f(r,X_{r})-f(s,X_{s})}&\leqslant\abs{f(r,X_{r})-f(s,X_{r})}+\abs{f(s,X_{s})-f(s,X_{r})}\\&\lesssim \norm{f}_{C_{T}^{\theta/\alpha}L^{\infty}}\abs{r-s}^{\theta/\alpha}+\norm{f}_{C_{T}\calC^{\theta}}\abs{X_{r}-X_{s}}^{\theta}\\&\lesssim \norm{f}_{C_{T}^{\theta/\alpha}L^{\infty}\cap C_{T}\calC^{\theta} }[\abs{r-s}^{\theta/\alpha}+\abs{X_{r}-X_{s}}^{\theta}],
\end{align*} where we used that the norm in $\calC^{\theta}$ is equivalent to the norm in the Hölder space $C^{\theta}_{b}$ of bounded, $\theta$-Hölder continuous functions for $\theta\in (0,1)$ (cf.~\cite[Section 2.7, Examples]{Bahouri2011}). 
Hence, with $\abs{X_{r}-X_{s}}\leqslant\abs{Z_{r}-Z_{s}}+\abs{L_{r}-L_{s}}$, we can estimate 
\begin{align*}
\MoveEqLeft
\E[\abs{\E_{s}[\delta\Xi_{srt}]}^{2}]
\\&=\E[\abs{\E_{s}[f(r,X_{r})-f(s,X_{s}))\E_{r}[Z_{rt}^{i}]]}^{2}]\\&\leqslant\E[\E_{s}[\abs{f(r,X_{r})-f(s,X_{s})}\abs{\E_{r}[Z_{rt}^{i}]}]^{2}]\\&\leqslant\norm{\E_{r}[Z_{rt}^{i}]}_{L^{\infty}(\p)}^{2}\E[\E_{s}[\abs{f(r,X_{r})-f(s,X_{s})}]^{2}]\\&\lesssim\norm{Z\mid\F}_{\frac{\alpha+\beta}{\alpha},\infty}^{2}\abs{t-r}^{2(\alpha+\beta)/\alpha}\norm{f}_{C_{T}^{\theta/\alpha}L^{\infty}\cap C_{T}\calC^{\theta}}^{2}\times\\&\qquad\qquad\paren[\big]{\abs{r-s}^{2\theta/\alpha}+\E[\abs{Z_{r}-Z_{s}}^{2\theta}]+\E[\E_{s}[\abs{L_{r}-L_{s}}^{\theta}]^{2}]}\\&\lesssim_{T}\norm{Z\mid\F}_{\frac{\alpha+\beta}{\alpha},\infty}^{2}(1+\norm{Z}_{\frac{\alpha+\beta}{\alpha},2}^{2\theta})\norm{f}_{C_{T}^{\theta/\alpha}L^{\infty}\cap C_{T}\calC^{\theta}}^{2}\abs{t-s}^{2(\theta+\alpha+\beta)/\alpha}
\end{align*} 
In the last estimate above, we used stationarity, scaling and independence of the increment $L_{r}-L_{s}$ of $\F_{s}$ for $s\leqslant r$, such that almost surely 
\begin{align*}
\E_{s}[\abs{L_{r}-L_{s}}^{\theta}]=\E[\abs{L_{r}-L_{s}}^{\theta}]=\abs{r-s}^{\theta/\alpha}\E[\abs{L_{1}}^{\theta}]\lesssim \abs{r-s}^{\theta/\alpha}
\end{align*} and the expectation is finite due to $\theta<\alpha$. Furthermore, as $\theta\in (0,1)$, we used above Jensen's inequality and the following estimate on $Z$:
\begin{align*}
\E[\abs{Z_{r}-Z_{s}}^{2\theta}]\leqslant\E[\abs{Z_{r}-Z_{s}}^{2}]^{\theta}&\leqslant\norm{Z}_{\frac{\alpha+\beta}{\alpha},2}^{2\theta}\abs{r-s}^{2\theta(\alpha+\beta)/\alpha}\lesssim_{T}\norm{Z}_{\frac{\alpha+\beta}{\alpha},2}^{2\theta}\abs{r-s}^{2\theta/\alpha}.
\end{align*} 
Due to $(\theta+\alpha+\beta)/\alpha>1$, \cref{lem:ss} applies and yields the bound for $I(f,Z^i)$.\\
In the Young case, $\beta>(1-\alpha)/2$, we can take $f=J^{T}(\partial_{i}V^{j})\in C_{T}^{(\alpha+\beta-1)/\alpha}L^{\infty}\cap C_{T}\calC^{\alpha+\beta-1}$ with $(2\alpha+2\beta-1)/\alpha>1$, which yields existence of the integral $I(\Xi)_{t}=\mathbb{A}_{0,t}(i,j)$.
Let $I(\Xi^{m,n})$ be the sewing integral with germ \begin{align*}
\Xi_{sr}^{m,n}=J^{T}(\partial_{i}V^{m,j})(s,X_{s})(Z^{n,i}_{r}-Z^{n,i}_{s})
\end{align*} and let $f^{m}:=J^{T}(\partial_{i}V^{m,j})$.
By the Schauder and interpolation estimates we furthermore have that $f^{m}\to f$ in $\mathcal{L}_{T}^{0,\alpha+\beta-1}$ and in $C_{T}^{(\alpha+\beta-1)/\alpha}L^{\infty}\cap C_{T}\calC^{\alpha+\beta-1}$.\\
Let $\mathbb{Z}^{V}_{st}=\mathbb{A}_{st}-\Xi_{st}=I(\Xi)_{st}-\Xi_{st}$. Then, \cref{lem:ss} yields that $\norm{\mathbb{Z}^{V}}_{\frac{\alpha+\beta}{\alpha},2}<\infty$ and $\norm{\mathbb{Z}^{V}\mid \F}_{\frac{2\alpha+2\beta-1}{\alpha},2}<\infty$. The stability of the sewing integral implies the convergence $\norm{\mathbb{Z}^{m,n}-\mathbb{Z}^{V}}_{\frac{\alpha+\beta}{\alpha},2}\to 0$, since by the estimates above we obtain for $(s,t)\in\Delta_{T}$ that
\begin{align*}
\MoveEqLeft
\norm{\mathbb{Z}^{m,n}_{st}-\mathbb{Z}^{V}_{st}}_{L^{2}}\\&\leqslant\norm{I(\Xi^{m,n})_{st}-\Xi^{m,n}_{s,t}-(I(\Xi)_{st}-\Xi_{s,t})}_{L^{2}}
\\&\lesssim_{T}\abs{t-s}^{(\alpha+\beta)/\alpha}\paren[\Big]{\sup_{n}\norm{Z^{n}\mid\F}_{\frac{\alpha+\beta}{\alpha},\infty}(1+\norm{Z}_{\frac{\alpha+\beta}{\alpha},2}^{\theta})\norm{f^{m}-f}_{\mathcal{L}_{T}^{0,\alpha+\beta-1}}
\\&\qquad \qquad\hspace{3em} + \norm{Z^{n}-Z\mid\F}_{\frac{\alpha+\beta}{\alpha},\infty}(1+\norm{Z}_{\frac{\alpha+\beta}{\alpha},2}^{\theta})\norm{f}_{\mathcal{L}_{T}^{0,\alpha+\beta-1}}
\\&\qquad\qquad\hspace{3em}+\norm{f^{m}-f}_{\mathcal{L}_{T}^{0,\alpha+\beta-1}}\sup_{n}\norm{Z^{n}}_{\frac{\alpha+\beta}{\alpha},2}
\\&\qquad\qquad\hspace{3em}+\norm{f}_{\mathcal{L}_{T}^{0,\alpha+\beta-1}}\norm{Z^{n}-Z}_{\frac{\alpha+\beta}{\alpha},2}}
\\&\to 0
\end{align*} for $n,m\to\infty$.
Analogously we can show that $\norm{\mathbb{Z}^{m,n}-\mathbb{Z}^{V}\mid\F}_{\frac{2\alpha+2\beta-1}{\alpha},2}\to 0$.
\end{proof}

\noindent Let $\calV=(V,\calV_{2})\in\calX^{\beta,\gamma}$ and let $(X,L,\mathbb{Z}^{V})$ be a weak solution. Then we expect (for a proof see \cref{thm:mpiws} below) the following representations of $Z$ and $\mathbb{A}$ (and thus of $\mathbb{Z}^{V}$):
\begin{align*}
\E_{s}[Z_{s,t}^{i}]=\E_{s}[u^{t,i}(t,X_{t})-u^{t,i}(s,X_{s})]\,\text{ and }\,\E_{s}[\mathbb{A}_{s,t}(i,j)]=\E_{s}[v^{t,i,j}(t,X_{t})-v^{t,i,j}(s,X_{s})]
\end{align*} for the solutions $u^{t}=(u^{t,i})_{i=1,\dots,d}$, $v^{t}=(v^{t,i,j})_{i,j=1,\dots,d}$ of the backward PDEs
\begin{align}\label{eq:u}
\mathcal{G}^{\calV}u^{t,i}=V^{i},\quad u^{t,i}(t,\cdot)=0.
\end{align} and 
\begin{align}\label{eq:v}
\mathcal{G}^{\calV}v^{t,i,j}&=J^{T}(\partial_{i} V^{j})\cdot V^{i}\nonumber\\&=\calV_{2}(i,j)+J^{T}(\partial_{i} V^{j})\para V^{i}+J^{T}(\partial_{i} V^{j})\arap V^{i},\quad v^{t,i,j}(t,\cdot)=0.
\end{align}
Let for $i\in\{1,\dots,d\}$, $v^{t,i}:=(v^{t,i,j})_{j=1,\dots,d}$.\\
Together with the bound $\norm{\mathbb{Z}^{V}\mid\F}_{\frac{2\alpha+2\beta-1}{\alpha},\infty}<\infty$ from \eqref{eq:ahoelder1}, this motivates the definition of the class of processes $\mathcal{K}^{\vartheta}=\mathcal{K}^{\vartheta}(\calV)$.
\begin{definition}[Class $\mathcal{K}^{\vartheta}$]\label{def:class}
Let $\beta\in (\frac{2-2\alpha}{3},0)$, $\calV\in\calX^{\beta,\gamma}$ and $\vartheta \in ((\alpha+\beta)/\alpha,1]$. Let for $t\in(0,T]$, $v^{t}$  be the solution of the PDE \eqref{eq:v} 
and $u^{t}$ be the solution of \eqref{eq:u}.\\
An adapted càdlàg stochastic process $(X_{t})_{t\in[0,T]}$ is said to be of class $\mathcal{K}^{\vartheta}(\calV)$ if for all $i=1,\dots,d$:
\begin{align}\label{eq:class}
\MoveEqLeft
\norm{\E_{s}[(v^{t,i}(t,X_{t})-v^{t,i}(s,X_{s}))-J^{T}(\partial_{i} V)(s,X_{s})(u^{t,i}(t,X_{t})-u^{t,i}(s,X_{s}))]}_{(L^{\infty}(\p))^d}\nonumber\\&=\max_{j=1,\dots,d}\norm{\E_{s}[(v^{t,i,j}(t,X_{t})-v^{t,i,j}(s,X_{s}))-J^{T}(\partial_{i} V^{j})(s,X_{s})(u^{t,i}(t,X_{t})-u^{t,i}(s,X_{s}))]}_{L^{\infty}(\p)}\nonumber
\\&\lesssim\abs{t-s}^{\vartheta}.
\end{align}
\end{definition}
\begin{remark}
Using $v^{t,i,j}, u^{t,i}\in C_{t}^{(\alpha+\beta)/\alpha}L^{\infty}$ with norm uniformly bounded for $t\in [0,T]$ by \cite[Corollary 4.14]{kp-sk},
a trivial estimate yields \eqref{eq:class} for $\vartheta=(\alpha+\beta)/\alpha$. We are interested in the case $\vartheta>(\alpha+\beta)/\alpha$.
\end{remark}
\noindent As an intermediate step in the equivalence proof of the solution concepts we will prove (cf.~\cref{thm:class}) that the solution of the martingale problem for $\mathcal{G}^{\calV}$ is a process of class $\mathcal{K}^{\vartheta}(\calV)$ for $\vartheta=(2\alpha+2\beta-1)/\alpha$.\\
The following result is the main theorem of this paper. The proof follows directly from \cref{thm:wsimp,thm:mpiws} in \cref{sec:wsemp}.
\begin{theorem}\label{thm:mainthm}
Let $\alpha\in (1,2]$, $\beta\in (\frac{2-2\alpha}{3},0)$
and $V\in\calX^{\beta,\gamma}$.  Let $x\in\R^{d}$.
Then, $X$ is a rough weak solution in the sense of \cref{def:ws} if and only if $X$ solves the $(\mathcal{G}^{\calV},\delta_{x})$-martingale problem.
 In particular, rough weak solutions are unique in law.
\end{theorem}
\noindent The second main theorem of this work generalizes Itô's formula for rough weak solutions. We state an informal version of \cref{thm:genIto} in \cref{thm:igenIto} below.
\begin{theorem}\label{thm:igenIto}
Let $\alpha\in (1,2]$, $\beta\in (\frac{2-2\alpha}{3},0)$
and $V\in\calX^{\beta,\gamma}$. Let $(X,L,\mathbb{Z}^{V})$ be a rough weak solution and
let $u\in C_{T}\calC^{\alpha+\beta}\cap C^{1}([0,T],\calC^{\beta})$ be such that $(\partial_{t}-\La)u$ is paracontrolled by $V$. 
Then, if $\alpha\in (1,2)$, the following Itô-formula holds true:
\begin{align*}
u(t,X_{t})&=u(0,x)+\int_{0}^{t}(\partial_{s}-\La)u(s,X_{s})ds + \int_{0}^{t}\nabla u(s,X_{s})\cdot d(Z,\mathbb{Z}^{V})_{s}\\&\qquad+\int_{0}^{t}\int_{\R^{d}\setminus\{0\}}[u(s,X_{s-}+y)-u(s,X_{s-})]\hat{\pi}(ds,dy),
\end{align*} where $Z=X-x-L$ and $\hat{\pi}$ denotes the compensated Poisson random measure of $L$. If $\alpha=2$ and $L=B$ for a Brownian motion $B$, the martingale is replaced by $\int_{0}^{t}\nabla u(s,X_{s})\cdot dB_{s}$.
\end{theorem}

\noindent To prove \cref{thm:mainthm}, it turns out that for a weak solution $X$ we need to define the limit of the stochastic integrals 
\begin{align*}
\int_{0}^{t}\nabla u^{m}(s,X_{s})\cdot dZ_{s}=\sum_{i=1}^{d}\int_{0}^{t}\partial_{i}u^{m}(s,X_{s})dZ_{s}^{i},
\end{align*} as $m\to\infty$, where $u^{m}\in C^{1,2}_{b}([0,T]\times\R^{d})$ with $u^{m}\to u$ and $u$ solves the Kolmogorov backward equation (for regular data $f,u^{T}$). In the rough regime, $\partial_{i} u$ won't have enough regularity, such that we can define the integral of $\partial_{i}u(s,X_{s})$ against $Z$ in a stable manner utilizing \cref{lem:generalws}. Indeed, a regularity counting argument yields that the time regularity of $\partial_{i} u$, i.e. $\partial_{i}u\in C_{T}^{(\alpha+\beta-1)/\alpha}L^{\infty}$, together with the $(\alpha+\beta)/\alpha$-Hölder regularity of $Z$ in $L^{2}(\p)$ sum up to $(2\alpha+2\beta-1)/\alpha\leqslant 1$ for $\beta\leqslant(1-\alpha)/2$. The idea is thus to enhance $Z$ by $\mathbb{Z}^{V}$ from \cref{def:ws} in order to correct for the irregular terms in $\partial_{i}u$ and to define the stable rough stochastic integral 
\begin{align*}
\int_{0}^{t}\nabla u(s,X_{s})\cdot d(Z,\mathbb{Z}^{V})=\sum_{i=1}^{d}\int_{0}^{t}\partial_{i} u(s,X_{s})d(Z^{i},\mathbb{Z}^{V,i})_{s}.
\end{align*}
For regular integrands $\nabla u^{m}$ the stochastic integral against $Z$ and the rough stochastic integral against $(Z,\mathbb{Z}^{V})$ coincide. This motivates the general theory in the next section.

\end{section}

\begin{section}{A rough stochastic sewing integral}\label{sec:roughint}
In this section, we construct in \cref{thm:roughint} a rough stochastic integral using the stochastic sewing \cref{lem:ss}. The theory is inspired by \cite{le-friz-hocq}. 
Nonetheless, the results from \cite{le-friz-hocq} do not apply in our setting and our rough stochastic integral differs from the one constructed there. Instead of considering a $\gamma$-rough path as an integrator, we consider a lifted stochastic process $(Z,\mathbb{Z}^{A})$ with bounds with respect to the semi-norms $\norm{\cdot}_{\theta,2}$ and $\norm{\cdot\mid\F}_{\theta,\infty}$,  cf.~\eqref{eq:ss-notation}. The Hölder exponent $\sigma$ of $Z$ is assumed to satisfy $\sigma>1/2$. That Hölder-regularity is satisfied in the case of the drift term we are interested in. In this sense the integrator in good. However the integrand is rough and we need to assume that it is stochastically controlled by a given stochastic process $A$.  
We construct the integral in $L^{2}(\p)$, but replacing $2$ by $p\geqslant 2$ in the bounds below, we could more generally construct the integral in $L^{p}(\p)$. We will not bother doing so, as we aim for the situation where the integrand is given by a $\vartheta$-Hölder function of the $\alpha$-stable process for $\vartheta\in (0,\alpha/2)$, which lacks arbitrarily high moments, i.e.~$\E[\abs{L_{1}}^{2\vartheta}]<\infty$ as $2\vartheta<\alpha$, but we may not necessarily have $\E[\abs{L_{1}}^{p\vartheta}]<\infty$ for $p>2$.\\ 
Let us start with the definition of stochastically controlled processes. Let here and below, $(\Omega,\F, (\F_{t})_{t\in[0,T]},\p)$ be a complete filtered probability space.
\begin{definition}[Stochastically controlled processes]
Let $(A_{t})_{t\in[0,T]}$ be an $\R^{d}$-valued adapted stochastic process with $\sup_{t\in[0,T]}\norm{A_{t}}_{L^{\infty}(\p,\R^{d})}<\infty$ and $\norm{A}_{\varsigma,2}<\infty$ for $\varsigma\in (0,1)$. 
We call an adapted stochastic process $(f_{t}, f'_{t})_{t\in [0,T]}$ with $\sup_{t\in[0,T]}[\norm{f_{t}}_{L^{\infty}(\p,\R)}+\norm{f'_{t}}_{L^{\infty}(\p,\R^{d})}]<\infty$ stochastically controlled by $A$, if for $\varsigma,\varsigma'\in (0,1)$ the following bounds hold true
\begin{align*}
\norm{f}_{\varsigma,2}+\norm{f^{\prime}}_{\varsigma',2}<\infty
\end{align*} and the remainder $(R^{f}_{st})_{(s,t)\in\Delta_{T}}$ defined by
\begin{align*}
R^{f}_{st}:=f_{st}-f'_{s}\cdot A_{st},
\end{align*} satisfies 
\begin{align*}
\norm{R^{f}}_{\varsigma+\varsigma',2}<\infty.
\end{align*}  
The remainder is then an adapted process indexed by $\Delta_{T}$. We denote the space of all such $(f,f^{\prime})$, that are stochastically controlled by $A$ by $D^{\varsigma,\varsigma'}_{T}(A)$ and we define the complete norm 
\begin{align*}
\norm{f}_{D^{\varsigma,\varsigma'}_{T}(A)}&:=\sup_{t\in[0,T]}[\norm{f_{t}}_{L^{\infty}(\p)}+\norm{f'_{t}}_{L^{\infty}(\p,\R^{d})}] \\&\qquad+\norm{f}_{\varsigma,2}+\norm{f^{\prime}}_{\varsigma',2}+\norm{R^{f}}_{\varsigma+\varsigma',2}.
\end{align*}
\end{definition}
\begin{definition}[Rough stochastic integrator]\label{def:rough-integrator}
Let $(A_{t})_{t\in [0,T]}$ be an adapted $\R^{d}$-valued stochastic process, which satisfies $\sup_{t\in[0,T]}\norm{A_{t}}_{L^{\infty}(\p,\R^{d})}<\infty$ and $\norm{A}_{\varsigma,2}<\infty$ for $\varsigma\in (0,1)$. Let $\sigma\in (1/2,1)$. 
Then we call $(Z,\mathbb{Z}^{A})$ a rough integrator, if $(Z_{t})_{t\in[0,T]}$, $(\mathbb{Z}^{A}_{st})_{(s,t)\in\Delta_{T}}$ are adapted stochastic processes with $Z$ being $\R$-valued and $\mathbb{Z}^{A}$ being $\R^{d}$-valued, such that $Z_{0}=0$ and for all $0\leqslant s\leqslant l\leqslant t\leqslant T$, the following algebraic relation holds
\begin{align*}
\mathbb{Z}_{st}^{A}=\mathbb{Z}^{A}_{sl}+\mathbb{Z}^{A}_{lt}+A_{sl}Z_{lt}.
\end{align*} 
Furthermore the following Hölder-type moment bounds hold:
\begin{align*}
\norm{(Z,\mathbb{Z}^{A})}_{\mathcal{R}^{\sigma}_{T}(A)}:=\norm{Z}_{\sigma,2}+\norm{Z\mid\F}_{\sigma,\infty}+\norm{\mathbb{Z}^{A}}_{\sigma,2}+\norm{\mathbb{Z}^{A}\mid\F}_{\sigma+\varsigma,\infty}<\infty.
\end{align*} 
We call the space of stochastic processes, that are rough integrators $\mathcal{R}^{\sigma}_{T}(A)$. Equipped with the norm  
\begin{align*}
\norm{(Z,\mathbb{Z}^{A})}_{\mathcal{R}_{T}^{\sigma}(A)}&:=\norm{Z}_{\sigma,2}+\norm{Z-W\mid\F}_{\sigma,\infty}\\&\qquad+\norm{\mathbb{Z}^{A}}_{\sigma,2}+\norm{\mathbb{Z}^{A}\mid\F}_{\sigma+\varsigma,\infty}
\end{align*} the space becomes a Banach space. 
\end{definition}
\begin{remark}
We refer to \cite{fh} for the connection to rough paths. The space of rough stochastic integrators is a vector space, because the algebraic relation here is linear. Furthermore, we assume that $Z_{0}=0$ to obtain a norm.
\end{remark}
\begin{theorem}\label{thm:roughint}
Let $(A_{t})_{t\in [0,T]}$ be an adapted $\R^{d}$-valued stochastic process with $A\in C_{T}L^{\infty}(\p,\R^{d})$ and let $f$ be stochastically controlled by $A$. Let $(Z,\mathbb{Z}^{A})$ be a rough integrator. Let the parameters be such that
\begin{align*}
\sigma+\varsigma'+\varsigma>1.
\end{align*}
Then for $t\in(0,T]$, the rough stochatic sewing integral
\begin{align*}
I(f,Z)_{t}:=\int_{0}^{t}f_{s}\,d(Z,\mathbb{Z}^{A})_{s}=\lim_{\abs{\Pi}\to 0}\sum_{r,l\in\Pi}[f_{r}Z_{rl}+f'_{r}\cdot\mathbb{Z}^{A}_{rl}],
\end{align*} exists in $L^{2}(\p)$, where the limit ranges over partitions $\Pi$ of $[0,t]$ with mesh size $\abs{\Pi}\to 0$. Moreover, the following Lipschitz bound holds:
\begin{align*}
\sup_{(s,t)\in\Delta_{T}}\frac{\norm{I(f,Z)_{s,t}}_{L^{2}(\p)}}{\abs{t-s}^{\sigma}}\lesssim\norm{f}_{D^{\varsigma,\varsigma'}_{T}(A)}\norm{(Z,\mathbb{Z}^{A})}_{\mathcal{R}^{\sigma}_{T}(A)}.
\end{align*} 
Furthermore, if $\sigma+\varsigma>1$, then the rough stochastic integral $I(f,Z)_{t}$ almost surely agrees with the integral $\tilde{I}(f,Z)_{t}=\lim_{\abs{\Pi}\to 0}\sum_{r,l\in\Pi}f_{r}Z_{rl}\in L^{2}(\p)$.
\end{theorem}
\begin{proof}
Define for $r<s$, $\Xi_{rs}:=f_{r}Z_{rs}+f_{r}^{\prime}\cdot\mathbb{Z}^{A}_{rs}$. Then we have that for $r<l<s$, by the algebraic relation of $Z,\mathbb{Z}^{A}$ and the definition of the remainder $R=R^{f}$,
\begin{align*}
\delta\Xi_{r,l,s}=\Xi_{rs}-\Xi_{rl}-\Xi_{ls}&=-R_{r,l}Z_{l,s}-f_{r,l}'\cdot\mathbb{Z}^{A}_{l,s}.
\end{align*}
To apply the stochastic sewing \cref{lem:ss}, we need to show the bounds on the expectation and conditional expectation.
We start with the bound on the expectation. As $\sigma>1/2$ it suffices to trivially estimate the $L^{2}$-norm using  
\begin{align*}
\sup_{r,l\in\Delta_{T}}\norm{R_{r,l}}_{L^{\infty}}\leqslant 2\sup_{t\in[0,T]}[\norm{f_{t}}_{L^{\infty}(\p)}+\norm{f'_{t}}_{L^{\infty}(\p,\R^{d})}\norm{A_t}_{L^{\infty}(\p,\R^{d})}]<\infty,
\end{align*} such that
\begin{align*}
\E[\abs{\delta\Xi_{r,l,s}}^{2}]^{1/2}&\lesssim\sup_{r,l\in\Delta_{T}}\norm{R_{r,l}}_{L^{\infty}}\norm{Z_{l,s}}_{L^{2}}+ \sup_{t\in[0,T]}\norm{f'_{t}}_{L^{\infty}(\p,\R^{d})}\norm{\mathbb{Z}^{A}_{l,s}}_{L^{2}} \\&\lesssim\norm{f}_{D_{T}^{\varsigma,\varsigma'}(A)}\norm{(Z,\mathbb{Z}^{A})}_{\mathcal{R}_{T}^{\sigma}(A)}\abs{s-r}^{\sigma}.
\end{align*} 
For the conditional expectation, we have
\begin{align*}
\E[\abs{\E_{r}[\delta\Xi_{r,l,s}]}^{2}]&\leqslant\E\bigl[\E_{r}[\abs{R_{r,l}\E_{l}[Z_{l,s}]}]^{2}\bigr]+\E\bigl[\E_{r}[\abs{f'_{r,l}\cdot\E_{l}[\mathbb{Z}^{A}_{l,s}]}]^{2}\bigr]\\&\leqslant\norm{R_{r,l}}_{L^{2}(\p)}^{2}\norm{\E_{l}[Z_{l,s}]}_{L^{\infty}}^{2}+\norm{f'_{r,l}}_{L^{2}(\p,\R^{d})}^{2}\norm{\E_{l}[\mathbb{Z}^{A}_{l,s}]}_{L^{\infty}(\p,\R^{d})}^{2}
\\&\leqslant \norm{f}_{D_{T}^{\varsigma,\varsigma'}(A)}\norm{(Z,\mathbb{Z}^{A})}_{\mathcal{R}_{T}^{\sigma}(A)}\abs{s-r}^{2(\varsigma'+\sigma+\varsigma)},
\end{align*} with $\varsigma'+\sigma+\varsigma>1$ by assumption. Thus \cref{lem:ss} applies for the existence of $I(f,Z)_{t}$, $t\in[0,T]$. The prescribed Lipschitz bound follows from the bound on $\norm{I(f,Z)_{st}-\Xi_{st}}_{L^{2}}$ from \cref{lem:ss} and $\norm{\Xi_{st}}_{L^{2}}\leqslant\norm{f}_{D_{T}^{\varsigma,\varsigma'}(A)}[\norm{Z}_{\sigma,2}+\norm{\mathbb{Z}^{A}}_{\sigma,2}]\abs{t-s}^{\sigma}$.\\ If $\sigma+\varsigma>1$, then by the uniqueness of the stochastic sewing integral
\begin{align*} 
\tilde{I}(f,Z)_{t}= \lim_{\abs{\Pi}\to 0}\sum_{r<l\in\Pi}f_{r}Z_{rl}\in L^{2}(\p),
\end{align*} 
we obtain that $\tilde{I}(f,Z)_{t}=I(f,Z)_{t}$ almost surely. Indeed, this uses that $I$ satisfies the bounds for the germ $\tilde{\Xi}_{st}=f_{s}Z_{st}$ of $\tilde{I}$, that is
\begin{align*}
\norm{I(f,Z)_{st}-f_{s}Z_{st}}_{L^{2}}&\leqslant\norm{I(f,Z)_{st}}_{L^{2}}+\norm{f_{s}Z_{st}}_{L^{2}}\lesssim\abs{t-s}^{\sigma}, 
\end{align*} with $\sigma>1/2$ and using the bound of $I(f,Z)$ from \cref{lem:ss},
\begin{align*}
\norm{\E_{s}[I(f,Z)_{s,t}-f_{s}Z_{st}]}_{L^{2}}&\leqslant\norm{\E_{s}[I(f,Z)_{s,t}-f_{s}Z_{st}-f^{\prime}_{s}\mathbb{Z}^{A}_{st}]}_{L^{2}}+\norm{\E_{s}[f^{\prime}_{s}\mathbb{Z}^{A}_{st}]}_{L^{2}}
\\&\leqslant\norm{\E_{s}[I(f,Z)_{s,t}-f_{s}Z_{st}-f^{\prime}_{s}\mathbb{Z}^{A}_{st}]}_{L^{2}}+\norm{\E_{s}[f^{\prime}_{s}\mathbb{Z}^{A}_{st}]}_{L^{\infty}}
\\&\lesssim\abs{t-s}^{\sigma+\varsigma},
\end{align*} where $\sigma+\varsigma>1$ by assumption.
\end{proof}
\end{section}
\begin{section}{Equivalence of weak solution concepts}\label{sec:wsemp}
In this section, we prove in \cref{thm:wsimp,thm:mpiws} showing that the weak solution concept from \cref{def:ws} yields an equivalent notion of solution to the martingale solutions. Furthermore, we show in \cref{thm:genIto} an extension of the Itô formula for rough weak solutions.\\ 
To prove that a weak solution solves the martingale problem
we employ the stability of the rough stochastic integral. In functionalanalytic terms, the rough stochastic integral is the unique extension of the integral of regular integrands $u(s,X_{s})$, i.e.~for $u\in C^{1,2}_{b}$, against $Z$ to paracontrolled integrands:
\begin{align}\label{eq:target-r-int}
D_{T}(V)\times \mathcal{R}_{T}(V)\ni (u,(Z,\mathbb{Z}^{V}))\mapsto\int_{0}^{t}\nabla u(s,X_{s})\cdot d(Z,\mathbb{Z}^{V})_{s}\in L^{2}(\p).
\end{align}
Here $D_{T}(V)$ denotes the space of paracontrolled distributions defined by 
\begin{align}\label{def:good-para-dist}
D_{T}(V)=\bigl\{(u,u')\in \mathcal{L}_{T}^{0,\alpha+\beta}\times\mathcal{L}_{T}^{0,\alpha+\beta-1}\bigm| u^{\sharp}:=u-u'\para J^{T}(V)\in\mathcal{L}_{T}^{0,2(\alpha+\beta)-1}\bigr\}
\end{align} and
\begin{align*}
\mathcal{R}_{T}(V):=\bigtimes_{i=1}^{d}\mathcal{R}_{T}^{(\alpha+\beta)/\alpha}(A^{i}),
\end{align*} for $A^{i}=(J^{T}(\partial_{i}V)(t,X_{t}))_{t\in[0,T]}$ and $\mathcal{R}_{T}^{\sigma}(A)$ from \cref{def:rough-integrator} in \cref{sec:roughint}.\\
To obtain stability of the rough integral \eqref{eq:target-r-int}, we apply \cref{thm:roughint}, for any $i=1,\dots,d$, to $f=(\partial_{i} u(t,X_{t}))_{t\in[0,T]}$ and the rough integrator $(Z^{i},\mathbb{Z}^{V,i})$ that is given by the definition of a weak solution (cf.~\cref{def:ws}).\\ 
\cref{lem:nablau-stochcont} below shows that, if $u\in D_{T}(V)$, then $(\partial_{i} u(t,X_{t}))$ is stochastically controlled by $(A_{t}^{i})=(J^{T}(\partial_{i} V)(t,X_{t}))$ for $\varsigma=\varsigma'=(\alpha+\beta-1)/\alpha$. 
The following lemma verifies that $(Z^{i},\mathbb{Z}^{V,i})$ is a rough stochastic integrator in the sense of \cref{def:rough-integrator}.
\begin{lemma}\label{lem:z-rough-integrator}
Let $\beta\in (\frac{2-2\alpha}{3},0)$ and $V\in\calX^{\beta,\gamma}$. Let $(X,L,\mathbb{Z}^{V})$ be a weak solution. Then, $(Z,\mathbb{Z}^{V})\in\mathcal{R}_{T}(V)$, i.e.~$(Z^{i},\mathbb{Z}^{V,i})\in\mathcal{R}_{T}^{(\alpha+\beta)/\alpha}(A^i)$ for $A^i=(J^{T}(\partial_{i} V)(t,X_{t}))_{t\in[0,T]}$ and $i=1,\dots,d$. Moreover, the following convergence holds for $n,m\to\infty$ 
\begin{align*}
\norm{(Z^{n},\mathbb{Z}^{m,n})-(Z,\mathbb{Z}^{V})}_{\mathcal{R}_{T}(V)}\to 0.
\end{align*}
\end{lemma}
\begin{remark}
Notice that we use the notation $\norm{(Z^{n},\mathbb{Z}^{m,n})-(Z,\mathbb{Z}^{V})}_{\mathcal{R}_{T}(V)}$ despite the fact, that $(Z^{n},\mathbb{Z}^{m,n})$ and $(Z,\mathbb{Z}^{V})$ do not live in the same space. The notation means that the four semi-norms in \cref{def:rough-integrator} converge. 
\end{remark}
\begin{proof} 
Let $i=1,\dots,d$. To prove that $(Z^{i},\mathbb{Z}^{V,i})$ is a rough integrator, notice that for $\mathbb{A}^{i}=\lim_{n,m\to\infty}\mathbb{A}^{m,n,i}$ the following additivity holds: $\mathbb{A}^{i}_{st}=\mathbb{A}^{i}_{sl}+\mathbb{A}^{i}_{lt}$ (using that additivity holds for $\mathbb{A}^{m,n,i}$). Thus for $\mathbb{Z}^{V,i}_{st}=\mathbb{A}^{i}_{st}-J^{T}(\partial_{i} V)(s,X_{s})Z^{i}_{st}$, we obtain the algebraic relation
\begin{align*}
\mathbb{Z}^{V,i}_{st}=\mathbb{Z}^{V,i}_{sl}+\mathbb{Z}^{V,i}_{lt}+[J^{T}(\partial_{i} V)(l,X_{l})-J^{T}(\partial_{i} V)(s,X_{s})]Z^{i}_{lt}.
\end{align*} 
Furthermore, \eqref{eq:zhoelder1} and \eqref{eq:ahoelder1} yield the bounds on $Z^{i}$ and $\mathbb{Z}^{V,i}$ and hence $(Z^{i},\mathbb{Z}^{V,i})\in\mathcal{R}_{T}(V^{i})$.
We have that $\mathbb{Z}^{m,n,i}_{st}=\mathbb{A}^{m,n,i}_{st}-J^{T}(\partial_{i} V^{m})(s,X_{s})Z^{n,i}_{st}$. Due to bounds \eqref{eq:zhoelder1} and \eqref{eq:ahoelder1} and the convergence in $2.)$, it follows that $(Z^{n,i},\mathbb{Z}^{m,n,i})\in \mathcal{R}_{T}(V^{m,i}):=\mathcal{R}_{T}^{(\alpha+\beta)/\alpha}((J^{T}(\partial_{i}V^{m})(t,X_{t}))_{t}) $. The convergence with respect to $\norm{\cdot}_{\mathcal{R}_{T}(V)}$ follows directly from the convergence in $2.)$.
\end{proof}
\noindent To prove that $(\partial_{i} u(t,X_{t}))_{t}$ is stochastically controlled by $A=(J^{T}(\partial_{i} V)(t,X_{t}))_{t}$, we will need the following auxillary lemma, which is also of independent interest. Its proof only relies on regularity properties of the solution of the Kolmogorov backward equation. The lemma proves the bound \eqref{eq:para-est} on the time-space differences of the paracontrolled remainder. The proof of the lemma can be found in \cref{Appendix A}.
\begin{lemma}\label{lem:parastr-est}
Let $\alpha\in (1,2]$, $\beta\in(\frac{2-2\alpha}{3},0)$, $\calV\in\calX^{\beta,\gamma}$  
and $u\in D_{T}(V)$ (cf.~\eqref{def:good-para-dist}).
Then, we have the following time-space Hölder bound:
\begin{align}\label{eq:para-est}
\MoveEqLeft
\abs{\partial_{i} u(r,x)-\partial_{i} u(s,y)-\nabla u(s,y)\cdot(J^{T}(\partial_{i} V)(r,x)-J^{T}(\partial_{i} V)(s,y))}\nonumber\\&\lesssim \norm{u}_{D_{T}(V)}(1+\norm{\calV}_{\calX^{\beta,\gamma}})[ \abs{r-s}^{(2(\alpha+\beta)-2)/\alpha}+\abs{x-y}^{2(\alpha+\beta)-2}],
\end{align} for $i=1,\dots,d$.  
\end{lemma}

\begin{lemma}\label{lem:nablau-stochcont}
Let $\beta\in ((2-2\alpha)/3,(1-\alpha)/2]$, $\alpha\in (1,2]$. Let $u\in D_{T}(V)$ and let $(X,L,\mathbb{Z}^{V})$ be a weak solution. Then, for every $i=1,\dots,d$, $(\partial_{i} u(t,X_{t}))_{t\in[0,T]}$ is stochastically controlled by $(J^{T}(\partial_{i} V)(t,X_{t}))_{t\in[0,T]}$ with $\varsigma=\varsigma'=(\alpha+\beta-1)/\alpha$.
\end{lemma}
\begin{remark}
In the pure stable noise case, $\alpha\in (1,2)$, the proof of \cref{lem:nablau-stochcont} does not apply for $\beta\in ((1-\alpha)/2,0)$ (i.e. in the Young case), while for $\alpha=2$, the statement of the lemma is also valid in the Young regime. The reason is the integrability issue for the $\alpha$-stable process. Indeed, in the proof we need that $4(\alpha+\beta)-4<\alpha$, that is $\beta<(4-3\alpha)/4$. 
If $\alpha<2$ and $\beta\leqslant (1-\alpha)/2$, then in particular $\beta<(4-3\alpha)/4$. However, the latter doesn't need to be satisfied in the Young regime, unless $\alpha<4/3$.
\end{remark}
\begin{proof}[Proof of \cref{lem:nablau-stochcont}]
By \cref{lem:parastr-est}, we obtain that
\begin{align}\label{eq:exp}
\MoveEqLeft
\abs{R_{st}}:=\abs{\partial_{i} u(t,X_{t})-\partial_{i} u(s,X_{s})-\nabla u_{s}\cdot [J^{T}(\partial_{i} V)(t,X_{t})-J^{T}(\partial_{i} V)(s,X_{s})]}\nonumber\\&\lesssim\norm{u}_{D_{T}}(1+\norm{\calV}_{\calX^{\beta,\gamma}})[\abs{t-s}^{(2(\alpha+\beta)-2)/\alpha}+\abs{X_{t}-X_{s}}^{2(\alpha+\beta)-2}].
\end{align} 
Using the triangle inequality and $u\in\mathcal{L}_{T}^{0,\alpha+\beta}$ as well as the interpolation estimate \cite[Lemma 3.7, (3.11)]{kp-sk}, we can bound
\begin{align*}
\MoveEqLeft
\E[\abs{\partial_{i} u(t,X_{t})-\partial_{i} u(s,X_{s})}^{2}]\\&\lesssim\norm{\partial_{i} u}_{C_{T}^{(\alpha+\beta-1)/\alpha}L^{\infty}}^{2}\abs{t-s}^{2(\alpha+\beta-1)/\alpha}+\norm{\partial_{i} u}_{C_{T}\calC^{\alpha+\beta-1}}^{2}\E[\abs{X_{t}-X_{s}}^{2(\alpha+\beta-1)}]\\&\lesssim_{T} \norm{u}_{D_{T}(V)}^{2}\abs{t-s}^{2(\alpha+\beta-1)/\alpha},
\end{align*} and the same bound is also valid for $\nabla u$. In the last estimate above, we used that due to the bound \eqref{eq:zhoelder1} on $Z$ in $L^{2}$,
\begin{align*}
\MoveEqLeft
\E[\abs{X_{t}-X_{s}}^{2(\alpha+\beta-1)}]\\&\leqslant\E[\abs{Z_{t}-Z_{s}}^{2(\alpha+\beta-1)}]+\E[\abs{L_{t}-L_{s}}^{2(\alpha+\beta-1)}]\\&\leqslant\E[\abs{Z_{t}-Z_{s}}^{2}]^{(\alpha+\beta-1)}+\E[\abs{L_{t}-L_{s}}^{2(\alpha+\beta-1)}]\\&\leqslant\norm{Z}_{\frac{\alpha+\beta}{\alpha},2}^{2(\alpha+\beta-1)}\abs{t-s}^{2(\alpha+\beta)(\alpha+\beta-1)/\alpha}+\abs{t-s}^{2(\alpha+\beta-1)/\alpha}\E[\abs{L_{1}}^{2(\alpha+\beta-1)}]\\&\lesssim_{T}\abs{t-s}^{2(\alpha+\beta-1)/\alpha},
\end{align*} due to Jensen's inequality with $\alpha+\beta-1\in (0,1)$ and due to $(L_{t}-L_{s})\stackrel{d}{=}(t-s)^{1/\alpha}L_{1}$ for $s\leqslant t$, and $2(\alpha+\beta-1)<\alpha$, such that $\E[\abs{L_{1}}^{2(\alpha+\beta-1)}]<\infty$. 
For the remainder, we employ the bound \eqref{eq:exp} to obtain
\begin{align*}
\E[\abs{R_{st}}^{2}]&\lesssim\norm{u}_{D_{T}}^{2}\paren[\big]{\abs{t-s}^{(4(\alpha+\beta)-4)/\alpha}+\E[\abs{X_{t}-X_{s}}^{4(\alpha+\beta)-4}]}\\&\lesssim_{T}\norm{u}_{D_{T}}^{2}\abs{t-s}^{4\varsigma/\alpha},
\end{align*} using an analogoue argument to estimate $\E[\abs{X_{t}-X_{s}}^{4(\alpha+\beta)-4}]$ as above and $\E[\abs{L_{1}}^{4(\alpha+\beta)-4}]<\infty$, since $4(\alpha+\beta)-4<\alpha$ in the case of $\alpha<2$, as $\beta\leqslant (1-\alpha)/2$. In the case of $\alpha=2$, we have all moments on the Brownian motion $B_{1}=L_{1}$. Thus, we obtain that $(\partial_{i}u (t,X_{t}))_t$ is stochastically controlled by $(J^{T}(\partial_{i}V)(t,X_{t}))_t$ with $\varsigma=\varsigma'=(\alpha+\beta-1)/\alpha$ and $(\partial_{i} u)^{\prime}=\nabla u$.
\end{proof}
\begin{proposition}\label{thm:stabilityint}
Let $\beta\in (\frac{2-2\alpha}{3},\frac{1-\alpha}{2}]$ and $V\in\calX^{\beta,\gamma}$. Let $(X,L,\mathbb{Z}^{V})$ be a weak solution.\\ 
Then for all $0\leqslant t\leqslant T$ and for $(Z,\mathbb{Z}^{V})\in\mathcal{R}_{T}(V)$ given by \cref{def:ws} and for $u\in D_{T}(V)$, the rough stochastic integral
\begin{align*}
\int_{0}^{t}\nabla u(s,X_{s})\cdot d(Z,\mathbb{Z}^{V})_{s}:=\sum_{i=1}^{d}\lim_{\abs{\Pi}\to 0}\sum_{r,l\in\Pi}[\partial_{i} u(r,X_{r})Z^{i}_{rl}+\nabla u (r,X_{r})\cdot\mathbb{Z}^{V,i}_{rl}]\in L^{2}(\p),
\end{align*} where the limit ranges over all partitions $\Pi$ of $[0,t]\subset [0,T]$ with mesh-size $\abs{\Pi}:=\max_{r,s\in\Pi}\abs{r-s}\to 0$,
is well-defined and allows for the bound
\begin{align*}
\norm[\bigg]{\int_{s}^{t}\nabla u(r,X_{r})\cdot d(Z,\mathbb{Z}^{V})_{r}}_{L^{2}(\p)}\lesssim_{T} \abs{t-s}^{(\alpha+\beta)/\alpha}\norm{u}_{D_{T}(V)}\norm{(Z,\mathbb{Z}^{V})}_{\mathcal{R}_{T}(V)}.
\end{align*}
In particular, for the sequence $(V^{n})$ from \cref{def:ws} and $(Z^{n},\mathbb{Z}^{m,n}), (Z,\mathbb{Z}^{m,\infty})\in \mathcal{R}_{T}^{\sigma}(V^{m})$ 
and for a sequence $(u^{m})_{m}\subset D_{T}(V^{m})$ with $\norm{u^{m}-u}_{D_{T}(V)}\to 0$,
it follows, that almost surely 
\begin{align*}
\int_{0}^{t}\nabla u^{m}(r,X_{r})\cdot d(Z^{n},\mathbb{Z}^{m,n})_{r}&=\int_{0}^{t}\nabla u^{m}(r,X_{r})\cdot dZ^{n}_{r},
\\\int_{0}^{t}\nabla u^{m}(r,X_{r})\cdot d(Z,\mathbb{Z}^{m,\infty})_{r}&=\int_{0}^{t}\nabla u^{m}(r,X_{r})\cdot dZ_{r},
\end{align*} where the stochastic integrals on the right-hand side are defined as the $L^{2}(\p)$-limit of the classical Riemann sums. And the following convergence in $L^{2}(\p)$, uniformly in $t\in[0,T]$, holds for $n,m\to \infty$, respectively $m\to\infty$,
\begin{align*}
\int_{0}^{t}\nabla u^{m}(r,X_{r})\cdot d(Z^{n},\mathbb{Z}^{m,n})_{r}&\to \int_{0}^{t}\nabla u(r,X_{r})\cdot d(Z,\mathbb{Z}^{V})_{r},
\\\int_{0}^{t}\nabla u^{m}(r,X_{r})\cdot d(Z,\mathbb{Z}^{m,\infty})_{r}&\to \int_{0}^{t}\nabla u(r,X_{r})\cdot d(Z,\mathbb{Z}^{V})_{r}.
\end{align*} 
\end{proposition}

\begin{proof}
Recall that $u^{m}$ is paracontrolled by $J^{T}(V^{m})$ (i.e.~$u^{m}\in D_{T}(V^{m})$), 
\begin{align*}
\mathbb{Z}^{m,n,i}_{st}&=\int_{s}^{t}[J^{T}(\partial_{i}V^{m})(r,X_{r})-J^{T}(\partial_{i}V^{m})(s,X_{s})]dZ^{n,i}_{r}\quad\text{ and }\\ \mathbb{Z}^{m,\infty,i}_{st}&=\int_{s}^{t}[J^{T}(\partial_{i}V^{m})(r,X_{r})-J^{T}(\partial_{i}V^{m})(s,X_{s})]dZ^{i}_{r}.
\end{align*} Then the proof follows from \cref{thm:roughint}, \cref{lem:nablau-stochcont} and \cref{lem:z-rough-integrator}.
\end{proof}

\noindent \cref{thm:stabilityint} extendes the stochastic integral to the stable rough stochastic integral, which finally enables to prove the following theorem.

\begin{theorem}\label{thm:wsimp}
Let $\calV\in\calX^{\beta,\gamma}$ and $\beta\in((2-2\alpha)/3,0)$. Let $(X,L,\mathbb{Z}^{V})$ be a weak solution, starting at $x\in\R^{d}$. Then $X$ solves the martingale problem for the generator $\mathcal{G}^{\calV}$, starting at $x\in\R^{d}$.
\end{theorem}
\begin{proof}
Let $f\in\CTcalC^{\epsilon}$, $\epsilon>2-\alpha$, $u^{T}\in\calC^{3}$ and $u$ be the solution of $\calG^{\calV}u=f$, $u(T,\cdot)=u^{T}$. Let $(X,L,\mathbb{Z}^{V})$ be a weak solution starting at $x$ on the stochastic basis $(\Omega, \F, (\F_{t}), \p)$. 
Then the goal is to show that the process $(M_{t})_{t\in[0,T]}$ with
\begin{align}\label{eq:aim}
M_{t}=u(t,X_{t})-u(0,x)-\int_{0}^{t}f(s,X_{s})ds
\end{align} is a martingale with respect to $(\F^{X}_{t})\subset(\F_{t})$ under $\p$, where $\mathcal{F}^{X}_{t}:=\sigma(X_{s}\mid s\leqslant t)$ is the canonical filtration. Because $X$ is a weak solution, there exists a sequence $(V^{n})\subset C_{T}(C^{\infty}_{b})^{d}$ satisfying $(V^{n},\mathcal{K}(V^{m},V^{n}))\to\calV$ in $\calX^{\beta,\gamma}$ and such that the convergences from $1.),2.)$ from \cref{def:ws} hold. Consider the solution $u^{n}$ of $\mathcal{G}^{V^{n}}u^{n}=f$, $u^{n}(T,\cdot)=u^{T}$, which converges to $u$ in $D_{T}(V)$ by the continuity of the solution map from \cref{thm:singular}. Then as $u^{n}\in C^{1,2}_{b}([0,T]\times\R^{d})$ and $X$ is a Dirichlet process,
we can apply the Itô formula from \cite[Theorem 3.1]{cjml} to $u^{n}(t,X_{t})$ (see also \cref{rem:Dirichlet}), such that for $n\in\N$, 
\begin{align}
\MoveEqLeft
u^{n}(t,X_{t})-u^{n}(0,x)\nonumber\\&=\int_{0}^{t}(\partial_{t}-\La)u^{n}(s,X_{s})ds+\int_{0}^{t}\nabla u^{n}(s,X_{s})\cdot dZ_{s}+M^{n}_{t}\nonumber\\&=\int_{0}^{t}f(s,X_{s})ds+M_{t}^{n}\nonumber\\&\quad +\paren[\bigg]{\int_{0}^{t}\nabla u^{n}(s,X_{s})\cdot dZ_{s}-\int_{0}^{t}\nabla u^{n}(s,X_{s})\cdot V^{n}(s,X_{s})ds}\label{eq:remainder},
\end{align} where we furthermore used the equation for $u^{n}$ and abbreviate the martingale, in the case $\alpha\in (1,2)$, by
\begin{align}\label{eq:m-expression}
M_{t}^{n}:=\int_{0}^{t}\int_{\R^{d}\setminus\{0\}} (u^{n}(s,X_{s-}+y)-u^{n}(s,X_{s-}))\hat{\pi}(ds,dy).
\end{align} 
In the Brownian noise case, $\alpha=2$, we have that $M_{t}^{n}=\int_{0}^{t}\nabla u^{n}(s,X_{s})\cdot dB_{s}$.
It follows that $M^{n}$ is a $(\F_{t})$-martingale   
(cf.~the argument in the proof of \cref{thm:mainthm1}) and it is $(\F_{t}^{X})$-adapted by \eqref{eq:aim}. Thus $M^{n}$ is a $(\F^{X}_{t})$-martingale.
We claim that $(M^{n})_n$ converges in $L^{2}(\p)$, uniformly in $t\in [0,T]$, to a martingale $\tilde{M}$ given by the expression \eqref{eq:m-expression} for $u^{n}$ replaced by $u$. Indeed, this uses the Burkholder-Davis-Gundy-type inequality from \cite[Lemma 8.21]{peszat_zabczyk_2007} and an analogoue argument as in the proof of \cite[Lemma 4.4]{kp}.\\ 
Since $u^{n}\to u$ in $C_{T}L^{\infty}$, we thus obtain that the process $M$ in \eqref{eq:aim} is a martingale with $M=\tilde{M}$, provided that the remainder in \eqref{eq:remainder} vanishes in $L^{2}(\p)$. 
That is, for $Z^{n}=\int_{0}^{\cdot}V^{n}(s,X_{s})ds$, we claim that when $n\to\infty$, 
\begin{align}\label{eq:goal}
\sup_{t\in[0,T]}\norm[\bigg]{\int_{0}^{t}\nabla u^{n}(s,X_{s-})\cdot dZ^{n}_{t}-\int_{0}^{t}\nabla u^{n}(s,X_{s-})\cdot dZ_{s}}_{L^{2}(\p)}\to 0.
\end{align} 
The stochastic integrals are given by the limit of the classical Riemann sums in $L^{2}(\p)$.\\ If $\beta>(1-\alpha)/2$ (i.e.~in the Young case), the convergence follows from the stability of the stochastic sewing integral from \cref{lem:generalws} applied for $f=\partial_{i}u$ and using only that 
$\norm{Z^{n}-Z}_{\frac{\alpha+\beta}{\alpha},2}\to 0$, $\norm{Z^{n}-Z\mid\F}_{\frac{\alpha+\beta}{\alpha},\infty}\to 0$ and $u^{n}\to u$ in $D_{T}(V)$. Thus, in the Young case, the remainder vanishes.\\
If $\beta\in ((2-2\alpha)/3,(1-\alpha)/2]$, then an application of \cref{thm:stabilityint} yields  \eqref{eq:goal}.
\end{proof}

\noindent In what follows, we prove the reverse implication: a martingale solution is a (rough) weak solution in the sense of \cref{def:ws}. The first step is to prove that a martingale solution is of class $\mathcal{K}^{\vartheta}$ for $\vartheta=(2(\alpha+\beta)-1)/\alpha$, cf.~\cref{def:class}. With that, the bound on $\mathbb{Z}^{V}$ in \eqref{eq:ahoelder1} follows after identifying $Z$ and $\mathbb{A}$ with the respective solutions $u^{t},v^{t}$ of the backward PDEs in \eqref{eq:u}, \eqref{eq:v}. The latter will be established in the following proposition.

\begin{proposition}\label{thm:class}
Let $V\in\mathcal{X}^{\beta,\gamma}$ for $\beta\in((2-2\alpha)/3,0)$. Let $X$ be the solution of the martingale problem for the generator $\mathcal{G}^{\calV}$, starting at $x\in\R^{d}$. Then $X$ is of class $\mathcal{K}^{\vartheta}(\cal V)$ (cf.~\cref{def:class}) for $\vartheta=(2(\alpha+\beta)-1)/\alpha$. 
\end{proposition}
\begin{remark}
For a martingale solution $X$, we prove below that 
\begin{align}\label{eq:repr}
\mathbb{Z}^{V,i}_{st}&:=\int_{s}^{t}[J^{T}(\partial_{i} V)(r,X_{r})-J^{T}(\partial_{i} V)(s,X_{s})]dZ^{i}_{s}\nonumber\\&=v^{t,i}(t,X_{t})-v^{t,i}(s,X_{s})-J^{T}(\partial_{i} V)(s,X_{s})(u^{t,i}(t,X_{t})-u^{t,i}(s,X_{s}))+M^{i}_{st}
\end{align} for martingale differences $M^{i}_{st}$ and the solutions $v^{t,i},u^{t,i}$ from \eqref{eq:v}, \eqref{eq:u}. Then, \cref{thm:class} shows $\norm{\mathbb{Z}^{V}\mid\F}_{\vartheta,\infty}<\infty$ for $\vartheta=(2(\alpha+\beta)-1)/\alpha$.
It remains open, if one can prove the bound $\norm{\mathbb{Z}^{V}}_{\vartheta,2}<\infty$, which is a stronger bound than $\norm{\mathbb{Z}^{V}}_{(\alpha+\beta)/\alpha,2}<\infty$ (the latter can be inferred from the time regularity of $v^{t},u^{t}$). This is also the reason why we employ the \underline{stochastic} sewing lemma instead of the classical sewing lemma in $L^{2}(\p)$.\\
It is not straightforward to adjust the proof of \cref{thm:class} to show $\norm{\mathbb{Z}^{V}}_{\vartheta,2}<\infty$. Indeed, the proof exploits the conditional expectation in two ways. On the one hand, it removes the martingale differences, for which the estimates would be more involved. On the other hand, the conditional expectation enables to use the Markov property of $X$ and thus to transform the claim into a question on Schauder and commutator estimates for the semigroup $(T_{s,r})_{s\leqslant r}$ of $X$. In the proof below, we infer those estimates on $(T_{s,r})_{s\leqslant r}$ from regularity properties of the solutions of the generator PDE with \underline{singular} terminal conditons from \cref{thm:singular}.
\end{remark}
\begin{proof}[Proof of \cref{thm:class}]
Let $0\leqslant s<t\leqslant T$. 
By \cref{thm:mainthm} the martingale solution $X$ for $\mathcal{G}^{\calV}$ is a strong Markov process.
Replacing $v^{t,i}$ by the solution $v^{n,i}$ of $\mathcal{G}^{(V,\calV_{2})}v^{n,i}=J^{T}(\partial_{i} V)\cdot V^{n,i}$, $v^{n,i}(t,\cdot)=0$ and $u^{t,i}$ by $u^{n,i}$ with $\mathcal{G}^{(V,\calV_{2})}u^{n,i}=V^{n,i}$, $u^{n,i}(t,\cdot)=0$ and rewriting the conditional expectation in \eqref{eq:class} with the semigroup $(T_{r,l})_{0\leqslant r\leqslant l\leqslant T}$ of the (time-inhomogeneous) Markov process $X$, we obtain
\begin{align}
\hspace{1em}&\hspace{-1em}
\E_{s}[v^{t,i}(t,X_{t})-v^{t,i}(s,X_{s})-J^{T}(\partial_{i} V)(s,X_{s})(u^{t,i}(t,X_{t})-u^{t,i}(s,X_{s}))]\nonumber\\&=\squeeze[1]{\lim_{n\to\infty}\E_{s}[v^{n,i}(t,X_{t})-v^{n,i}(s,X_{s})-J^{T}(\partial_{i} V)(s,X_{s})(u^{n,i}(t,X_{t})-u^{n,i}(s,X_{s}))]}\label{eq:l1}\\&=\lim_{n\to\infty}\E_{s}\bigg[\int_{s}^{t}(J^{T}(\partial_{i} V)(r,X_{r})-J^{T}(\partial_{i} V)(s,X_{s}))\cdot V^{n,i}(r,X_{r})dr\bigg]\label{eq:mp}\\&=\lim_{n\to\infty}\int_{s}^{t}T_{s,r}(J^{T}(\partial_{i} V)\cdot V^{n,i})_{r}(X_{s})-J^{T}(\partial_{i} V)(s,X_{s})T_{s,r}(V^{n,i}_{r})(X_{s})dr\nonumber
\allowdisplaybreaks
\\&=\lim_{n\to\infty}\int_{s}^{t}T_{s,r}(J^{T}(\partial_{i} V)\cdot V^{n,i})_{r}(X_{s})-J^{T}(\partial_{i} V)(r,X_{s})T_{s,r}(V^{n,i}_{r})(X_{s})dr\nonumber\\&\qquad +\lim_{n\to\infty}\int_{s}^{t}(J^{T}(\partial_{i} V)(s,X_{s})-J^{T}(\partial_{i} V)(r,X_{s}))T_{s,r}V^{n,i}_{r} (X_{s})dr\nonumber
\\&=\int_{s}^{t}w^{r,i}_{s}(X_{s})-J^{T}(\partial_{i} V)(r,X_{s})y^{r,i}_{s}(X_{s})dr\label{eq:t1}\\&\qquad +\int_{s}^{t}(J^{T}(\partial_{i} V)(s,X_{s})-J^{T}(\partial_{i} V)(r,X_{s}))y^{r,i}_{s} (X_{s})dr\label{eq:t2}.
\end{align} 
The convergence above is in $L^{\infty}(\p)$, uniformly in $s,t$. The convergence in \eqref{eq:l1} follows from $\sup_{t\in[0,T]}\norm{v^{n,i}- v^{t,i}}_{C_{t}L^{\infty}_{\R^{d}}}\to 0$ and $\sup_{t\in[0,T]}\norm{u^{n,i}- u^{t,i}}_{C_{t}L^{\infty}}\to 0$ by \cite[Corollary 4.14]{kp-sk} (taking $f^{t,n}=V^{n,i}_{\mid [0,t]}$ and $y^{n,t}=u^{n,i}(t,\cdot)=0$).
For the equality \eqref{eq:mp}, we used that $X$ solves the martingale problem for $\mathcal{G}^{\calV}$, such that 
\begin{align*}
u^{n,i}(t,X_{t})-u^{n,i}(s,X_{s})=\int_{s}^{t}V^{n,i}(r,X_{r})dr+M_{r,t}^{n,i},
\end{align*} for martingale differences $M^{n,i}_{r,t}$, that vanishes after taking the conditional expectation, analogously for $v^{n,i}$. Moreover, $y^{r}, w^{r}$ are defined as follows. Let
\begin{align}\label{eq:equations-yw}
T_{s,r}V_{r}^{n,i}=:y^{r,n,i}_{s},\quad T_{s,r}(J^{T}(\partial_{i}V)\cdot V_{r}^{n,i})=:w^{r,n,i}_{s},\quad s\in[0,r].
\end{align}  
Then we have that $y^{r,n,i}$ solves $\mathcal{G}^{\calV}y^{r,n,i}=0$ with singular terminal condition $y^{n,r}_{r}=V_{r}^{n,i}$ at time $r$ and $w^{r,n,i}=(w^{r,n,i,j})_{j=1,\dots,d}$ solves  $\mathcal{G}^{\calV}w^{n,r,i,j}=0$ with $w^{r,n,i,j}_{r}=J^{T}(\partial_{i}V^{j})_{r}\cdot V_{r}^{n,i}$. By the continuity of the PDE solution map (cf.~\cref{thm:singular}), we conclude that 
\begin{align*}
&(y^{r,n,i},y^{r,n,i,\sharp})\to (y^{r,i},y^{r,i,\sharp}),\\& (w^{r,n,i,j},w^{r,n,i,j,\sharp})\to (w^{r,i,j},w^{r,i,j,\sharp})\in\mathcal{L}_{r}^{\gamma,\alpha+\beta}\times\mathcal{L}_{r}^{\gamma,2(\alpha+\beta)-1},
\end{align*} for $n\to\infty$, where $y^{r,i},w^{r,i}$ solve $\mathcal{G}^{\calV}y^{r,i}=0$, $\mathcal{G}^{\calV}w^{r,i,j}=0$ with terminal conditions $y^{r}_{r}=V^{i}_{r}$, respectively $w^{r,i,j}_{r}=J^{T}(\partial_{i}V^{j})_{r}\cdot V^{i}_{r}$, and $\gamma\in (0,1)$. The convergence is uniform in $r\in[0,T]$ by \cite[Corollary 4.14]{kp-sk}. In particular, it follows that $y^{r,n,i}\to y^{r,i}$ and $w^{r,n,i,j}\to w^{r,i,j}$ in $\mathcal{M}_{r}^{\gamma}L^{\infty}$, uniformly in $r\in[0,T]$, which implies the convergence of the integrals \eqref{eq:t1} and \eqref{eq:t2} in $L^{\infty}(\p)$, uniformly in $0\leqslant s<t\leqslant T$.\\
Our goal is now to estimate both integrals  \eqref{eq:t1} and \eqref{eq:t2} in $L^{\infty}(\p)$ by $\abs{t-s}^{\vartheta}$.\\
To estimate the integral \eqref{eq:t1}, we apply again \cite[Corollary 4.14]{kp-sk} for the solutions $(y^{r,i})$, $(w^{r,i,j})$.
The uniform bound from the corollary together with the interpolation bound (3.13) from \cite[Lemma 3.7]{kp-sk} for $\tilde{\theta}=\alpha+\beta\gamma^{-1}$ (note $\beta(1-\gamma^{-1})>0$) , $\theta=\alpha+\beta$ and the embedding $\calC^{(1-\gamma')\alpha}\hookrightarrow L^{\infty}$ yield that 
\begin{align*}
\sup_{r\in [0,T]}\norm{y^{r,i}}_{\mathcal{M}_{r}^{-\beta/\alpha}L^{\infty}}\lesssim
\sup_{r\in [0,T]}\norm{y^{r,i}}_{\mathcal{M}_{r}^{-\beta/\alpha}\calC^{\beta(1-\gamma^{-1})}}\lesssim
\sup_{r\in [0,T]}\norm{y^{r,i}}_{\mathcal{L}_{r}^{\gamma,\alpha+\beta}}\lesssim_{T,\calV} 1.
\end{align*}
Thus, we obtain  
\begin{align*}
\MoveEqLeft
\norm[\bigg]{\int_{s}^{t}(J^{T}(\partial_{i} V)_{s}-J^{T}(\partial_{i} V)_{r})y^{r,i}_{s}dr}_{L^{\infty}}\\&\lesssim
\sup_{r\in [0,T]}\norm{y^{r,i}}_{\mathcal{M}_{r}^{-\beta/\alpha}L^{\infty}}\norm{J^{T}(\partial_{i} V)}_{C_{T}^{(\alpha+\beta-1)/\alpha}L^{\infty}_{\R^{d}}}\int_{s}^{t}\abs{r-s}^{(\alpha+2\beta-1)/\alpha}dr\\&\lesssim_{T}\norm{V}_{C_{T}\calC^{\beta}_{\R^{d}}}\abs{t-s}^{(2\alpha+2\beta-1)/\alpha}\\&=\norm{V}_{C_{T}\calC^{\beta}_{\R^{d}}}\abs{t-s}^{\vartheta},
\end{align*} using that $2\alpha+2\beta-1>0$ and $J^{T}(\partial_{i} V)\in C_{T}^{(\alpha+\beta-1)/\alpha}L^{\infty}_{\R^{d}}$.\\
To estimate the term in \eqref{eq:t2}, we use cancellations between the solution $w^{r,i}$ and $y^{r,i}$. For the argument, we need to distinguish between the cases $\beta\in ((1-\alpha)/2,0)$ and $\beta\in ((2-2\alpha)/3,(1-\alpha)/2]$.
 
\noindent First, we consider the Young case, $\beta\in ((1-\alpha)/2,0)$. Then, we have that $\alpha+2\beta-1>0$. 
We can write the difference of the solutions as follows:
\begin{align}
\MoveEqLeft
w^{r,i}_{s}-J^{T}(\partial_{i}V)_{r} y^{r,i}_{s}\nonumber\\&=[P_{r-s}(J^{T}(\partial_{i}V)_{r}\cdot V^{i}_{r})-J^{T}(\partial_{i}V)_{r}\cdot P_{r-s}V_{r}^{i}]\nonumber\\&\qquad+J^{r}_{s}(\nabla w^{r,i}\cdot V) - J^{r}_{s}(\nabla y^{r,i}\cdot V)
\nonumber\\& =(P_{r-s}-\operatorname{Id})(J^{T}(\partial_{i}V)_{r}\reso V^{i}_{r}+J^{T}(\partial_{i}V)_{r}\arap V^{i}_{r})\label{eq:y1}
\\&\qquad -J^{T}(\partial_{i}V)_{r}\reso (P_{r-s}-\operatorname{Id})V_{r}^{i}+J^{T}(\partial_{i}V)_{r}\arap (P_{r-s}-\operatorname{Id})V_{r}^{i}\label{eq:y2}
\\&\qquad +[P_{r-s}(J^{T}(\partial_{i}V)_{r}\para V^{i}_{r})-J^{T}(\partial_{i}V)_{r}\para P_{r-s}V_{r}^{i}]\label{eq:y3}
\\&\qquad+J^{r}_{s}(\nabla w^{r,i}\cdot V) - J^{T}(\partial_{i}V)_{r}J^{r}_{s}(\nabla y^{r,i}\cdot V).\label{eq:y4}
\end{align} 
Above and below we use the notation $J^{t}_{r}(v):=J^{t}(v)_{r}=J^{t}(v)(r,\cdot)$ for $r\leqslant t$.
The term in \eqref{eq:y1}, we estimate with the semigroup estimate (2.8) from \cite[Lemma 2.5]{kp-sk}, using that $\alpha+2\beta-1\in (0,1)$,
\begin{align*}
\MoveEqLeft
\norm{(P_{r-s}-\operatorname{Id})(J^{T}(\partial_{i}V)_{r}\reso V^{i}_{r}+J^{T}(\partial_{i}V)_{r}\arap V^{i}_{r})}_{L^{\infty}_{\R^{d}}}\\&\lesssim\abs{r-s}^{(\alpha+2\beta-1)/\alpha}\norm{J^{T}(\partial_{i}V)_{r}\reso V^{i}_{r}+J^{T}(\partial_{i}V)_{r}\arap V^{i}_{r})}_{\calC^{\alpha+2\beta-1}_{\R^{d}}} \\&\lesssim\abs{r-s}^{(\alpha+2\beta-1)/\alpha}\norm{V}_{C_{T}\calC^{\beta}_{\R^{d}}}^{2}.
\end{align*} 
The terms in \eqref{eq:y2}, we also estimate with the semigroup and Schauder estimates using that $V\in C_{T}(\calC^{\beta+(1-\gamma')\alpha})^d$ for $\gamma'\in [(1-\beta)/\alpha,1)$ and again that $\alpha+2\beta-1\in (0,1]$ and the embedding $\calC^{(1-\gamma')\alpha}\hookrightarrow L^{\infty}$,
\begin{align*}
\MoveEqLeft
\norm{J^{T}(\partial_{i}V)_{r}\arap (P_{r-s}-\operatorname{Id})V_{r}^{i}}_{L^{\infty}}\\&\lesssim \norm{J^{T}(\partial_{i}V)_{r}\arap (P_{r-s}-\operatorname{Id})V_{r}^{i}}_{\calC^{(1-\gamma')\alpha}_{\R^{d}}}\\&\lesssim \norm{J^{T}(\partial_{i}V)_{r}}_{\calC^{(2-\gamma')\alpha+\beta-1}_{\R^{d}}} \norm{(P_{r-s}-\operatorname{Id})V_{r}^{i}}_{\calC^{-(\alpha+\beta-1)}}\\&\lesssim\norm{V}_{C_{T}(\calC^{\beta+(1-\gamma')\alpha})^d}\abs{r-s}^{(\alpha+2\beta-1)/\alpha}\norm{V}_{C_{T}(\calC^{\beta})^d}.
\end{align*}
The argument for the $\reso$-product is analogous.
Moreover, the term in \eqref{eq:y3} equals the semigroup commutator from \cite[Lemma 2.7]{kp-sk} and using that $\alpha+2\beta-1<\alpha$, we obtain
\begin{align*}
\MoveEqLeft
\norm{[P_{r-s}(J^{T}(\partial_{i}V)_{r}\para V^{i}_{r})-J^{T}(\partial_{i}V)_{r}\para P_{r-s}V_{r}^{i}]}_{L^{\infty}}\\&\lesssim\norm{[P_{r-s}(J^{T}(\partial_{i}V)_{r}\para V^{i}_{r})-J^{T}(\partial_{i}V)_{r}\para P_{r-s}V_{r}^{i}]}_{\calC^{(1-\gamma')\alpha}_{\R^{d}}}\\&\lesssim \abs{r-s}^{(\alpha+2\beta-1)/\alpha}\norm{V}_{C_{T}(\calC^{\beta+(1-\gamma')\alpha})^d}\norm{V}_{C_{T}(\calC^{\beta})^d}.
\end{align*}
For the last term \eqref{eq:y4}, we give the argument for the term involving $y^{r,i}$ and the argument for the term with $w^{r,i}$ is analogue. We decompose
\begin{align*}
J^{r}_{s}(\nabla y^{r,i}\cdot V)&=J^{r}_{s}(\nabla y^{r,i}\reso V+\nabla y^{r,i}\arap V)+\nabla y^{r,i}_s\para J^{r}_{s}(V)\\&\qquad+[J^{r}_{s}(\nabla y^{r,i}\para V)-\nabla y^{r,i}_s\para J^{r}_{s}(V)].
\end{align*}
Due to the interpolation bound (3.13) from \cite[Lemma 3.7]{kp-sk}, we obtain that 
\begin{align*}
\nabla y^{r,i}\in\mathcal{M}_{r}^{(1-\beta)/\alpha}\calC^{(1-\beta)(\frac{\alpha}{\gamma'}-1)}_{\R^{d}}\hookrightarrow\mathcal{M}_{r}^{(1-\beta)/\alpha}L^{\infty}_{\R^{d}}
\end{align*} as $\gamma'<1\leqslant\alpha$. This yields together with \cite[Lemma 2.5]{kp-sk} that
\begin{align}\label{eq:para-est-1}
\norm{\nabla y^{r,i}_s\para J^{r}_{s}(V)}_{L^{\infty}_{\R^{d}}}&\lesssim\norm{\nabla y^{r,i}}_{\mathcal{M}_{r}^{(1-\beta)/\alpha}L^{\infty}_{\R^{d}}}\abs{r-s}^{(\beta-1)/\alpha}\norm{J^{r}_{s}(V)}_{L^{\infty}_{\R^{d}}}
\nonumber\\&\lesssim\norm{\nabla y^{r,i}}_{\mathcal{M}_{r}^{(1-\beta)/\alpha}L^{\infty}_{\R^{d}}}\abs{r-s}^{(\beta-1)/\alpha}\norm{J^{r}_{s}(V)}_{\calC^{(1-\gamma')\alpha}_{\R^{d}}}
\nonumber\\&\lesssim\abs{r-s}^{(\alpha+2\beta-1)/\alpha}\norm{\nabla y^{r,i}}_{\mathcal{M}_{r}^{(1-\beta)/\alpha}L^{\infty}_{\R^{d}}}\norm{V}_{C_{T}\calC^{\beta+(1-\gamma')\alpha}_{\R^{d}}}.
\end{align}
\cite[Corollary 3.2]{kp-sk} furthermore yields that $J^{r}(\nabla y^{r,i}\reso V)\in\mathcal{L}_{r}^{\gamma,2\beta+\alpha-1+(2-\gamma')\alpha}$ using that $V\in C_{T}(\calC^{\beta+(1-\gamma')\alpha})^d$, which implies by the interpolation bound \cite[Lemma 3.7, (3.11)]{kp-sk} and $J^{r}_{r}(v)=0$ that $J^{r}(\nabla y^{r,i}\reso V)\in C_{r}^{(2\beta+\alpha-1+(2-\gamma')\alpha-\gamma\alpha)/\alpha}L^{\infty}$ (and analogously for the $\arap$-product). Thus, as $J^{r}_{r}(v)=0$ and $2-\gamma'-\gamma >0$ as $\gamma\in (0,1)$, $\gamma'\in (0,1)$ , we obtain that
\begin{align*}
\norm{J^{r}_{s}(\nabla y^{r,i}\reso V+\nabla y^{r,i}\arap V)}_{L^{\infty}}\lesssim\abs{r-s}^{(\alpha+2\beta-1)/\alpha}\norm{y^{r,i}}_{\mathcal{L}_{r}^{\gamma,\alpha+\beta}}\norm{V}_{C_{T}(\calC^{\beta+(1-\gamma')\alpha})^d}.
\end{align*}
Due to the commutator \cite[Lemma 3.4]{kp-sk}, we obtain that \begin{align*}
[J^{r}(\nabla y^{r,i}\para V)-\nabla y^{r,i}\para J^{r}(V)]\in\mathcal{L}_{r}^{\gamma,2\beta+\alpha-1+(2-\gamma')\alpha}
\end{align*} and by an interpolation argument as above we thus find
\begin{align*}
\norm{J^{r}_{s}(\nabla y^{r,i}\para V)-\nabla y^{r,i}_s\para J^{r}_{s}(V)}_{L^{\infty}}\lesssim\abs{r-s}^{(\alpha+2\beta-1)/\alpha}\norm{y^{r,i}}_{\mathcal{L}_{r}^{\gamma,\alpha+\beta}}\norm{V}_{C_{T}(\calC^{\beta+(1-\gamma')\alpha})^d}.
\end{align*}
Together, we obtain for the term \eqref{eq:y4}:
\begin{align*} 
\MoveEqLeft
\norm{J^{T}(\partial_{i}V)_{r}J^{r}_{s}(\nabla y^{r,i}\cdot V)}_{L^{\infty}}
\lesssim\abs{r-s}^{(\alpha+2\beta-1)/\alpha}\norm{y^{r}}_{\mathcal{L}_{r}^{\gamma,\alpha+\beta}}\norm{V}_{C_{T}(\calC^{\beta+(1-\gamma')\alpha})^d}^{2}.
\end{align*}
Together with the uniform bound of the norm of $y^{r},w^{r}$ from \cite[Corollary 4.14]{kp-sk} 
the above estimates yield
\begin{align*}
\int_{s}^{t}\norm{w^{r,i}_{s}-J^{T}(\partial_{i}V)_{r} y^{r,i}_{s}}_{L^{\infty}_{\R^{d}}}dr\lesssim\abs{t-s}^{(2\alpha+2\beta-1)/\alpha},
\end{align*}
which, together with the estimate for the integral \eqref{eq:t1} yields the claim in the Young case.

\noindent If $\beta\in ((2-2\alpha)/3,(1-\alpha)/2]$, we need to estimate the integral \eqref{eq:t2} differently. In order to rewrite the difference of the solutions $w^{r,i}_{s}-J^{T}(\partial_{i}V)y^{r,i}_{s}$ in a way so that we can prove the claimed bound, we first specify the paracontrolled structure of $y^{r},w^{r}$.\\
The terminal condition of $w^{r,i}$ has the paracontrolled structure
\begin{align*}
w_{r}^{r,i,j}=J^{T}(\partial_{i} V^{j})_{r}\cdot V^{i}_{r}&=J^{T}(\partial_{i} V^{j})_{r}\reso V^{i}_{r}+J^{T}(\partial_{i} V^{j})_{r}\arap V^{i}_{r}+J^{T}(\partial_{i} V^{j})_{r}\para V^{i}_{r}\\&=:w^{R,\sharp,j}+w^{R,\prime,j}\para V_{r}
\end{align*} with $w^{R,\sharp,j}\in\calC^{(2-\gamma')\alpha+2\beta-1}$ (due to $V\in C_{T}\calC^{\beta+(1-\gamma')\alpha}_{\R^{d}}$) and $w^{R,\prime,j}=J^{T}(\partial_{i} V^{j})_{r}e_{i}\in\calC^{\alpha+\beta-1}_{\R^{d}}$ ($e_{i}$ denoting the $i$-th unit vector). By \cref{thm:singular}, the solution has the following paracontrolled structure
\begin{align}\label{eq:weq}
w^{r,i,j}_{s}
=w^{r,\sharp,i,j}_{s}+\nabla w^{r,i,j}_{s}\para J^{r}(V)_{s}+J^{T}(\partial_{i} V^{j})_{r}e_{i}\para P_{r-s}V_{r}
\end{align} with $w^{r,\sharp,i,j}\in\mathcal{L}_{r}^{\gamma,2(\alpha+\beta)-1}$. 
Then, again by the interpolation estimate \cite[Lemma 3.7, (3.13)]{kp-sk} applied for $\theta=\tilde{\theta}=2\alpha+2\beta-1\in (0,\alpha)$, since $\beta\leqslant (1-\alpha)/2$, and for $\theta=\tilde{\theta}=\alpha+\beta-1\in (0,\alpha)$ 
we obtain together with \cite[Corollary 4.14]{kp-sk} the uniform bound
\begin{align}\label{eq:uniform-w}
\MoveEqLeft
\sup_{r\in [0,T]}[\norm{w^{r,\sharp,i}}_{\mathcal{M}_{r}^{(1-\alpha-2\beta)/\alpha}L^{\infty}_{\R^{d}}}+\norm{\nabla w^{r,i}}_{\mathcal{M}_{r}^{(1-\beta)/\alpha}L^{\infty}_{\R^{d\times d}}}]\nonumber\\&\lesssim
\sup_{r\in [0,T]}[\norm{w^{r,\sharp,i}}_{(\mathcal{L}_{r}^{\gamma,2(\alpha+\beta)-1})^{d}}+\norm{w^{r,i}}_{(\mathcal{L}_{r}^{\gamma,\alpha+\beta})^{d}}]\lesssim_{T,\calV} 1.
\end{align} 
Thus, using an estimate as for \eqref{eq:para-est-1}, we can estimate 
\begin{align}\label{eq:w-bound}
\norm{\nabla w^{r,i,j}_{s}\para J^{r}(V)_{s}}_{L^{\infty}}
\lesssim\abs{r-s}^{(\alpha+2\beta-1)/\alpha}\norm{V}_{C_{T}\calC^{\beta+(1-\gamma')\alpha}_{\R^{d}}}\sup_{r\in[0,T]}\norm{\nabla w^{r,i}}_{\mathcal{M}_{r}^{(1-\beta)/\alpha}L^{\infty}_{\R^{d\times d}}}.
\end{align} 
Furthermore, \eqref{eq:uniform-w} implies that
\begin{align}\label{eq:w-sharp-bound}
\int_{s}^{t}\norm{w^{r,\sharp,i,j}_{s}}_{L^{\infty}}dr&\lesssim \sup_{r\in [0,T]}\norm{w^{r,\sharp,i}}_{\mathcal{M}_{r}^{(1-\alpha-2\beta)/\alpha}L^{\infty}_{\R^{d}}}\,\int_{s}^{t}\abs{r-s}^{(2\beta+\alpha-1)/\alpha}dr\nonumber\\&\lesssim \abs{t-s}^{(2\alpha+2\beta-1)/\alpha}\sup_{r\in [0,T]}\norm{w^{r,\sharp,i}}_{\mathcal{M}_{r}^{(1-\alpha-2\beta)/\alpha}L^{\infty}_{\R^{d}}},
\end{align} since $(2\alpha+2\beta-1)/\alpha>0$.\\
Moreover, for the solution $y^{r,i}$ from above, we have the paracontrolled structure
\begin{align*}
y^{r,i}_{s}=y_{s}^{r,\sharp,i}+e_{i}\para P_{r-s}V_{r}+\nabla y^{r,i}_{s}\para J^{r}(V)_{s}.
\end{align*} 
The bounds \eqref{eq:w-bound}, \eqref{eq:w-sharp-bound} hold analogously for $\nabla w^{i,j}$, $w^{r,\sharp,i,j}$ replaced by $\nabla y^{r,i}$, $y^{r,i,\sharp}$.\\
Furthermore, the interpolation estimate \cite[Lemma 3.7, (3.13)]{kp-sk} (again for $\theta=\alpha+\beta$, $\tilde{\theta}=2\alpha+2\beta-1\in (0,\alpha)$ as $\beta\leqslant (1-\alpha)/2$)   
yields that 
\begin{align*}
\sup_{r\in [0,T]}[\norm{y^{r,i}}_{\mathcal{M}_{r}^{(1-\alpha-2\beta)/\alpha}\calC^{-(\alpha+\beta-1)}}+\norm{y^{r,\sharp,i}}_{\mathcal{M}_{r}^{(1-\alpha-2\beta)/\alpha}L^{\infty}}]\lesssim_{T,\calV} 1.
\end{align*}
With the latter bound, we can estimate, 
\begin{align*}
\MoveEqLeft
\norm{J^{T}(\partial_{i} V)_{r}\reso y^{r,i}_{s}+J^{T}(\partial_{i}V)_{r}\arap y^{r,i}_{s}}_{L^{\infty}_{\R^{d}}}\\&\lesssim\norm{J^{T}(\partial_{i} V)_{r}\reso y^{r,i}_{s}+J^{T}(\partial_{i}V)_{r}\arap y^{r,i}_{s}}_{\calC^{(1-\gamma)\alpha}_{\R^{d}}}\\&\lesssim\norm{J^{T}(\partial_{i} V)_{r}}_{(\calC^{\alpha+\beta+(1-\gamma')\alpha-1})^d}\norm{y^{r,i}_{s}}_{\calC^{-(\alpha+\beta-1)}}\\&\lesssim \abs{r-s}^{(2\beta+\alpha-1)/\alpha}\norm{y^{r,i}}_{\mathcal{M}_{r}^{(1-\alpha-2\beta)/\alpha}\calC^{-(\alpha+\beta-1)}}\norm{V}_{C_{T}(\calC^{\beta+(1-\gamma')\alpha})^d}.
\end{align*}
Using the paracontrolled structure of the solutions $w^{r,i,j},y^{r,i}$, we obtain for $j=1,\dots,d$,
\begin{align}\label{eq:difference}
\MoveEqLeft
w^{r,i,j}_{s}-J^{T}(\partial_{i} V^{j})_{r}\para y^{r,i}_{s}\nonumber\\&=w^{r,\sharp,i,j}_{s}+J^{T}(\partial_{i} V^{j})_{r}\para y^{r,\sharp,i}_{s}+ J^{T}(\partial_{i} V^{j})_{r}\para(\nabla y^{r,i}\para J^{r}(V))_{s} \nonumber\\&\qquad+ (\nabla w^{r,i,j}\para J^{r}(V))_{s}.
\end{align}  
Finally, we can estimate the integral \eqref{eq:t2} using the bounds derived above and \eqref{eq:difference}:
\begin{align}\label{eq:class-estimate}
\MoveEqLeft
\norm[\bigg]{\paren[\bigg]{\int_{s}^{t}(w^{r,i,j}_{s}-J^{T}(\partial_{i} V^{j})_{r}\cdot y^{r,i}_{s})dr}_{j}}_{L^{\infty}_{\R^{d}}}\nonumber\\&\leqslant\norm[\bigg]{\paren[\bigg]{\int_{s}^{t}(w^{r,i,j}_{s}-J^{T}(\partial_{i} V^{j})_{r}\para y^{r,i}_{s})dr}_{j}}_{L^{\infty}_{\R^{d}}}\nonumber
\\&\qquad\qquad+\norm[\bigg]{\paren[\bigg]{\int_{s}^{t}[J^{T}(\partial_{i} V^j)_{r}\reso y^{r,i}_{s}+J^{T}(\partial_{i}V^j)_{r}\arap y^{r,i}_{s}]dr}_{j}}_{L^{\infty}_{\R^{d}}}\nonumber
\allowdisplaybreaks
\\&\leqslant \norm[\bigg]{\int_{s}^{t}w^{r,\sharp,i}_{s}dr}_{L^{\infty}_{\R^{d}}}+\norm{J^{T}(\partial_{i}V)}_{L^{\infty}_{\R^{d}}}\norm[\bigg]{\int_{s}^{t}y^{r,\sharp,i}_{s}dr}_{L^{\infty}}
\nonumber\\&\qquad\qquad+\norm{J^{T}(\partial_{i}V)}_{L^{\infty}_{\R^{d}}}\norm[\bigg]{\int_{s}^{t}(\nabla y^{r,i}\para J^{r}(V))_{s}dr}_{L^{\infty}} \nonumber\\&\qquad\qquad+ \norm[\bigg]{\paren[\bigg]{\int_{s}^{t}(\nabla w^{r,i,j}\para J^{r}(V))_{s}dr}_j}_{L^{\infty}_{\R^{d}}}\nonumber\\&\qquad\qquad+\norm[\bigg]{\paren[\bigg]{\int_{s}^{t}[J^{T}(\partial_{i} V^j)_{r}\reso y^{r,i}_{s}+J^{T}(\partial_{i} V^j)_{r}\arap y^{r,i}_{s}]dr}_{j}}_{L^{\infty}_{\R^{d}}}
\nonumber\\&\lesssim \abs{t-s}^{\vartheta}\sup_{r\in[0,T]}\Big[\norm{w^{r,i}}_{(\mathcal{L}_{r}^{\gamma,\alpha+\beta})^{d}}+\norm{w^{r,i,\sharp}}_{(\mathcal{L}_{r}^{\gamma,2(\alpha+\beta)-1})^{d}}\nonumber\\&\qquad\qquad\qquad\qquad+\norm{\calV}_{\calX^{\beta,\gamma'}}(\norm{y^{r,i}}_{\mathcal{L}_{r}^{\gamma,\alpha+\beta}}+\norm{y^{r,i,\sharp}}_{\mathcal{L}_{r}^{\gamma,2(\alpha+\beta)-1}})\Big].
\end{align} In the last estimate, we used also that $\vartheta=(2\alpha+2\beta-1)/\alpha>0$. Together with the bound for \eqref{eq:t1}, this yields the claim in the rough case. 
\end{proof}
\begin{theorem}\label{thm:mpiws}
Let $\calV\in\calX^{\beta,\gamma}$ for $\beta\in((2-2\alpha)/3,0)$. 
Let $X$ be the solution of the martingale problem for the generator $\mathcal{G}^{\calV}$, starting at $x\in\R^{d}$. \\
Then, there exists a stochastic basis $(\Omega,\F,(\F_t),\p)$ and an $\alpha$-stable symmtric non-degenerate $(\F_{t})$-Lévy process $L$, such that $(X,L,\mathbb{Z}^{V})$ is a weak solution starting at $x\in\R^{d}$. Furthermore, the following representations for $Z=X-x-L$ and $\mathbb{Z}^{V}$ follow for $(s,t)\in\Delta_{T}$
\begin{align*}
&Z_{st}=\E_{s}[u^{t}(t,X_{t})-u^{t}(s,X_{s})],\\& \mathbb{Z}^{V,i}_{st}=\E_{s}[v^{t,i}(t,X_{t})-v^{t,i}(s,X_{s})-J^{T}(\partial_{i} V)(s,X_{s})(u^{t,i}(t,X_{t})-u^{t,i}(s,X_{s}))]
\end{align*} almost surely, for the solutions $v^{t,i},u^{t,i}$ from \eqref{eq:v}, \eqref{eq:u}. 
\end{theorem}
\begin{proof}
Let $(V^{n})\subset C_{T}C^{\infty}_{b}(\R^{d},\R^{d})$ with $(V^{n},\mathcal{K}(V^{n},V^{m}))\to\calV$ in $\calX^{\beta,\gamma}$ for $n,m\to\infty$ (existence by \cref{def:enhanced-dist}). 
Let $X$ be the solution of the $(\mathcal{G}^{\calV},\delta_{x})$-martingale problem. Let, as in the proof of \cref{thm:mainthm1}, $X^{n}$ be the strong solution of
\begin{align*}
X^{n}=x+\int_{0}^{\cdot}V^{n}(s,X_{s}^{n})ds+L=:x+Z^{n,n}+L
\end{align*} and let for $l,m,n\in\N$, 
\begin{align*}
Z^{m,n}_{t}&:=\int_{0}^{t}V^{m}(s,X_{s}^{n})ds,\quad t\in[0,T]\\\mathbb{Z}^{l,m,n}_{st}&:=\paren[\bigg]{\int_{s}^{t}[J^{T}(\partial_{i}V^{j,l})(r,X_{r}^{n})-J^{T}(\partial_{i}V^{j,l})(s,X_{s}^{n})]dZ_{r}^{i,m,n}}_{i,j},\quad (s,t)\in\Delta_{T}.
\end{align*} 
\cref{thm:mainthm1} proves the distributional convergence $(X^{n}, Z^{n,n})\Rightarrow (X,Z)$, where $Z$ is a continuous process.
Let $u^{t,m,n}=(u^{t,m,n,i})_{i=1,\dots,d}$ and $v^{t,l,m,n}=(v^{t,l,m,n,i,j})_{i,j=1,\dots,d}$ solve 
\begin{align*}
\mathcal{G}^{V^{n}}u^{t,m,n,i}=V^{m,i}, \quad \mathcal{G}^{V^{n}}v^{t,l,m,n,i,j}=J^{T}(\partial_{i}V^{l,j})\cdot V^{m,i}
\end{align*} with zero terminal condition at time $t\in [0,T]$. By convergence of the mixed resonant products, i.e.~$(V^{n},\mathcal{K}(V^{n},V^{m}))\to\calV$ in $\calX^{\beta,\gamma}$, and by continuity of the PDE solution map (\cref{thm:singular}), we obtain that $u^{t,m,n,i}\to u^{t,i}$  and $v^{t,l,m,n,i,j}\to v^{t,i,j}$ in $D_{t}$, where $u^{t},v^{t}$ solve the PDEs \eqref{eq:u}, \eqref{eq:v}.
An application of Itô's formula (cf.~in the proof of \cref{thm:mainthm1}) then yields the representations
\begin{align}\label{eq:Zmn-mart}
Z^{m,n}_{st}=u^{t,m,n}(t,X_{t}^{n})-u^{t,m,n}(s,X_{s}^{n})+M_{st}^{u^{t,m,n}}
\end{align} for the martingale differences $M^{u^{t,m,n}}_{st}$ defined as 
\begin{align*}
M^{u^{t,m,n}}_{st}:=\int_{s}^{t}\int_{\R^{d}\setminus\{0\}}[u^{t,m,n}(r,X_{r-}^{n}+y)-u^{t,m,n}(r,X_{r-}^{n})]\hat{\pi}(dr,dy)
\end{align*} for the compensated Poisson random measure $\hat{\pi}$, respectively $M_{st}^{u^{t,m,n}}:=\int_{s}^{t}\nabla u^{t,m,n}(r,X_{r}^{n})dB_{r}$ in the case $\alpha=2$. Moreover we have that
\begin{align}\label{eq:ZZmn-mart}
\mathbb{Z}^{l,m,n}_{st}(i,j)&=v^{t,l,m,n,i,j}(t,X_{t}^{n})-v^{t,l,m,n,i,j}(s,X_{s}^{n})\nonumber\\&\qquad-J^{T}(\partial_{i}V^{j,l})(s,X_{s}^{n})\paren[\big]{u^{t,m,n,i}(t,X_{t}^{n})-u^{t,m,n,i}(s,X_{s}^{n})}\nonumber\\&\qquad+M_{st}^{l,m,n}(i,j)
\end{align} for the martingale differences defined by 
\begin{align}\label{eq:mart-diff}
M^{l,m,n}_{st}(i,j):=M_{st}^{v^{t,l,m,n}}-J^{T}(\partial_{i}V^{j,l})(s,X_{s}^{n})M_{st}^{u^{t,m,n}}.
\end{align} 
Notice that $\E_{s}[M^{l,m,n}_{st}]=0$.
By convergence of the PDE solutions, we obtain that for all large enough $l,m,n$,
\begin{align}\label{eq:lmn-bound}
\sup_{t\in[0,T]}[\norm{u^{t,m,n}}_{D_{t}^d}+\norm{v^{t,l,m,n}}_{D_{t}^{d\times d}}]\leqslant 2 \sup_{t\in[0,T]}[\norm{u^{t}}_{D_{t}^d}+\norm{v^{t}}_{D_{t}^{d\times d}}]<\infty,
\end{align} where the right-hand side is finite due to \cite[Corollary 4.14]{kp-sk}.\\
By \eqref{eq:lmn-bound}, $J^{T}(\partial_{i}V^{j})\in C_{T}L^{\infty}$ and \cite[Lemma 4.4]{kp} applied to both $M^{u^{t,m,n}}_{st},M^{v^{t,l,m,n}}_{st}$ we can estimate $\norm{(s,t)\mapsto M_{st}^{l,m,n}}_{(\alpha+\beta)/\alpha,p}$ for $p\in 2\N$. Together with \cref{lem:hd}
we thus obtain, for any $p\in 2\N$ the following uniform bound for large enough $l,n,m$:
\begin{align}\label{eq:p-bound}
\sup_{l,m,n}[\norm{Z^{m,n}}_{(\alpha+\beta)/\alpha,p}+\norm{\mathbb{Z}^{l,m,n}}_{(\alpha+\beta)/\alpha,p}]<\infty.
\end{align} 
Kolmogorov's continuity criterion and \eqref{eq:p-bound} yield tightness in $C(\Delta_{T},\R^{d+d\times d})$ of the laws of the processes $(Z^{m,n},\mathbb{Z}^{l,m,n})_{l,m,n}$. Furthermore, the uniform bound 
\begin{align}\label{eq:Z-infty-bound}
\sup_{m,n}\norm{Z^{m,n}\mid\F}_{(\alpha+\beta)/\alpha,\infty}<\infty
\end{align} follows from the representation \eqref{eq:Zmn-mart} and $u^{t,m,n}(t,X_{t}^{n})=0=u^{t,m,n}(t,X_{s}^{n})$ and thus for $m,n$ large enough,
\begin{align*}
\norm{\E_{s}[Z^{m,n}_{st}]}_{L^{\infty}(\p)}&\leqslant \sup_{t\in[0,T]}\norm{u^{t,m,n}}_{C_{t}^{(\alpha+\beta)/\alpha}L^{\infty}_{\R^{d}}}\abs{t-s}^{(\alpha+\beta)/\alpha}\\&\leqslant 2\sup_{t\in[0,T]}\norm{u^{t}}_{(\mathcal{L}_{t}^{0,\alpha+\beta})^d}\abs{t-s}^{(\alpha+\beta)/\alpha}
\end{align*} and the $\mathcal{L}_{t}^{0,\alpha+\beta}$-norm of $u^{t}$ is bounded in $t\in[0,T]$ due to \cite[Corollary 4.14]{kp-sk}.
By \cref{thm:class}, $X$ is of class $\mathcal{K}^{\vartheta}$. The bound from that proposition applied to the solutions $u^{t,m,n},v^{t,l,m,n}$ yields together with \eqref{eq:lmn-bound}, that for large enough $l,m,n$
\begin{align}\label{eq:infty-bound}
\sup_{l,m,n}\norm{\mathbb{Z}^{l,m,n}\mid\F}_{(2(\alpha+\beta)-1)/\alpha,\infty}<\infty.
\end{align} 
By tightness of $(Z^{m,n}, \mathbb{Z}^{l,m,n})_{l,m,n}$, uniqueness of the limit, since $X$ is the unique solution of the martingale problem (cf. \cref{thm:mainthm1}), and continuity of $(Z,\mathbb{Z}^{V})$ we can thus deduce the distributional convergence 
\begin{align*}
(Z^{m,n}, \mathbb{Z}^{l,m,n},L)\stackrel{l,m,n\to\infty}{\Longrightarrow} (Z,\mathbb{Z}^{V},L)\qquad\text{ in }\quad C(\Delta_{T},\R^{d+d\times d})\times D([0,T],\R^{d}).
\end{align*}
Hence, we obtain the distributional convergence
\begin{align*}
(X^{n},Z^{m,n},\mathbb{Z}^{l,m,n})\stackrel{l,m,n\to\infty}{\Longrightarrow} (X,Z,\mathbb{Z}^{V})\qquad\text{ in }\quad D([0,T],\R^{d})\times C(\Delta_{T},\R^{d+d\times d})
\end{align*} as $X^{n}=x+Z^{n,n}+L$, i.e. $X^{n}$ is given by a continuous map of $(Z^{n,n},L)$ (using \cite[Proposition VI.1.23]{Jacod2003}).\\
An application of Skorokhods representation theorem (cf.~\cite[Theorem 6.7]{Billingsley1968}) then yields that there exists a probability space $(\Omega,\F,\p)$ and random variables 
\begin{align*}
(Y^{n},W^{m,n},\mathbb{W}^{l,m,n})_{l,m,n}\quad\text{and}\quad(Y,W,\mathbb{W})
\end{align*} on $(\Omega,\F,\p)$ with $\text{Law}((Y^{n},W^{m,n},\mathbb{W}^{l,m,n}))=\text{Law}((X^{n},Z^{m,n},\mathbb{Z}^{l,m,n}))$ for $l,m,n\in\N$ and $\text{Law}((Y,W,\mathbb{W}))=\text{Law}((X,Z,\mathbb{Z}^{V}))$, such that the convergence 
\begin{align}\label{eq:conv-triple}
(Y^{n},W^{m,n},\mathbb{W}^{l,m,n})\stackrel{l,m,n\to\infty}{\longrightarrow}(Y,W,\mathbb{W})
\end{align} holds almost surely with respect to the topology on $D([0,T],\R^{d})\times C(\Delta_{T},\R^{d+d\times d})$ (i.e.~$J_{1}$-topology on the Skorokhod space and uniform topology on the space of continuous functions on $\Delta_{T}$).\\ 
We define the filtration as the completion of the canonical filtration of $Y$, $(\F_{t}):=(\F_{t}^{Y})\subset\F$. The filtration is right-continuous, since $Y$ is càdlàg and by construction complete. It follows that $L:=Y-x-W$ is an $\alpha$-stable symmetric non-degenerate $(\F_{t}^{(Y,W)})$-Lévy process, because $\text{Law}(X,Z)=\text{Law}(Y,W)$. Below, we show that $W$ is the almost sure limit of $(\F^{Y}_{t})$-adapted processes $(W^{m})_{m}$. This then implies, using also completeness of the filtration, that $L$ is $(\F^{Y}_t)$-measurable and thus also a $(\F_{t}^{Y})$-Lévy process.\\
Moreover, we have that $0=Z^{m,n}-\int_{0}^{\cdot}V^{m}(r,X^{n}_{r})dr\stackrel{d}{=}W^{m,n}-\int_{0}^{\cdot}V^{m}(r,Y^{n}_{r})dr$. This implies that $W^{m,n}=\int_{0}^{\cdot}V^{m}(r,Y^{n}_{r})dr$ almost surely. Analogously, we deduce the representation $\mathbb{W}^{l,m,n}_{st}(i,j)=\int_{s}^{t}[J^{T}(\partial_{i}V^{j,l})(r,Y_{r}^{n})-J^{T}(\partial_{i}V^{j,l})(s,Y_{s}^{n})]dW_{r}^{i,m,n}$. Let $W^{m}$, $\mathbb{W}^{l,m}$ be defined analogously with $Y^{n}$ replaced by $Y$. This yields the represenations \eqref{eq:Zmn-mart}, \eqref{eq:ZZmn-mart} for $W^{m,n},\mathbb{W}^{l,m,n}$. By almost sure convergence of \eqref{eq:conv-triple}, letting first $n\to\infty$, we obtain that for $(s,t)\in\Delta_{T}$,
\begin{align*}
W_{st}&=u^{t}(t,Y_{t})-u^{t}(s,Y_{s})+M_{st}^{u^{t}},\\
W^{m}_{st}&=u^{t,m}(t,Y_{t})-u^{t,m}(s,Y_{s})+M_{st}^{u^{t,m}},
\end{align*} and
\begin{align*}
\mathbb{W}_{st}(i,j)&=v^{t,i,j}(t,Y_{t})-v^{t,i,j}(s,Y_{s})\nonumber\\&\qquad-J^{T}(\partial_{i}V^{j,l})(s,Y_{s})\paren[\big]{u^{t,i}(t,Y_{t})-u^{t,i}(s,Y_{s})}\nonumber\\&\qquad+M_{st}^{v^{t,i,j}}-J^{T}(\partial_{i}V^{j,l})(s,Y_{s})M_{st}^{u^{t,i}},
\end{align*} and
\begin{align*}
\mathbb{W}_{st}^{l,m}(i,j)&=v^{t,l,m,i,j}(t,Y_{t})-v^{t,l,m,i,j}(s,Y_{s})\nonumber\\&\qquad-J^{T}(\partial_{i}V^{j,l})(s,Y_{s})\paren[\big]{u^{t,m,i}(t,Y_{t})-u^{t,m,i}(s,Y_{s})}\nonumber\\&\qquad+M_{st}^{v^{t,l,m,i,j}}-J^{T}(\partial_{i}V^{j,l})(s,Y_{s})M_{st}^{u^{t,m,i}}.
\end{align*}
Herein, we used that convergence in the $J_{1}$-topology implies in particular convergence Lebesgue almost everywhere in $[0,T]$ and that $Y$ almost surely does not jump at fixed times $t$ (cf.~in the proof of \cref{thm:mainthm1}).
The differences $M^{u^{t}}_{st},M^{u^{t,m}}_{st},M^{v^{t}}_{st},M^{v^{t,l,m}}_{st}$ are defined analogously as above and $u^{t,m}=(u^{t,m,i})_{i=1,\dots,d}$ and $v^{t,l,m}=(v^{t,l,m,i,j})_{i,j=1,\dots,d}$ solve 
\begin{align*}
\mathcal{G}^{\calV}u^{t,m,i}=V^{m,i}, \quad \mathcal{G}^{\calV}v^{t,l,m,i,j}=J^{T}(\partial_{i}V^{l,j})\cdot V^{m,i}
\end{align*} with zero terminal condition at $t\in [0,T]$.

\noindent It remains to prove that for $l,m\to\infty$
\begin{align}\label{eq:conv-1}
\norm{W^{m}-W}_{(\alpha+\beta)/\alpha,2}+\norm{W^{m}-W\mid\F}_{(\alpha+\beta)/\alpha,\infty}\to 0
\end{align} and
\begin{align}\label{eq:conv-2}
\norm{\mathbb{W}^{l,m}-\mathbb{W}}_{(\alpha+\beta)/\alpha,2}+\norm{\mathbb{W}^{l,m}-\mathbb{W}\mid\F}_{(2(\alpha+\beta)-1)/\alpha,\infty}\to 0.
\end{align}
The bounds \eqref{eq:zhoelder1}, \eqref{eq:ahoelder1} for the limits $W,\mathbb{W}$ then follow from the convergences \eqref{eq:conv-1} and \eqref{eq:conv-2} and the uniform bounds \eqref{eq:p-bound}, \eqref{eq:Z-infty-bound} and \eqref{eq:infty-bound} above.\\
The proof of \cref{lem:hd} shows the bound
\begin{align*}
\norm{W^{m}-W}_{(\alpha+\beta)/\alpha,2}+\norm{W^{m}-W\mid\F}_{(\alpha+\beta)/\alpha,\infty}\lesssim\sup_{t\in[0,T]}\norm{u^{t,m}-u^{t}}_{(\mathcal{L}_{t}^{0,\alpha+\beta})^d}.
\end{align*} 
The right-hand side vanishes if $m\to\infty$ by the uniform Lipschitz continuity from \cite[Corollary 4.14]{kp-sk}. 
Together this yields \eqref{eq:conv-1}. Analogously, we can argue for the convergence of $\norm{\mathbb{W}^{l,m}-\mathbb{W}}_{(\alpha+\beta)/\alpha,2}$, using \cite[Corollary 4.14]{kp-sk} and that
\begin{align*}
\norm{\mathbb{W}^{l,m}-\mathbb{W}}_{(\alpha+\beta)/\alpha,2}\lesssim\sup_{t\in[0,T]}[\norm{v^{t,l,m}-v^{t}}_{(\mathcal{L}_{t}^{0,\alpha+\beta})^{d\times d}}+\norm{V}_{C_{T}\calC^{\beta}_{\R^{d}}}\norm{u^{t,m}-u^{t}}_{(\mathcal{L}_{t}^{0,\alpha+\beta})^d}].
\end{align*}
Moreover, the estimate \eqref{eq:class-estimate} from the proof of \cref{thm:class} yields the bound
\begin{align*}
\MoveEqLeft
\norm{\mathbb{W}^{l,m}-\mathbb{W}\mid\F}_{(2(\alpha+\beta)-1)/\alpha,\infty}\\&\lesssim\sup_{ r\in [0,T]}\Big[\norm{w^{r,l,m}-w^{r}}_{(\mathcal{L}_{r}^{\gamma,\alpha+\beta})^{d\times d}}+\norm{w^{r,l,m,\sharp}-w^{r,\sharp}}_{(\mathcal{L}_{r}^{\gamma,2(\alpha+\beta)-1})^{d\times d}}\\&\qquad\qquad+\norm{\calV}_{\calX^{\beta,\gamma'}}(\norm{y^{r,m}-y^{r}}_{(\mathcal{L}_{r}^{\gamma,\alpha+\beta})^{d}}+\norm{y^{r,m}-y^{r}}_{(\mathcal{L}_{r}^{\gamma,2(\alpha+\beta)-1})^d})\Big],
\end{align*} where $w^{r}=(w^{r,i,j})_{i,j}$ denotes the solution of $\mathcal{G}^{\calV}w^{r,i,j}=0$ with terminal condition $w^{r,i,j}_{r}=J^{T}(\partial_{i}V^{j})_{r}\cdot V^{i}_{r}$ at time $r$ and $y^{r}=(w^{r,i})_{i}$ denotes the solution of $\mathcal{G}^{\calV}y^{r,i}=0$ with $y^{r,i}_{r}=V^{i}_{r}$. The right-hand side vanishes by the uniform Lipschitz bound from \cite[Corollary 4.14]{kp-sk}. Thus, together \eqref{eq:conv-2} follows.
Finally, we can of course rename $(Y,W,\mathbb{W})$ as $(X,Z,\mathbb{Z}^{V})$.
\end{proof}
\noindent The following theorem generalizes Itô's formula for rough weak solutions $X$.
\begin{theorem}\label{thm:genIto}
Let $\calV\in\mathcal{X}^{\beta,\gamma}$ for $\beta\in (\frac{2-2\alpha}{3},0)$ and let $(X,L,\mathbb{Z}^{V})$ be a weak solution started in $x\in\R^{d}$. Let $u\in D_{T}(V)$ be such that $v:=(\partial_{t}-\La)u$ is paracontrolled by $V$ in the sense that $v=v^{\sharp}+v^{\prime}\para V$ with $v^{\sharp}\in C_{T}\calC^{(2-\gamma)\alpha+2\beta-1}$ 
and $v^{\prime}\in C_{T}\calC^{\alpha+\beta-1}_{\R^{d}}$ and such that $u\in C^{1}([0,T],\calC^{\beta})$ (i.e. continuously differentiable in time with values in $\calC^{\beta}$).
Then, if $\alpha\in (1,2)$, the following Itô-formula holds true:
\begin{align*}
u(t,X_{t})&=u(0,x)+\int_{0}^{t}(\partial_{s}-\La)u(s,X_{s})ds + \int_{0}^{t}\nabla u(s,X_{s})\cdot d(Z,\mathbb{Z}^{V})_{s}\\&\qquad+\int_{0}^{t}\int_{\R^{d}\setminus\{0\}}[u(s,X_{s-}+y)-u(s,X_{s-})]\hat{\pi}(ds,dy),
\end{align*} where $Z=X-x-L$ and $\hat{\pi}$ denotes the compensated Poisson random measure of $L$. If $\alpha=2$ and $L=B$ for a Brownian motion $B$, the martingale is replaced by $\int_{0}^{t}\nabla u(s,X_{s})\cdot dB_{s}$. In the formula above, $\int_{0}^{t}(\partial_{s}-\La)u(s,X_{s})ds$ is defined as the limit of $\paren[\big]{\int_{0}^{t}(\partial_{s}-\La)u^{n}(s,X_{s})ds}_{n}$ in $L^{2}(\p)$ for the smooth mollifications $(u^{n}=P_{n^{-1}}u)_{n}$.
\end{theorem}
\begin{remark}
The assumptions on $u$ are satisfied for solutions $u$ of Kolmogorov backward equations with regular terminal conditions $u^{T}\in\calC^{2\alpha+2\beta-1}$ and right-hand sides $f\in C_{T}L^{\infty}$.
\end{remark}
\begin{proof}[Proof of \cref{thm:genIto}]
We give the proof in the case $\alpha\in (1,2)$, $\alpha=2$ is similar. Since $u\in C^{1}([0,T],\calC^{\beta})$ the mollification satisfies for $n\in\N$,
\begin{align*}
u^{n}_{t}=P_{n^{-1}}u_{t}\in C^{1}([0,T],C^{\infty}_{b})
\end{align*} 
In particular, $u^{n}\in C^{1,2}_b([0,T]\times\R^{d})$ and applying \cite[Theorem 3.1]{cjml} to $u^{n}(t,X_{t})$ yields
\begin{align*}
u^{n}(t,X_{t})&=u^{n}(0,x)+\int_{0}^{t}(\partial_{s}-\La)u^{n}(s,X_{s})ds + \int_{0}^{t}\nabla u^{n}(s,X_{s})\cdot dZ_{s}\\&\quad+\int_{0}^{t}\int_{\R^{d}\setminus\{0\}}[u^{n}(s,X_{s-}+y)-u^{n}(s,X_{s-})]\hat{\pi}(ds,dy).
\end{align*}
Due to convergence of the mollifications $u^{n}\to u$ in $\mathcal{L}^{0,\alpha+\beta}_{T}$ and analogue arguments as for \cite[Lemma 4.4]{kp} (with Burkholder-Davis-Gundy inequality and separating in large and small jumps), we obtain that the $\hat{\pi}$-martingales converge in $L^{2}(\p)$ to the one with $u^{n}$ replaced by $u$.\\
By \cref{thm:roughint} it follows that 
almost surely,
\begin{align*}
\int_{0}^{t}\nabla u^{n}(s,X_{s})\cdot dZ_{s}=\int_{0}^{t}\nabla u^{n}(s,X_{s})\cdot d(Z,\mathbb{Z}^{V})_{s}.
\end{align*}
The stability of the rough stochastic integral from \cref{thm:stabilityint} for the rough stochastic integral yields convergence of the integrals if $(u^{n},u^{\prime})\to (u,u^{\prime})$ in $D_{T}(V)$. The latter follows from the convergence of the mollification $u^{n}\to u$ in $\mathcal{L}_{T}^{0,\alpha+\beta}$ and 
\begin{align*}
u^{n,\sharp}&=u^{n}-u^{\prime}\para J^{T}(V)\\&=P_{n^{-1}}u^{\sharp}+\paren[\big]{P_{n^{-1}}[u^{\prime}\para J^{T}(V)]-u^{\prime}\para P_{n^{-1}}J^{T}(V)}\\&\to u^{\sharp}
\end{align*} in $\mathcal{L}_{T}^{0,2(\alpha+\beta)-1}$ due to \cite[Corollary 3.2]{kp-sk} and the commutator \cite[Lemma 3.4]{kp-sk}. 
It remains to show that the additive functional 
\begin{align*}
\int_{0}^{t}(\partial_{s}-\La)u^{n}(s,X_{s})ds
\end{align*} 
converges in $L^{2}(\p)$. 
For $r\in [0,T]$ we have by assumption for $v_{r}:=(\partial_{r}-\La)u_{r}$, that
\begin{align*}
v_{r}=v_{r}^{\sharp}+v^{\prime}_{r}\para V_{r},
\end{align*} 
for $v_{r}^{\sharp}\in \calC^{(2-\gamma)\alpha+2\beta-1}$  and $v_{r}^{\prime}\in\calC^{\alpha+\beta-1}$. Thus $v_{r}$ is an admissible terminal condition (in the sense of \cref{thm:singular})  for solving the Kolmogorov equation on $[0,r]$ with paracontrolled terminal condition $v_{r}$.\\
By \cref{thm:wsimp}, $X$ is a martingale solution and in particular a strong Markov process. 
Denote again by $(T_{s,r})_{0\leqslant s\leqslant r\leqslant T}$ its semigroup.
Utilizing the ideas of \cite[section 3]{le}, we define,
\begin{align*}
\Xi_{s,t}:=\int_{s}^{t}T_{s,r}v_{r}(X_{s})dr.
\end{align*}
We let $y_{s}:=T_{s,r}v_{r}$ for $s\in[0,r]$. Then $y$ is a mild solution of the Kolmogorov equation with terminal condition $v_{r}$.\\
We apply the stochastic sewing \cref{lem:ss} to $\Xi$.
Therefore, we write for $s\leqslant u\leqslant t$,
\begin{align*}
\Xi_{s,u,t}=\Xi_{s,t}-\Xi_{s,u}-\Xi_{u,t}&=\int_{u}^{t}[T_{s,r}v_{r}(X_{s})-T_{u,r}v_{r}(X_{u})]dr
\end{align*}
Due to the Markov property of $X$, we see that $\E_{s}[\Xi_{s,u,t}]=0$. It is left to estimate $\norm{\Xi_{s,u,t}}_{L^{2}(\p)}$.
To this end, we use that $y=T_{\cdot,r}v_{r}\in\mathcal{M}_{r}^{-\beta/\alpha}L^{\infty}$ with a uniform bound in $r\in[0,T]$ due to \cite[Corollary 4.14]{kp-sk} and the interpolation estimate \cite[Lemma 3.7, (3.13)]{kp-sk}, such that we can estimate,
\begin{align*}
\norm{y_{u}}_{L^{\infty}}=\norm{T_{u,r}v_{r}}_{L^{\infty}}\lesssim \abs{r-u}^{\beta/\alpha}\sup_{r\in[0,T]}\norm{y}_{\mathcal{M}_{r}^{-\beta/\alpha}L^{\infty}}
\end{align*} and thus it follows that
\begin{align*}
\norm{\Xi_{s,u,t}}_{L^{2}(\p)}\lesssim\abs{t-u}^{1+\beta/\alpha}=\abs{t-u}^{(\alpha+\beta)/\alpha},
\end{align*} with $(\alpha+\beta)/\alpha>1/2$. Thus the stochastic sewing lemma applies and we obtain existence of the integral $I_{t}^{X}(v):=\lim_{\abs{\Pi}\to 0}\sum_{s,t\in\Pi}\Xi_{s,t}\in L^{2}(\p)$. Furthermore, if we consider $v^{n}:=(\partial_{t}-\La)u^{n}$, we obtain that $I_{t}^{X}(v^{n})\to I_{t}^{X}(v)$ due to stability of the stochastic sewing integral and continuity of the Kolmogorov solution map. The last step is then to see that almost surely $I_{t}^{X}(v^{n})=\int_{0}^{t}v^{n}(s,X_{s})ds$. Indeed, this can be deduced from uniqueness of the stochastic sewing integral and $\E_{s}[v^{n}_{r}(X_{r})]=T_{s,r}v^{n}_{r}(X_{s})$ as $v^{n}_{r}\in C_{b}$ (cf.~also \cite[Proposition 3.7]{le}). 
\end{proof}

\end{section}

\begin{section}{Ill-posedness of the canonical weak solution concept in the rough regime}\label{sec:wsyc}

Below, we introduce so-called canonical weak solutions. We prove in \cref{cor:wsemp-y}, that in the Young case the canonical weak solution concept is well-posed (and equivalent to martingale solutions). We finalize by proving in \cref{thm:nonu} that canonical weak solutions are in general non-unique. We construct a counterexample, that justifies the latter.\\ 
We start by defining the concept of canonical weak solutions. 
\begin{definition}[Canonical weak solution]\label{def:wsy}
Let $\alpha\in (1,2]$, $V\in\CTcalC^{\beta}_{\R^{d}}$ for $\beta\in (\frac{2-2\alpha}{3},0)$. Let $x\in\R^{d}$. We call a tuple $(X,L)$ a canonical weak solution starting at $X_{0}=x\in\R^{d}$, if there exists a stochastic basis $(\Omega,\F, (\F_{t})_{t\geqslant 0},\p)$, such that $L$ is an $\alpha$-stable symmetric non-degenerate $(\F_{t})$-Lévy process and almost surely
\begin{align*}
X=x+Z+L,
\end{align*} where $Z$ is an $(\F_{t})$-adapted, continuous process with 
\begin{align*}
\norm{Z}_{(\alpha+\beta)/\alpha,2}+\norm{Z\mid\F}_{(\alpha+\beta)/\alpha,\infty}<\infty.
\end{align*} 
Moreover there exists a sequence $(V^{n})\subset C_{T}C^{\infty}_{b}$ with $V^{n}\to V$ in $\CTcalC^{\beta}$, such that 
\begin{align}\label{eq:yz}
\lim_{n\to\infty}\int_{0}^{\cdot}V^{n}(s,X_{s})ds=Z,
\end{align} with convergence in $L^{2}(\p)$, uniformly in $[0,T]$. 
\end{definition}
\begin{remark}
The definition of canonical weak solutions is similar to \cite[Definition 2.1]{Athreya2018}. The difference is the assumption that $\norm{Z\mid\F}_{(\alpha+\beta)/\alpha,\infty}<\infty$. However, boths bounds on $Z$ are natural to assume and motivated by \cref{lem:hd}.
\end{remark}
\begin{remark}[One-dimensional, time-homogeneous case with $\alpha=2$]\label{rem:cansol} In this case, for any approximation $(V^{n})$ of $V$, the weak limit of the strong solutions $(X^{n})$ with 
\begin{align*}
dX^{n}_{t}=V^{n}(X_{t}^{n})dt+dB_{t}
\end{align*} is the same and given by the solution of the $\mathcal{G}^{\calV}$-martingale problem.  
The one-dimensional case is special in this sense. The above is true, because in $d=1$ the resonant products $(J^{T}(\partial_{x} V^{n})\reso V^{n})$ converge to the same  limit for any approximation $(V^{n})$.
This can be seen with Leibniz rule considering $((-\Delta)^{-1}(\partial_{x} V^{n})\reso V^{n})_{n}=(J^{\infty}(\partial_{x}V^{n})\reso V^{n})_{n}$. Indeed, we have with $v^{n}$, such that $V^{n}=\partial_{x} v^{n}$
\begin{align*}
\lim_{n\to\infty}(-\Delta)^{-1}(\partial_{x}V^{n})\reso V^{n}=-\frac{1}{2}\lim_{n\to\infty}\partial_{x}(v^{n}\reso v^{n})=-\frac{1}{2}\partial_{x}(v\reso v)
\end{align*} using that $(-\Delta_{x})^{-1}(\partial_{x}V^{n})=-v^{n}$ and $v^{n}\reso v^{n}\to v\reso v$ as products of functions. Notice moreover that this does not imply, that the limit of the mixed resonant products $((-\Delta)^{-1}(\partial_{x} V^{n})\reso V)_{n}$ is uniquely determined. In fact in the rough case the latter limit is in general not unique (cf. \cref{prop:exv} below). This fact will imply that, even in the one-dimensional case, a weak solution in the sense of \cref{def:wsy} is in general non-unique in law.
\end{remark}
\begin{remark}[One and all sequences $(V^{n})$] 
One may ask, if requiring the convergence \eqref{eq:yz} to hold for all (instead of one) approximating sequences $(V^{n})$ renders the solution concept well-posed. Though, if we require \eqref{eq:yz} to hold for all such sequences $(V^{n})$, then in the rough case the solution of the $\mathcal{G}^{\calV}$-martingale problem won't be a solution. This follows also from the fact that the limit of the mixed resonant products is non-unique. Thus, we would expect non-existence of solutions in the rough case. 
Hence, requiring \eqref{eq:yz} to hold for all such sequences $(V^{n})$ makes the definition too restrictive.
\end{remark}
\noindent The results from the previous section imply the following corollary.
\begin{corollary}\label{cor:wsemp-y}
Let $\beta\in ((1-\alpha)/2,0)$ (Young regime) and $V\in C_{T}\calC^{\beta}$. Then, $X$ is a solution of the $(\mathcal{G}^{V},\delta_{x})$-martingale problem if and only if $X$ is a canonical weak solution starting at $x\in\R^{d}$. In particular, the canonical weak solution concept is well-posed in the Young regime.
\end{corollary}
\begin{proof}
If $X$ is a martingale solution, then $X$ is a canonical weak solution by \cref{thm:mpiws}, since a weak solution in the sense of \cref{def:ws} is in particular a canonical weak solution. For the reverse implication, notice that \eqref{eq:yz} and the assumption $Z\in C_{T}^{(\alpha+\beta)/\alpha}L^{2}(\p)$ imply that $\norm{Z^{n}-Z}_{\theta,2}\to 0$ for any $\theta<(\alpha+\beta)/\alpha$. Then, \cref{lem:generalws} and the arguments in the proof of \cref{thm:wsimp} for $\beta$ in the Young regime (that do not use the bounds \eqref{eq:ahoelder1} on $\mathbb{Z}^{V}$ and the convergence $2.)$) imply that $X$ is a solution of the $(\mathcal{G}^{V},\delta_{x})$-martingale problem. 
\end{proof}
\noindent If $\beta\leqslant\frac{1-\alpha}{2}$, we show that the solution concept from \cref{def:wsy} is ill-posed. The idea is as follows.
We construct two different lifts of $V$ with two different solutions of the respective Kolmogorov backward equations, which yields two different weak solutions $X$ in law. 
The next lemma proves the existence of such desired lifts. We give its proof after the proof of \cref{thm:nonu}.
\begin{lemma}\label{prop:exv}
Let $d=1$ and $\alpha=2$. Let $(P_{t})_{t\geqslant 0}$ be the heat semigroup, that is $P_{t}f=\mathcal{F}^{-1}(\exp(\frac{t}{2}\abs{2\pi \cdot}^{2})\hat{f})$.  
Then, there exists $V\in C_{T}\calC^{\beta}$ (in fact, time-independent) and two sequences $(V^{n})\subset C_{T}C^{\infty}_{b}$, $(W^{n})\subset C_{T}C^{\infty}_{b}$ such that $V^{n}\to V\in C_{T}\calC^{\beta}$ and $W^{n}\to V\in C_{T}\calC^{\beta}$ and such that  $J^{T}(\partial_{x} V^{n})\reso V\to\mathcal{V}_{2}$ and $J^{T}(\partial_{x} W^{n})\reso V\to\mathcal{W}_{2}$ in $C_{T}\calC^{2\beta+1}$, where $\mathcal{V}_{2}=\mathcal{W}_{2}+C$ for a constant $C\neq 0$. 
\end{lemma}
\begin{theorem}\label{thm:nonu}
Let $\beta\leqslant\frac{1-\alpha}{2}$ and $V\in C_{T}\calC^{\beta}_{\R^{d}}$ and $V\neq 0$. Then, canonical weak solutions in the sense of \cref{def:wsy} are in general non-unique in law.
\end{theorem}
\begin{proof}[Proof of \cref{thm:nonu}]
We construct two solutions that are not equal in law. To that end, we let $d=1$, $\alpha=2$. 
By \cref{prop:exv}, there exists $V\in C_{T}\calC^{\beta}$ and two sequences $(V^{n,i})$ for $i=1,2$ with $(V^{n,1},J^{T}(\partial_{x}V^{n,1})\reso V)\to (V,\mathcal{V}_{2})$ and $(V^{n,2},J^{T}(\partial_{x}V^{n,2})\reso V)\to (V,\mathcal{W}_{2})$ in $C_{T}\calC^{\beta}\times C_{T}\calC^{2\beta+\alpha-1}$ and $\mathcal{V}_{2}=\mathcal{W}_{2}+C$ for $C\neq 0$.\\
Let $u$ be the solution of $\mathcal{G}^{(V,\mathcal{V}_{2})}u=V$ with $u(T,\cdot)=0$ and $\tilde{u}$ be the solution of $\mathcal{G}^{(V,\mathcal{W}_{2})}\tilde{u}=V$, $\tilde{u}(T,\cdot)=0$. Since $\mathcal{G}^{(V,\mathcal{V}_{2})}=\mathcal{G}^{(V,\mathcal{W}_{2})}+C\partial_{x}$, it follows that 
\begin{align*}
\tilde{u}(t,y)=u(t,y-(T-t)C),\quad (t,y)\in [0,T]\times\R.
\end{align*} 
In particular, there exists $(s,x)\in[0,T)\times\R$, such that $u(s,x)\neq\tilde{u}(s,x)$ (otherwise $u$ is constant, which implies that $V= 0$, which is excluded in the theorem). Consider the shifted solutions $v(t,x):=u(t+s,x)$ and  $\tilde{v}(t,x):=\tilde{u}(t+s,x)$, $(t,x)\in [0,T]\times\R$. Then we have that $v(0,x)\neq\tilde{v}(0,x)$. \\ 
Let $X^{1}$ be the solution of the $(\mathcal{G}^{(V,\mathcal{V}_{2})},\delta_{x})$ martingale problem and $X^{2}$ be the solution of the  $(\mathcal{G}^{(V,\mathcal{W}_{2})},\delta_{x})$ martingale problem. 
By \cref{thm:mpiws}, there exist stochastic basis $(\Omega_{1},\F^{1},(\F_{t}^{1}),\p_{1})$ and $(\Omega_{2},\F^{2},(\F_{t}^{2}),\p_{2})$, such that $X^{i}$ for $i=1,2$ satisfy
\begin{align*}
X^{i}=x+Z^{i}+B^{i},\text{ a.s., where } Z^{i}=\lim_{n\to\infty}\int_{0}^{\cdot}V^{n,i}(s,X^{i}_{s})ds\in L^{2}(\p_{i})
\end{align*}  
and such that $\norm{Z^{i}}_{(\alpha+\beta)/\alpha,L^{2}(\p_{i})}+\norm{Z^{i}\mid\F^{i}}_{(\alpha+\beta)/\alpha,L^{\infty}(\p_{i})}<\infty$.
In particular, $X^{1}$ and $X^{2}$ are canonical weak solutions in the sense of \cref{def:wsy}.\\ We prove that $\text{Law}(X^{1})\neq \text{Law}(X^{2})$.
But this is clear, if we show that 
\begin{align*}
\E_{1}[Z^{1}_{T-s}]\neq\E_{2}[Z^{2}_{T-s}].
\end{align*} 
Let $u_{n,1}$ be the solution of $\mathcal{G}^{(V,\mathcal{V}_{2})}u_{n,1}=V^{n,1}$ with $u_{n,1}(T,\cdot)=0$ and $u_{n,2}$ be the solution of $\mathcal{G}^{(V,\mathcal{W}_{2})}u_{n,2}=V^{n,2}$ with $u_{n,2}(T,\cdot)=0$. Let $v^{n,1},v^{n,2}$ be the shifted solutions, that solve the same equations as $u^{n,1},u^{n,2}$ with $v^{n,1}(T-s,\cdot)=v^{n,2}(T-s,\cdot)=0$. Then $u_{n,1}\to u$ and $u_{n,2}\to\tilde{u}$ in $C_{T}L^{\infty}_{\R^{d}}$. As $X^{1}$ solves the $(\mathcal{G}^{(V,\mathcal{V}_{2})},\delta_{x})$ martingale problem and $X^{2}$ solves the  $(\mathcal{G}^{(V,\mathcal{W}_{2})},\delta_{x})$ martingale problem, we have (abbreviating the martingale term by $M^{v_{n,1}}$) 
\begin{align*}
\E_{1}[Z^{1}_{T-s}]&=\lim_{n\to\infty}\E_{1}\bigg[\int_{0}^{T-s}V^{n,1}(r,X^{1}_{r})dr\bigg]
\\&=\lim_{n\to\infty}\E_{1}[v_{n,1}(T-s,X^{1}_{T})-v_{n,1}(0,x)-M_{0,T-s}^{v_{n,1}}]
\\&=\lim_{n\to\infty}u_{n,1}(s,x)\\&=u(s,x)\neq \tilde{u}(s,x)=\E_{2}[Z^{2}_{T-s}],
\end{align*} such that the claim follows.
\end{proof}
\begin{proof}[Proof of \cref{prop:exv}]
Recall that $d=1$ and $L=B$ for a standard Brownian motion $B$. 
We construct a distribution $V$, that is time independent and can thus integrate out time in $J^{T}(\partial_{x} V)$, i.e. instead of 
\begin{align*}
J^{T}(\partial_{x} V)(r)&=\int_{r}^{T}\F^{-1}\paren[\big]{\exp\paren[\big]{(s-r)\frac{1}{2}\abs{2\pi \cdot}^2}\F(\partial_{x} V)}ds\\&=\F^{-1}\paren[\bigg]{(1-\exp((r-T)\frac{1}{2}\abs{2\pi \cdot}^2))\frac{1_{\cdot\neq 0}}{\frac{1}{2}\abs{2\pi\cdot}^{2}}\F(\partial_{x} V)}\\&=:J^{\infty}(\partial_{x} V)-\varphi_{r,T}\ast J^{\infty}(\partial_{x} V),
\end{align*} we consider w.l.o.g.~$J^{\infty}(\partial_{x} V)=\F^{-1}(\frac{1_{\cdot\neq 0}}{\frac{1}{2}\abs{2\pi\cdot}^{2}}\F(\partial_{x} V))=(-\frac{1}{2}\Delta)^{-1}(\operatorname{Id}-\Pi_{0})\partial_{x} V$, where $\Pi_{0}$ is the projection onto the zero-order Fourier mode.\\
Inspired by \cite[Lemma 3.6]{Chouk18}, we set
\[ f (x) = \sum_{k > 0} a_k e^{2 \pi i 2^k x} \]
for $(a_k) \subset \mathbb{C}$ to be determined. To simplify the calculations we assume that $a_{k}=0$ for all odd $k$.  We define
\[ F:=J^{\infty}(\partial_{x} f) = \left( - \frac{1}{2} \Delta \right)^{- 1} (\tmop{id} - \Delta_0)
   \partial_x f = \sum_{k > 0} a_k \frac{2 \pi i 2^k}{\frac{1}{2} | 2 \pi 2^k
   |^2} e^{2 \pi i 2^k x} = \sum_{k > 0} \frac{i a_k}{\pi 2^k} e^{2 \pi i 2^k
   x} . \]
Then, it follows that $\Delta_k f (x) = a_k e^{2 \pi i 2^k x}$ and $\Delta_k F (x) = \frac{i
a_k}{2 \pi 2^k} e^{2 \pi i 2^k x}$, and therefore (because $a_{k}\neq 0$ only for even $k$)
\begin{align*}
  F \reso f &= \sum_{k > 0} \frac{i a_k^2}{\pi 2^k} e^{2 \pi i 2^{k + 1} x},
  \qquad \overline{F} \reso \overline{f} = - \sum_{k > 0} \frac{i
  \overline{a_k}^2}{\pi 2^k} e^{- 2 \pi i 2^{k + 1} x},\\
  F \reso \overline{f} &= \sum_{k > 0} \frac{i | a_k |^2}{\pi 2^k}, \qquad\qquad\quad
  \overline{F} \reso f = - \sum_{k > 0} \frac{i | a_k |^2}{\pi 2^k} .
\end{align*}
\noindent Letting $a_k := 2^{-k/2}$ for even $k$,
we obtain that $f \in \calC^{-1/2-}$. 
Moreover, taking $V := \tmop{Re} (f) = \frac{1}{2} \left( f + \overline{f}
\right)$ and $J^{\infty}(\partial_{x}V) = \tmop{Re} (F) = \frac{1}{2} \left( F + \overline{F}
\right)$, we obtain for the resonant product
\begin{align*}
  \calV_{2}:=J^{\infty}(\partial_{x}V) \reso V & = \frac{1}{4} \left( F \reso f + \overline{F} \reso \overline{f}
  + F \reso \overline{f} + \overline{F} \reso f \right)\\
  & = \frac{1}{4} \sum_{k > 0} \frac{i a_k^2}{\pi 2^k} (e^{2 \pi i 2^{k + 1}
  x} - e^{- 2 \pi i 2^{k + 1} x})\\
  & = \frac{1}{4} \sum_{k > 0} \frac{i }{\pi} 2 i \sin (2 \pi 2^{k + 1}
  x) = - \sum_{k > 0} \frac{1}{2 \pi} \sin (2 \pi 2^{k + 1} x) .
\end{align*}
The latter is a distribution in $\calC^{0-}$. 
Clearly, we have that $V^n \rightarrow V$ in $\calC^{-
1 / 2 -}$ for $V^{n} (x) := \tmop{Re} \left( \sum_{k =
1}^n a_k e^{2 \pi i 2^k x} \right)$. Then it follows 
\begin{align*}
J^{\infty}(\partial_{x} V^{n})
\reso V \rightarrow J^{\infty}(\partial_{x}V) \reso V=:\calV_{2},
\end{align*} in $\calC^{0 -}$ for $n\to\infty$.\\
Let $(c_{n})\subset\mathbb{C}$ to be determined.
The other sequence $(W^{n})$, we then define as follows
\begin{align*}
  W^{n} (x) &:= \tmop{Re} \left( \sum_{k = 1}^n a_k e^{2 \pi i 2^k x}
  \right) + \tmop{Re} (c_n e^{2 \pi i 2^n x}) = V^{n} (x) + \tmop{Re} (c_n e^{2
  \pi i 2^n x}),\\
J^{\infty}(\partial_{x} W^{n})(x)&= \tmop{Re} \left( \sum_{k = 1}^n \frac{i a_k}{\pi 2^k} e^{2
  \pi i 2^k x} \right) + \tmop{Re} \left( \frac{i c_n}{\pi 2^n} e^{2 \pi i 2^n
  x} \right) \\&= J^{\infty}(\partial_{x}V^{n}) (x) + \tmop{Re} \left( \frac{i c_n}{\pi 2^n} e^{2 \pi i 2^n
  x} \right).
\end{align*}
Writing $G_n = \frac{i c_n}{\pi 2^{n}} e^{2 \pi i 2^{n} x}$, we have
(recall that $a_n \in \mathbb{R}$)
\begin{align*}
  G_n \reso f &= \frac{i a_n c_n}{\pi 2^n} e^{2 \pi i 2^{n + 1} x}, \qquad
  \overline{G_n} \reso \overline{f} = - \frac{i a_n \overline{c_n}}{\pi 2^n}
  e^{- 2 \pi i 2^{n + 1} x},\\
  G_n \reso \overline{f}& = \frac{i a_n c_n}{\pi 2^n}, \qquad\qquad\quad \overline{G_n}
  \reso f = - \frac{i a_n \overline{c_n}}{\pi 2^n},
\end{align*}
and thus for $c_n := 2^{n / 2} d_n$ for $(d_{n})\subset\mathbb{C}$ to be determined below,
\begin{align*}
  \tmop{Re} (G_n) \reso V & = \frac{1}{4} \frac{i a_n}{\pi 2^n} \left( c_n
  e^{2 \pi i 2^{n + 1} x} - \overline{c_n} e^{- 2 \pi i 2^{n + 1} x} \right) +
  \frac{1}{4} \frac{a_n}{\pi 2^n} \left( i c_n - i \overline{c_n} \right)\\
  & = \frac{1}{4} \frac{i }{\pi} \left( d_n e^{2 \pi i 2^{n + 1} x} -
  \overline{d_n} e^{- 2 \pi i 2^{n + 1} x} \right) - \frac{1}{2} \frac{
   \tmop{Im} d_n}{\pi}\\
  & = - \frac{1}{2 \pi} [ \tmop{Im} (d_n e^{2 \pi i 2^{n + 1} x}) + \tmop{Im} (d_n)] .
\end{align*}
We now set $d_n := - 2 \pi C i$ for a constant $C \in \mathbb{R}\setminus \{0\}$. Then we have
$W^{n} \rightarrow V$ in $\calC^{- 1 / 2 -}$ and the following convergences
in $\calC^{0 -}$:
\[ \tmop{Re}(G_n) \reso V \to C, \qquad J^{\infty}(\partial_{x}W^{n}) \reso V \to J^{\infty}(\partial_{x}V) \reso V + C=:\mathcal{W}_{2}, \]
for $n\to\infty$.
\end{proof}
\end{section}

\appendix
\begin{section}{Appendix}\label{Appendix A}
\begin{proof}[Proof of \cref{lem:parastr-est}]
We use the paracontrolled structure of $u$ to prove the bound \eqref{eq:para-est}, but nonetheless the bound does not trivially follow from that. We abbreviate $\theta:=\alpha+\beta$. 
Let us first make some observations. By the Schauder estimates, \cite[Corollary 3.2]{kp-sk}, we have that $J^{T}(\partial_{i} V)\in(\mathcal{L}_{T}^{0,\theta-1+(1-\gamma)\alpha})^{d}$ as $V\in C_{T}\calC^{\beta+(1-\gamma)\alpha}_{\R^{d}}$, and thus by the interpolation estimate (3.12) from \cite[Lemma 3.7]{kp-sk} (applied for $\tilde{\theta}=2(\theta-1)\in (0,\alpha)$), 
we obtain that $J^{T}(\partial_{i} V)\in C_{T}^{2(\theta-1)/\alpha}\calC^{(1-\gamma)\alpha-(\theta-1)}_{\R^{d}}$ and $\nabla u\in C_{T}^{2(\theta-1)/\alpha}\calC^{-(\theta-1)}_{\R^{d}}$. With that, we can estimate the resonant and paraproduct product as $(1-\gamma)\alpha>0$ and the following notation $\nabla u_{sr}(x):=\nabla u(r,x)-\nabla u(s,x)$ (and analogously for $J^{T}(\partial_{i}V)$)
\begin{align}\label{eq:goodtimereg1}
\MoveEqLeft
\norm{\nabla u_{sr}\reso J^{T}(\partial_{i} V)_{r}}_{L^{\infty}}+\norm{\nabla u_{sr}\para J^{T}(\partial_{i} V)_{r}}_{L^{\infty}}\nonumber\\&\lesssim \norm{\nabla u_{sr}\reso J^{T}(\partial_{i} V)_{r}}_{\calC^{(1-\gamma)\alpha}}+\norm{\nabla u_{sr}\para J^{T}(\partial_{i} V)_{r}}_{\calC^{(1-\gamma)\alpha}}\nonumber\\&\lesssim\norm{J^{T}(\partial_{i} V)}_{C_{T}\calC^{\theta-1+(1-\gamma)\alpha}}\norm{\nabla u}_{C_{T}^{(2\theta-2)/\alpha}\calC^{-(\theta-1)}} \abs{r-s}^{(2\theta-2)/\alpha}.
\end{align} 
Furthermore, using that by assumption $(1-\gamma)\alpha-(\theta-1)=-\gamma\alpha-\beta+1<0$ since $\gamma\geqslant(2\beta+2\alpha-1)/\alpha>(1-\beta)/\alpha$ in the rough case, and $\gamma>(1-\beta)/\alpha$ in the Young case (cf. \cref{def:enhanced-dist}), we obtain
\begin{align}\label{eq:goodtimereg2}
\MoveEqLeft
\norm{J^{T}(\partial_{i} V)_{sr}\reso \nabla u_{s}}_{L^{\infty}}+\norm{J^{T}(\partial_{i} V)_{sr}\para \nabla u_{s}}_{L^{\infty}}\nonumber\\&\lesssim \norm{J^{T}(\partial_{i} V)_{sr}\reso \nabla u_{s}}_{\calC^{(1-\gamma)\alpha}}+\norm{J^{T}(\partial_{i} V)_{sr}\para \nabla u_{s}}_{\calC^{(1-\gamma)\alpha}}\nonumber\\&\lesssim \norm{J^{T}(\partial_{i} V)}_{C_{T}^{2(\theta-1)/\alpha}\calC^{(1-\gamma)\alpha-(\theta-1)}}\norm{\nabla u}_{C_{T}\calC^{\theta-1}} \abs{r-s}^{(2\theta-2)/\alpha}.
\end{align} Notice that we do not have the symmetric estimate for the time difference in the upper part of the product.\\
Moreover, using the paracontrolled structure we have that
\begin{align}\label{eq:q1}
\MoveEqLeft
\partial_{i} u (r,x)-\partial_{i} u(s,y)\nonumber\\&=(\nabla u\para J^{T}(\partial_{i} V))(r,x)-(\nabla u \para J^{T}(\partial_{i} V))(s,y)+\partial_{i} u^{\sharp}(r,x)-\partial_{i} u^{\sharp}(s,y)\nonumber\\&\quad+(\partial_{i}\nabla u\para J^{T}(V))(r,x)-(\partial_{i}\nabla u\para J^{T}(V))(s,y)\nonumber\\&=(\nabla u\para J^{T}(\partial_{i} V))(r,x)-(\nabla u \para J^{T}(\partial_{i} V))(s,y)+g(r,x)-g(s,y)
\end{align} for $g$ defined as 
\begin{align*}
g:=\partial_{i} u^{\sharp}+\partial_{i}\nabla u\para J^{T}( V).
\end{align*}
By $\partial_{i}\nabla u\para J^{T}( V)\in C_{T}^{(2\theta-2)/\alpha}L^{\infty}\cap C_{T}\calC^{2\theta-2}$, $\partial_{i} u^{\sharp}\in C_{T}^{(2\theta-2)/\alpha}L^{\infty}\cap C_{T}\calC^{2\theta-2}$ and the interpolation estimates, we obtain the desired estimate for $g$:
\begin{align*}
\abs{g(r,x)-g(s,y)}&\leqslant\abs{g(r,y)-g(s,y)}+\abs{g(r,x)-g(r,y)}\\&\lesssim\abs{r-s}^{(2\theta-2)/\alpha}+\abs{x-y}^{2\theta-2}.
\end{align*}
Thus it is left to show that 
\begin{align*}
\MoveEqLeft
\squeeze[1]{\abs{(\nabla u\para J^{T}(\partial_{i} V))(r,\!x)-(\nabla u \para J^{T}(\partial_{i} V))(s,\!y)-\nabla u(s,\!y)\cdot(J^{T}(\partial_{i} V)(r,\!x)-J^{T}(\partial_{i} V)(s,\!y))}}\\&\lesssim\abs{t-s}^{(2\theta-2)/\alpha}+\abs{x-y}^{2\theta-2}.
\end{align*}
To that aim, we replace $r$ by $s$ and subtract the remainder, such that we obtain
\begin{align}\label{eq:x-t-separation}
\MoveEqLeft
\bigl|(\nabla u\para J^{T}(\partial_{i} V))(r,x)-(\nabla u \para J^{T}(\partial_{i} V))(s,y)\nonumber\\&\qquad-\nabla u(s,y)\cdot(J^{T}(\partial_{i} V)(r,x)-J^{T}(\partial_{i} V)(s,y))\bigr|\nonumber\\&\leqslant\bigl|(\nabla u\para J^{T}(\partial_{i} V))(s,x)-(\nabla u \para J^{T}(\partial_{i} V))(s,y)\nonumber\\&\qquad\qquad-\nabla u(s,y)\cdot(J^{T}(\partial_{i} V)(s,x)-J^{T}(\partial_{i} V)(s,y))\bigr|\nonumber\\&\qquad+\abs{(\nabla u\para J^{T}(\partial_{i} V))_{sr}(x)-\nabla u(s,y)J^{T}(\partial_{i} V)_{sr}(x)}.
\end{align}
For the first term in \eqref{eq:x-t-separation}, we abbreviate $v=J^{T}(\partial_{i} V)_{s}\in\calC^{\theta-1}_{\R^{d}}$ and $u=\nabla u_{s}\in\calC^{\theta-1}_{\R^{d}}$ and utilize the space regularities to prove the claim
\begin{align}\label{eq:claim}
\abs{u\para   v(x)- u\para   v(y)- u(y)( v(x)- v(y))}\lesssim\abs{x-y}^{2\theta-2}.
\end{align}
To prove \eqref{eq:claim}, we use ideas from the proof of \cite[Lemma B.2]{Gubinelli2015Paracontrolled}.
We write 
\begin{align}\label{eq:sums}
\MoveEqLeft
u\para   v(x)- u\para   v(y)- u(y)( v(x)- v(y))\nonumber\\&=\sum_{j\geqslant -1}(S_{j-1} u(x)- u(y))(\Delta_{j} v(x)-\Delta_{j} v(y))+\sum_{j}(S_{j-1} u(y)-S_{j-1} u(x))\Delta_{j}  v(y)
\end{align} with the notation $S_{j-1}u:=\sum_{-1\leqslant i\leqslant j-1}\Delta_{i}u$.
The first summand, we estimate in two different ways, once using that $\calC^{\theta-1}\subset C^{\theta-1}$, where $C^{\theta-1}$ is the Hölder space with $\theta-1\in (0,1)$ and on the other hand using the mean value theorem, such that
\begin{align*}
\MoveEqLeft
\abs{(S_{j-1} u(x)- u(y))(\Delta_{j} v(x)-\Delta_{j} v(y))}\\&\lesssim 2^{-j(\theta-1)}\norm{u}_{\theta-1}\paren[\big]{\abs{x-y}^{\theta-1}\norm{v}_{\theta-1}\wedge\abs{x-y}\norm{D\Delta_{j}v}_{L^{\infty}}}\\&\lesssim \squeeze[0.1]{2^{-j(\theta-1)}\norm{u}_{\theta-1}\paren[\big]{\abs{x-y}^{\theta-1}\norm{v}_{\theta-1}\wedge\abs{x-y}2^{-j(\theta-2)}\norm{v}_{\theta-1}}}.
\end{align*}
The second summand, we estimate analogously using $\theta-1\in (0,1)$ (and thus $\theta-2<0$) 
\begin{align*}
\MoveEqLeft
\abs{(S_{j-1} u(y)-S_{j-1} u(x))\Delta_{j}  v(y)}\\&\lesssim \paren[\big]{\abs{x-y}^{\theta-1}\norm{u}_{\theta-1}\wedge\abs{x-y}\norm{DS_{j-1}u}_{L^{\infty}}} 2^{-j(\theta-1)}\norm{v}_{\theta-1}\\&\lesssim\paren[\big]{\abs{x-y}^{\theta-1}\norm{u}_{\theta-1}\wedge\abs{x-y}2^{-j(\theta-2)}\norm{u}_{\theta-1}}2^{-j(\theta-1)}\norm{v}_{\theta-1}.
\end{align*} 
We can w.l.o.g.~assume that $\abs{x-y}\leqslant 1$. Otherwise the estimate \eqref{eq:claim} is trivial as $u,v\in L^{\infty}$. Then we let $j_{0}$ such that $2^{-j_{0}}\sim\abs{x-y}$ and decompose both sums in \eqref{eq:sums} in the part with $j>j_{0}$, such that $2^{-j}< 2^{-j_{0}}\leqslant\abs{x-y}$, and in the part with $j\leqslant j_{0}$ (i.p. a finite sum), such that $2^{-j}\geqslant\abs{x-y}$, and perform an analogous estimate for the two. We write down the estimate for the first sum. That is, we have
\begin{align*}
\MoveEqLeft
\sum_{j\geqslant -1}\abs{(S_{j-1} u(x)- u(y))(\Delta_{j} v(x)-\Delta_{j} v(y))}\\&\lesssim\norm{u}_{\theta-1}\norm{v}_{\theta-1}\paren[\bigg]{\sum_{j> j_{0}}\abs{x-y}^{\theta-1}2^{-j(\theta-1)}+\sum_{j\leqslant j_{0}}\abs{x-y}2^{-j(\theta-2)}2^{-j(\theta-1)}}\\&\lesssim\norm{u}_{\theta-1}\norm{v}_{\theta-1}\paren[\bigg]{\abs{x-y}^{\theta-1}2^{-j_{0}(\theta-1)}+\sum_{j\leqslant j_{0}}\abs{x-y}\abs{x-y}^{2\theta-3}}\\&\lesssim\norm{u}_{\theta-1}\norm{v}_{\theta-1}\paren[\bigg]{\abs{x-y}^{\theta-1}\abs{x-y}^{(\theta-1)}+\abs{x-y}\abs{x-y}^{2\theta-3}}\\&=\norm{u}_{\theta-1}\norm{v}_{\theta-1}\abs{x-y}^{2\theta-2}
\end{align*} using $\theta-1>0$, such that the first series converges, and $2\theta-3<0$ and that the second sum is a finite sum. Hence the claim \eqref{eq:claim} follows.
Rewritten in the previous notation, we thus obtain a bound uniformly in $s,x,y$, that is,
\begin{align*}
\MoveEqLeft
\bigl|(\nabla u\para J^{T}(\partial_{i} V))(s,x)-(\nabla u \para J^{T}(\partial_{i} V))(s,y)\\&\qquad-\nabla u(s,y)\cdot(J^{T}(\partial_{i} V)(s,x)-J^{T}(\partial_{i} V)(s,y))\bigr|\\&\lesssim\abs{x-y}^{2\theta-2}.
\end{align*}
We are left with the second term in \eqref{eq:x-t-separation}, that we furthermore decompose as follows
\begin{align*}
\MoveEqLeft
\abs{(\nabla u\para J^{T}(\partial_{i} V))_{sr}(x)-\nabla u(s,y)\cdot J^{T}(\partial_{i} V)_{sr}(x)}\\&\leqslant\abs{\nabla u_{s}\para J^{T}(\partial_{i} V)_{sr}(x)-\nabla u(s,x)\cdot J^{T}(\partial_{i} V)_{sr}(x)}+\abs{\nabla u_{sr}\para J^{T}(\partial_{i} V)_{r}(x)}\\&\quad + \abs{(\nabla u(s,x)-\nabla u(s,y))\cdot J^{T}(\partial_{i} V)_{sr}(x)}\\&\leqslant\abs{\nabla u_{s}\arap J^{T}(\partial_{i} V)_{sr} (x)}+\abs{\nabla u_{s}\reso J^{T}(\partial_{i} V)_{sr} (x)}+\abs{\nabla u_{sr}\para J^{T}(\partial_{i} V)_{r}(x)}\\&\quad+\abs{(\nabla u(s,x)-\nabla u(s,y))\cdot J^{T}(\partial_{i} V)_{sr}(x)}\\&\lesssim\abs{r-s}^{(2\theta-2)/\alpha}+\abs{(\nabla u(s,x)-\nabla u(s,y))\cdot J^{T}(\partial_{i} V)_{sr}(x)}\\&\lesssim\abs{r-s}^{(2\theta-2)/\alpha}+\abs{x-y}^{\theta-1}\abs{r-s}^{(\theta-1)/\alpha},
\end{align*}
where we used the estimates from the beginning, that is \eqref{eq:goodtimereg1} and \eqref{eq:goodtimereg2}.
Thus altogether we obtain the estimate for \eqref{eq:x-t-separation}:
\begin{align*}
\MoveEqLeft
\bigl|(\nabla u\para J^{T}(\partial_{i} V))(r,x)-(\nabla u \para J^{T}(\partial_{i} V))(s,y)\\&\qquad-\nabla u(s,y)\cdot(J^{T}(\partial_{i} V)(r,x)-J^{T}(\partial_{i} V)(s,y))\bigr|\\&\lesssim\abs{r-s}^{(2\theta-2)/\alpha}+\abs{x-y}^{2\theta-2}+\abs{x-y}^{\theta-1}\abs{r-s}^{(\theta-1)/\alpha},
\end{align*} where for $\abs{x-y}<\abs{r-s}^{1/\alpha}$ as well as for $\abs{x-y}\geqslant\abs{r-s}^{1/\alpha}$ (i.e. $\abs{r-s}\leqslant\abs{x-y}^{\alpha}$), we obtain the desired estimate \eqref{eq:para-est}.
\end{proof}
\end{section}

\section*{Acknowledgements}

H.K.~is supported by the Austrian Science Fund (FWF) Stand-Alone programme P 34992. Part of the work was done when H.K. was employed at Freie Universität Berlin and funded by the DFG under Germany's Excellence Strategy - The Berlin Mathematics Research Center MATH+ (EXC-2046/1, project ID: 390685689). N.P.~ gratefully acknowledges financial support by the DFG via Research Unit FOR2402.

\end{document}